\documentclass[12pt,reqno]{amsart}
\usepackage{amsfonts}

\usepackage{amssymb,amsmath,graphicx,amsfonts,euscript}
\usepackage{color}

\newtheorem{thm}{Theorem}[section]

\newtheorem{cor}{Corollary}[section]

\newtheorem{lemma}{Lemma}[section]

\theoremstyle{definition}

\theoremstyle{remark}
\newtheorem{rem}{Remark}[section]

\allowdisplaybreaks \voffset=-0.2in \numberwithin{equation}{section}

\begin{document}

\title[Fractional Boussinesq Equations]{Global regularity of the 2D fractional Boussinesq equations with subcritical dissipation}

\author[Stefanov, Wu, Xu and Ye]{Atanas Stefanov$^{1}$, Jiahong Wu$^{2}$,
Xiaojing Xu$^{3}$  and Zhuan Ye$^{4}$}

\address{$^{1}$ Department of Mathematics, University of Alabama at Birminghan,
Birmingham, AL 35294, USA}

\email{stefanov@uab.edu}

\address{$^2$ Department of Mathematics, University of Notre
	Dame, Notre Dame, IN 46556, USA}

\email{jwu29@nd.edu}

\address{$^3$ School of Mathematical Sciences, Beijing Normal University and
Laboratory of Mathematics and Complex Systems, Ministry of Education, Beijing
100875, P.R. China}

\email{xjxu@bnu.edu.cn}	

\address{$^4$ Department of Mathematics and Statistics, Jiangsu Normal
University, 101 Shanghai Road, Xuzhou 221116, Jiangsu, P.R. China}

\email{yezhuan815@126.com \ \ (the corresponding author)}



\medskip
\begin{abstract}
This paper studies the global regularity problem for the two-dimensional incompressible Boussinesq equations with fractional dissipation given by $(-\Delta)^{\frac\alpha2}u$ and $(-\Delta)^{\frac\beta2} \theta$. Attention is focused on the subcritical regime where $\alpha+ \beta>1$. The case $\alpha >\frac23$ was recently settled in a joint work of the authors [Math. Ann., \textbf{391} (2025), 5965-6012], which established global regularity under this condition. This paper addresses the remaining case  $\alpha \leq \frac23$. We obtain the sharpest regularity result by minimizing assumptions on $\alpha$ and $\beta$. We derive nonlinear lower bounds for the fractional Laplacian operator and implement an iterative procedure.
\end{abstract}
\maketitle

\section{Introduction}

This paper examines the following two-dimensional (2D) Boussinesq equations with fractional dissipation
\begin{equation}\label{Bouss}
\left\{\aligned
&\partial_{t}u+(u \cdot \nabla) u+ \Lambda^{\alpha}u+\nabla p=\theta e_{2},\,\,\,\,\,\,\,\,x\in \mathbb{R}^{2},\,\,t>0, \\
&\partial_{t}\theta+(u \cdot \nabla) \theta+ \Lambda^{\beta}\theta=0, \\
&\nabla\cdot u=0, \\
&u(x, 0)=u_{0}(x),  \quad \theta(x,0)=\theta_{0}(x),
\endaligned\right.
\end{equation}
where the unknowns are the velocity vector field $u(x,\,t)=(u_{1}(x,\,t),\,u_{2}(x,\,t))$, the scalar pressure $p=p(x,\,t)$, and the scalar quantity $\theta=\theta(x,\,t)$. Physically $\theta$ denotes the temperature  in the thermal convection or the density in the context of geostrophic
fluids. $e_{2}=(0,1)$ is the unit vector in the vertical direction.
The numbers $\alpha$ and $\beta$ are nonnegative real
parameters. The fractional Laplacian operator
$\Lambda^{\gamma}\triangleq(-\Delta)^{\frac{\gamma}{2}}$ is defined via the
Fourier transform  $\widehat{\Lambda^{\gamma}
f}(\xi)=|\xi|^{\gamma}\hat{f}(\xi)$. We adopt the convention that $\alpha=0$ means no dissipation term in the velocity equation, and $\beta=0$ means no diffusion in the temperature equation.

\vskip .1in
The Boussinesq equations emerge as a zeroth-order approximation to the coupling between the Navier-Stokes equations and thermodynamic equations. They are widely used to model geophysical flows, such as atmospheric fronts and oceanic circulation, and play a crucial role in studying Rayleigh-B\'enard convection (see \cite{MB,Pe1987}). Moreover, the 2D Boussinesq equations and their fractional counterparts hold significant mathematical interest. They serve as 2D models that capture the essential vortex-stretching mechanism present in the 3D Navier-Stokes and Euler equations. In fact, as noted in \cite{MB}, the inviscid Boussinesq equations share a deep formal analogy with the 3D axisymmetric Euler system with swirl.

\vskip .1in
Due to their fundamental role in modeling various astrophysical and geophysical phenomena, the classical Boussinesq equations and their fractional dissipation counterparts have been extensively studied. From a mathematical perspective, if the 2D Boussinesq equations include both full Laplacian dissipation terms, $\Delta u$ and $\Delta\theta$, standard energy methods guarantee the global existence of smooth solutions for arbitrarily large initial data (see, e.g., \cite{Can}). In contrast, determining whether local classical solutions of the inviscid Boussinesq equations can develop finite-time singularities remains an extremely challenging problem. Recently, some significant progress on the inviscid Boussinseq
equations finite-time blowup problem has been made (see, e.g.,
\cite{Chenhou21,Chenhou,Chenhou25,Clmzad,Elgindije,ElgindiPA}). This naturally leads to the fundamental question: how much dissipation is required to ensure global regularity? In recent years, significant progress has been made on the global regularity problem for the Boussinesq equations with partial or fractional dissipation. Chae
\cite{C1} and  Hou-Li \cite{HL} successfully established the global regularity
to (\ref{Bouss}) with $\alpha=2,\ \beta=0$ or $\alpha=0, \ \beta=2$. For the case $\alpha=0, \
\beta=2$ with rough initial data, we refer to \cite{Danchinp09,HK1}. For
$\alpha=0, \beta\in (1,\,2]$, \cite{Hzer10} established the global well-posedness of (\ref{Bouss}) with rough initial data. By introducing the combined quantities, Hmidi, Keraani and Rousset \cite{HK3,HK4} were able to establish the global regularity for the cases when either $\alpha=1,\,\beta=0$ or $\alpha=0,\,\beta=1$. An alternative proof for the case $\alpha=0,\,\beta=1$, which avoids the use of these combined quantities, was given in \cite{YeNA2020}.

\vskip .1in
Jiu, Miao, Wu, and Zhang \cite{JMWZ} discovered that the global regularity problem for \eqref{Bouss} depends crucially on the value of $\alpha+ \beta$. In the special case
$\alpha+ \beta =1$, the Boussinesq regularity problem boils down to the corresponding problem on the generalized surface quasi-geostrophic equation with critical dissipation. This insight led to the classification of $\alpha+ \beta$ into three regimes: the subcritical regime ($\alpha+ \beta>1$), the critical regime ($\alpha+ \beta=1$) and the supercritical regime ($\alpha+ \beta<1$). Naturally, the smaller $\alpha+ \beta$, the more challenging the regularity problem becomes. In particular, the global regularity problem for the supercritical regime remains largely out of reach, while substantial progress has been made in the critical and subcritical cases.

\vskip .1in
Compared to the two special critical cases studied in \cite{HK3,HK4}, analyzing the general critical case $\alpha+\beta=1$ with $0<\alpha,\,\beta<1$ is significantly more challenging. In \cite{JMWZ}, the authors established global regularity for any $\alpha$ and $\beta$ satisfying
$$
\alpha+\beta=1, \quad \alpha>\frac{23-\sqrt{145}}{12}\approx 0.9132.
$$
Subsequent efforts are devoted to enlarge the range of $\alpha$. Stefanov and Wu \cite{SW} improved this result by proving global regularity for
$$
\alpha+\beta=1, \quad \alpha> \frac{\sqrt{1777}-23}{12}\approx 0.7981.
$$
Further progress was made by Wu, Xu, Xue, and Ye \cite{wuxuxueye}, who expanded the range to
$$
\alpha+\beta=1,\quad \alpha>\frac{10}{13}\approx0.7692.
$$
Most recently, in a joint work of the authors \cite{SWXY}, the global regularity result in the critical regime
$$
\alpha+\beta=1,\quad \alpha>\frac23
$$
was established under the mild assumption that the $L^\infty$-norm of the initial temperature is small.

\vskip .1in
This paper focuses on the global regularity of \eqref{Bouss} in the subcritical regime
$\alpha+ \beta >1$. Even the subcritical global regularity problem  is  not trivial and not all subcritical cases have been resolved. For example, we do not know the global regularity for the case when $\alpha$ and $\beta$ are close to one half and $\alpha+\beta>1$. The main difficulty arises because direct energy estimates are insufficient to establish the necessary global {\it a priori} bounds when $\alpha,\beta<1$.

\vskip .1in
To provide a more precise account of existing results, we further divide the subcritical regime into two cases: the velocity dissipation dominated regime ($\alpha\geq\beta$) and the thermal diffusion dominated regime ($\alpha<\beta$).
In the velocity dissipation dominated case, Miao and Xue \cite{MX} obtained the global regularity for (\ref{Bouss}) with
$$ \alpha+\beta>1,\qquad 0.8876\approx\frac{6-\sqrt{6}}{4}<\alpha<1.$$
This result was later refined by \cite{Yemmas}, extending the admissible range to
$$\alpha+\beta>1,\qquad 0.7351\approx\frac{10-2\sqrt{10}}{5}<\alpha<1.$$
Zhou, Li, Shang, Wu, Yuan, and Zhao \cite{zlswyz}, utilizing \cite{Yemmas} and the nonlinear maximum principle for fractional Laplacians developed by Constantin and Vicol \cite{CV}, established global regularity for
\begin{eqnarray}
\beta>\frac{1-\alpha}{\alpha},\quad \frac{2}{3}<\alpha<1.\nonumber
\end{eqnarray}
Very recently, \cite{SWXY} was able to completely solve the global regularity problem for the subcritical range
$$
\alpha+ \beta >1, \quad\frac{2}{3}<\alpha<1.
$$
In the thermal diffusion dominated case, Constantin and Vicol \cite{CV} proved the global regularity of \eqref{Bouss} for
$$\beta>\frac{2}{2+\alpha}, \quad 0<\alpha<1.$$
Ye and Xu \cite{YX201502} extended this result by establishing global regularity under the condition
\begin{equation*}
	\beta>\left\{\aligned
	&\max\Big\{\frac{2}{3},\,\,\frac{4-\alpha^{2}}{4+3\alpha}\Big\},\, \,\,\quad \quad 0<\alpha\leq \frac{2}{3},\\
	&\frac{2-\alpha}{2}, \qquad\qquad\qquad\qquad
	\frac{2}{3}\leq \alpha<1.\\
	\endaligned\right.
\end{equation*}
Recently, \cite{SWXY} improved this result further, establishing the global regularity condition
\begin{eqnarray}\label{sdfghp987}
\beta>\max\Big\{\alpha,\,\,\frac{4-\alpha^{2}}{4+3\alpha}\Big\},\, \,\, \quad 0<\alpha\leq \frac{2}{3}.
\end{eqnarray}

\vskip .1in
As described above, the subcritical range $\alpha+ \beta>1$ with $\alpha>\frac23$ has been completely resolved. Our focus here is on the remaining case, where $\alpha+ \beta>1$ with $\alpha \leq\frac23$, and we aim to sharpen the current best result in this regime. As aforementioned, the existing requirement for
 $0<\alpha\leq\frac{2}{3}$ is \eqref{sdfghp987}. To refine this result, we further divide the range $\alpha \leq \frac23$ into two subregimes:
\begin{equation}\label{aa}
0<\alpha\leq  \sqrt{13}-3\approx 0.6056 \quad\mbox{and}\quad \sqrt{13}-3<\alpha\leq\frac{2}{3}.
\end{equation}
The threshold value  $\alpha_0 = \sqrt{13}-3$ is the positive root of the quadratic equation
$$
\alpha_0 = \frac{4-2\alpha_0}{4+\alpha_0}.
$$
This partition in \eqref{aa} naturally corresponds to two regimes: the thermal diffusion dominated regime and the velocity dissipation dominated regime. Each regime requires a different analytical approach to extract the sharpest possible parameter range.

\vskip .1in
The precise statement of our result is presented in the following theorem.

\begin{thm}\label{Th3}
Let $(u_{0}, \theta_{0})
\in H^{s}(\mathbb{R}^{2})\times H^{s}(\mathbb{R}^{2})\cap L^{\frac{4}{2+\alpha}}(\mathbb{R}^{2})$ with $s>2$. If $\alpha\in(0,1)$ and $\beta\in(0,1)$ satisfy
\begin{equation*}
\left\{\aligned
&\beta\geq\frac{4-2\alpha}{4+\alpha}, \qquad \ \ \qquad\qquad\qquad\qquad0<\alpha\leq  \sqrt{13}-3,\\
&\beta> \max\Big\{\frac{4-6\alpha}{\alpha},\ \frac{2-\alpha-2\alpha^{2}}{2\alpha}\Big\},\quad \sqrt{13}-3<\alpha\leq\frac{2}{3},
\endaligned\right.
\end{equation*}
then the system (\ref{Bouss}) has a unique global solution such that, for any
$T>0$,
$$u\in C([0, T]; H^{s}(\mathbb{R}^{2}))\cap L^{2}([0, T]; H^{s+\frac{\alpha}{2}}(\mathbb{R}^{2})),$$
$$\theta\in C([0, T]; H^{s}(\mathbb{R}^{2})\cap L^{\frac{4}{2+\alpha}}(\mathbb{R}^{2}))\cap L^{2}([0, T];
H^{s+\frac{\beta}{2}}(\mathbb{R}^{2})).$$
\end{thm}

\vskip .1in
\begin{rem}  Theorem \ref{Th3} sharpens Theorem 1.3 of \cite{SWXY}, which requires
	$$\beta>\max\Big\{\alpha,\,\,\frac{4-\alpha^{2}}{4+3\alpha}\Big\},\quad 0<\alpha\leq\frac{2}{3}.$$
It is easy to see that the requirement for $\beta$ in Theorem \ref{Th3} is much weaker.
When $0<\alpha\leq \sqrt{13}-3$, we have
$$\frac{4-2\alpha}{4+\alpha}<\frac{4-\alpha^{2}}{4+3\alpha}\leq  \max\Big\{\alpha,\,\,\frac{4-\alpha^{2}}{4+3\alpha}\Big\}.$$
When $\sqrt{13}-3<\alpha\leq\frac{2}{3}$, it holds
$$\max\Big\{\frac{4-6\alpha}{\alpha},\ \frac{2-\alpha-2\alpha^{2}}{2\alpha}\Big\}<\frac{4-2\alpha}{4+\alpha}<\frac{4-\alpha^{2}}{4+3\alpha}\leq  \max\Big\{\alpha,\,\,\frac{4-\alpha^{2}}{4+3\alpha}\Big\}.$$
\end{rem}

\begin{rem}
We point out that the additional condition $\theta_{0}\in L^{\frac{4}{2+\alpha}}(\mathbb{R}^{2})$ is only required for the case $\sqrt{13}-3<\alpha\leq\frac{2}{3}$ (see \eqref{sdfghp28} below for details).
\end{rem}

\vskip .1in
As a direct consequence of Theorem \ref{Th3}, by setting $\alpha=\beta = \frac{4-2\alpha}{4+\alpha}$, we have the following corollary.
\begin{cor}
Let $(u_{0}, \theta_{0})
\in H^{s}(\mathbb{R}^{2})\times H^{s}(\mathbb{R}^{2})$ with $s>2$.
Consider (\ref{Bouss}) with $\alpha=\beta\geq\sqrt{13}-3$, then there exists a unique global solution to the corresponding system such that, for any
$T>0$
$$u,\, \theta\in C([0, T]; H^{s}(\mathbb{R}^{2}))\cap L^{2}([0, T]; H^{s+\frac{\alpha}{2}}(\mathbb{R}^{2})).$$
\end{cor}

The proof of Theorem \ref{Th3} is naturally divided into two cases:
$$
	0<\alpha\leq  \sqrt{13}-3 \quad\mbox{and}\quad \sqrt{13}-3<\alpha\leq\frac{2}{3}.
$$
For the first case, $\beta\ge \alpha$ and the parameters fall within the thermal diffusion-dominated regime. Since velocity dissipation is relatively weak in this setting, we work directly with the vorticity formulation, as the equation for the combined quantity $G$ offers no significant advantage (see the following section for details). For the second case, we utilize the equation for $G$ (see \eqref{t305}) to establish the necessary bounds. The proof consecutively builds global estimates in increasingly regular function spaces. This approach is highly intricate, relying on delicate and  extremely fine estimates to achieve the desired results.

\vskip .3in
\section{The proof of Theorem \ref{Th3}}\setcounter{equation}{0}

This section proves Theorem \ref{Th3}. Since the local well-posedness of \eqref{Bouss} for initial data $(u_{0}, \theta_{0})
\in H^{s}(\mathbb{R}^{2})\times H^{s}(\mathbb{R}^{2})$ with $s>2$ is well-known (see for example \cite{CNpr97,MB}), our main efforts focus on establishing the global {\it a priori} bounds for $(u,\theta)$ on $[0,\,T]$ for any given
$T>0$. Throughout this paper, we denote by $C$ a universal positive
constant whose value may change from line to line. The symbol $C(a,b,...)$ means that $C$ depends on variables $a$, $b$ and so on.

\vskip .1in
The proof of Theorem \ref{Th3} is divided into two parts, namely,
\begin{align}\label{dfg456}
\mbox{Part 1:} \ \ \ \beta\geq\frac{4-2\alpha}{4+\alpha},\quad 0<\alpha\leq \sqrt{13}-3;\qquad\qquad\qquad\qquad \qquad
\end{align}
\begin{align}\label{dfgp66}
\qquad \mbox{Part 2:} \ \ \ \beta> \max\Big\{\frac{4-6\alpha}{\alpha},\ \frac{2-\alpha-2\alpha^{2}}{2\alpha}\Big\},\quad \sqrt{13}-3<\alpha\leq\frac{2}{3}.
\end{align}
Correspondingly the rest of this section is divided into two subsections.
\vskip .2in

\subsection{The proof of Part 1}
Before proving this part, it is worthwhile to mention the proof of \cite{YX201502}, which depends heavily on the new unknown $G$ satisfying the equation
\begin{eqnarray}\label{uk3rtr01}
\partial_{t}G+(u\cdot\nabla)
G+\Lambda^{\alpha}G=\Lambda^{\alpha-\beta}\partial_{x_{1}}\theta-[\mathcal {R}_{\beta},\,u\cdot\nabla]\theta
\end{eqnarray}
with $G=\omega+\mathcal {R}_{\beta}\theta$ and $\mathcal {R}_{\beta}=\partial_{x_{1}}\Lambda^{-\beta}$. By the Biot-Savart law, one has
\begin{eqnarray*}
u=\nabla^{\perp}\Delta^{-1}\omega
=\nabla^{\perp}\Delta^{-1}(G-\partial_{x_{1}}\Lambda^{-\beta}\theta)
=\nabla^{\perp}\Delta^{-1}G-\nabla^{\perp}\Delta^{-1}\partial_{x_{1}}\Lambda^{-\beta}
\theta\triangleq u_{G}+u_{\theta}.
\end{eqnarray*}
Therefore, the commutator at the right-hand side of \eqref{uk3rtr01} can be rewritten as
$$[\mathcal {R}_{\beta},\,u\cdot\nabla]\theta=[\mathcal {R}_{\beta},\,u_{G}\cdot\nabla]\theta+[\mathcal {R}_{\beta},\,u_{\theta}\cdot\nabla]\theta.$$
To derive the global $L^{2}$-bound of $G$, the obstacle we have to overcome first is the presence of the term
\begin{align}\label{klhlp98}
\int_{\mathbb{R}^{2}}
[\mathcal {R}_{\beta},\,u_{G}\cdot\nabla]\theta \,\,G\,dx.
\end{align}
In order to control the term \eqref{klhlp98}, it strongly requires the following restriction (see (2.23) of \cite{YX201502} for details)
\begin{eqnarray}\label{swety9}
\beta>\frac{4-\alpha^{2}}{4+3\alpha}.
\end{eqnarray}
As a matter of fact, it seems impossible to apply the argument adopted in \cite{YX201502} to weaken the above requirement \eqref{swety9}. In this subsection, we will not follow the approach of investigating \eqref{uk3rtr01}, which is frequently used to deal with the thermal diffusion dominated case (see \cite{YX201502,YJW,HK4,SWXY}). Indeed, the main argument used in this paper is completely different from \cite{YX201502}. On the contrary, we directly work on the vorticity equation  itself
\begin{eqnarray}\label{sdhyrp56}
\partial_{t}\omega+(u \cdot \nabla)\omega+ \Lambda^{\alpha}\omega=\partial_{x_{1}}\theta.
\end{eqnarray}
Consequently, we don't need to deal with the term \eqref{klhlp98}.
However, the price we need to pay is that we still face the loss of one derivative in $\theta$ in \eqref{sdhyrp56}. Fortunately, combining $L^{2}$-norm of $\omega$ and the fractional order regularity estimate of $\theta$, we are able to establish the global $L^{2}$-bound of $\omega$ and  $H^{\frac{(2+\alpha)\beta}{4}}$-bound of $\theta$ as long as $\beta\geq\frac{4-2\alpha}{4+\alpha}$, which is weaker than \eqref{swety9}. However, the global $L^{2}$-bound of $\omega$ is insufficient for our purpose. Interestingly, we carry out an iterative process, which would allow us to derive the global $L^{p}$-bound of $\omega$ for the $p$ as large as possible. This along with the nonlinear lower bounds for the fractional Laplacian established in \cite{CV} allows us to derive the global regularity of $u$ and $\theta$. In this sense, it has more advantages in working on the vorticity equation \eqref{sdhyrp56} itself. To see it clearly, on the one hand, let us consider a special interesting case, namely, (\ref{Bouss}) with $\alpha=\beta$. It should be pointed out that the combined quantity $G$ is not workable for this case. In fact, one may check that the corresponding combined quantity $G$ obeys
$$\partial_{t}G+(u \cdot \nabla)G+ \Lambda^{\alpha}G=\partial_{x_{1}}\theta-[\mathcal {R}_{\alpha},\,u\cdot\nabla]\theta,\quad G=\omega+\mathcal {R}_{\alpha}\theta$$
or
$$\partial_{t}G+(u \cdot \nabla)G+ \Lambda^{\alpha}G=\partial_{x_{1}}\theta+[\mathcal {R}_{\alpha},\,u\cdot\nabla]\theta,\quad G=\omega-\mathcal {R}_{\alpha}\theta.$$
Obviously, one still faces the loss of one derivative in $\theta$ in the equation $G$. To make matters worse, a commutator $[\mathcal {R}_{\alpha},\,u\cdot\nabla]\theta$ is present in the equation $G$. Notice that the dissipative term $\Lambda^{\alpha}G$ at the left-hand side of \eqref{uk3rtr01} is relatively weak when $\alpha$ is small, which also implies the insignificance of investigating the combined quantity $G$. On the other hand, the situation when $\alpha$ is small is different from the case when $\alpha$ is big. In this situation, it is natural to consider the following combined quantity $G\triangleq\omega-\mathcal {R}_{\alpha}\theta$ (see \cite{JMWZ,MX,SW,SWXY,wuxuxueye,Yemmas}),
which obeys
\begin{eqnarray}\label{sdhyrsp954} 
\partial_{t}G+(u\cdot\nabla)G+\Lambda^{\alpha}G=\Lambda^{\beta-\alpha}
\partial_{x_{1}}\theta+[\mathcal {R}_{\alpha},\,u\cdot\nabla]\theta.
\end{eqnarray}
Comparing with \eqref{uk3rtr01} and \eqref{sdhyrsp954}, the terms at the right-hand enjoy the similar structure. However, whether $\alpha$ is big or small, the dissipation term is the same, namely $\Lambda^{\alpha}G$, which further implies the insignificance of investigating the combined quantity $G$ when $\alpha$ is small.

\vskip .1in
Let us begin with the natural energy estimates.
\begin{lemma}
Assume that $u_{0}\in L^{2}$ and $\theta_{0}\in L^{p}\cap L^{2}$ with $p\in [1,\infty]$. Then
\begin{eqnarray}
\|\theta(t)\|_{L^{2}}^{2}+ \int_{0}^{t}{
\|\Lambda^{\frac{\beta}{2}}\theta(\tau)\|_{L^{2}}^{2}\,d\tau}\leq
\|\theta_{0}\|_{L^{2}}^{2},\quad\|\theta(t)\|_{L^{p}}\leq \|\theta_{0}\|_{L^{p}},\nonumber
\end{eqnarray}
\begin{eqnarray}
\|u(t)\|_{L^{2}}^{2}+ \int_{0}^{t}{
\|\Lambda^{\frac{\alpha}{2}}u(\tau)\|_{L^{2}}^{2}\,d\tau}\leq
(\|u_{0}\|_{L^{2}}+t\|\theta_{0}\|_{L^{2}})^{2}.\nonumber
\end{eqnarray}
\end{lemma}

It is noteworthy to mention that all the estimates obtained of
$\mbox{Part 1}$ are valid for all $\beta\geq\frac{4-2\alpha}{4+\alpha}$ with $\alpha\leq\frac{2}{3}$.
Now let us start the proof of the first part of Theorem \ref{Th3}.
The following lemma concerns the global $L^{2}$-bound of $\omega$ and  $H^{\frac{(2+\alpha)\beta}{4}}$-bound of $\theta$.
\begin{lemma}\label{addftlem1}
If $\alpha,\beta\in(0,1)$ satisfy $\beta\geq\frac{4-2\alpha}{4+\alpha}$,
then the following estimate holds
\begin{eqnarray}\label{vbhkp01}
\|\omega(t)\|_{L^{2}}^{2}+\|\Lambda^{\frac{(2+\alpha)\beta}{4}}\theta(t)\|_{L^{2}}^{2}
+\int_{0}^{t}{( \|\Lambda^{\frac{\alpha}{2}}\omega \|_{L^{2}}^{2}+
\|\Lambda^{\frac{(4+\alpha)\beta}{4}}\theta\|_{L^{2}}^{2})(\tau)\,d\tau}\leq
C(t,\,u_{0},\,\theta_{0}).
\end{eqnarray}
\end{lemma}

\begin{proof}
Taking scalar product of \eqref{sdhyrp56} with $\omega$ and integrating over the spatial variable, one derives that
\begin{align}\label{klhbqy1}
\frac{1}{2}\frac{d}{dt}\|\omega(t)\|_{L^{2}}^{2}
+ \|\Lambda^{\frac{\alpha}{2}}\omega\|_{L^{2}}^{2}
= \int_{\mathbb{R}^{2}}\partial_{x_{1}}\theta\, \omega\,dx.
\end{align}
Applying $\Lambda^{\frac{(2+\alpha)\beta}{4}}$ to $(\ref{Bouss})_{2}$, multiplying the resultant equation by $\Lambda^{\frac{(2+\alpha)\beta}{4}}\theta$ and integrating over $\mathbb R^2$, we deduce
\begin{align}\label{ksdh95}
\frac{1}{2}\frac{d}{dt} \|\Lambda^{\frac{(2+\alpha)\beta}{4}}\theta(t)\|_{L^{2}}^{2}
+ \|\Lambda^{\frac{(4+\alpha)\beta}{4}}\theta\|_{L^{2}}^{2}
= - \int_{\mathbb{R}^{2}}
[\Lambda^{\frac{(2+\alpha)\beta}{4}}, u \cdot \nabla]\theta\,\,\Lambda^{\frac{(2+\alpha)\beta}{4}}\theta\,dx.
\end{align}
We multiply both sides of the upper equality by $\eta$ to find that
\begin{align}\label{klhbqy2}
\frac{1}{2}\frac{d}{dt}\left(\eta \|\Lambda^{\frac{(2+\alpha)\beta}{4}}\theta(t)\|_{L^{2}}^{2}\right)
+ \eta \|\Lambda^{\frac{(4+\alpha)\beta}{4}}\theta\|_{L^{2}}^{2}
= - \eta \int_{\mathbb{R}^{2}}
[\Lambda^{\frac{(2+\alpha)\beta}{4}}, u \cdot \nabla]\theta\,\,\Lambda^{\frac{(2+\alpha)\beta}{4}}\theta\,dx,
\end{align}
where the suitable large $\eta>0$ will be fixed later.
Summarizing \eqref{klhbqy1} and \eqref{klhbqy2}, we directly obtain
\begin{align}\label{klhbqy3}
&\frac{1}{2}\frac{d}{dt}\left(\|\omega(t)\|_{L^{2}}^{2}
+ \eta \|\Lambda^{\frac{(2+\alpha)\beta}{4}}\theta(t)\|_{L^{2}}^{2}\right)
+ \|\Lambda^{ \frac{\alpha}{2}}\omega\|_{L^{2}}^{2}+
\eta \|\Lambda^{\frac{(4+\alpha)\beta}{4}}\theta\|_{L^{2}}^{2}
\nonumber\\&=\int_{\mathbb{R}^{2}}\partial_{x_{1}}\theta\, \omega\,dx -\eta \int_{\mathbb{R}^{2}}
[\Lambda^{\frac{(2+\alpha)\beta}{4}}, u \cdot \nabla]\theta\,\,\Lambda^{\frac{(2+\alpha)\beta}{4}}\theta\,dx.
\end{align}
Owing to the Young inequality, we see that
\begin{align}\label{ddklp1}
\left|\int_{\mathbb{R}^{2}}\partial_{x_{1}}\theta\, \omega\,dx\right|&\leq C \|\Lambda^{\frac{\alpha}{2}}\omega\|_{L^{2}}\|\Lambda^{1-\frac{\alpha}{2}}\theta\|_{L^{2}}
\nonumber\\&
\leq C\|\Lambda^{\frac{\alpha}{2}}\omega\|_{L^{2}}
\|\theta\|_{L^{2}}^{1-\frac{4-2\alpha}{(4+\alpha)\beta}}
\|\Lambda^{\frac{(4+\alpha)\beta}{4}}\theta\|_{L^{2}}
^{\frac{4-2\alpha}{(4+\alpha)\beta}}\nonumber\\&\leq \frac{1}{4}\|\Lambda^{ \frac{\alpha}{2}}\omega\|_{L^{2}}^{2}+\frac{\eta}{4}
\|\Lambda^{\frac{(4+\alpha)\beta}{4}}\theta\|_{L^{2}}^{2}
+C(\eta)\|\theta\|_{L^{2}}^{2},
\end{align}
where in the last line we have used the restriction $\beta>\frac{4-2\alpha}{4+\alpha}$. For the case $\beta=\frac{4-2\alpha}{4+\alpha}$, we should modify the estimate as follows
\begin{align}\label{ddklp2}
\left|\int_{\mathbb{R}^{2}}\partial_{x_{1}}\theta\, \omega\,dx\right|&\leq \widetilde{C} \|\Lambda^{\frac{\alpha}{2}}\omega\|_{L^{2}}\|\Lambda^{1-\frac{\alpha}{2}}\theta\|_{L^{2}}
\nonumber\\&
\equiv \widetilde{C}\|\Lambda^{\frac{\alpha}{2}}\omega\|_{L^{2}}
\|\Lambda^{\frac{(4+\alpha)\beta}{4}}\theta\|_{L^{2}}
\nonumber\\&\leq \frac{1}{4}\|\Lambda^{ \frac{\alpha}{2}}\omega\|_{L^{2}}^{2}+\widetilde{C}^{2}
\|\Lambda^{\frac{(4+\alpha)\beta}{4}}\theta\|_{L^{2}}^{2}\nonumber\\&\leq \frac{1}{4}\|\Lambda^{ \frac{\alpha}{2}}\omega\|_{L^{2}}^{2}+\frac{\eta}{4}
\|\Lambda^{\frac{(4+\alpha)\beta}{4}}\theta\|_{L^{2}}^{2},
\end{align}
where we have fixed $\eta$ satisfying $\eta\geq 4\widetilde{C}^{2}$. Concerning \eqref{ddklp1} and \eqref{ddklp2}, for all $\beta\geq\frac{4-2\alpha}{4+\alpha}$, we conclude that
\begin{align}\label{klhbqy4}
\left|\int_{\mathbb{R}^{2}}\partial_{x_{1}}\theta\, \omega\,dx\right|&\leq
 \frac{1}{4}\|\Lambda^{ \frac{\alpha}{2}}\omega\|_{L^{2}}^{2}+\frac{\eta}{4}
\|\Lambda^{\frac{(4+\alpha)\beta}{4}}\theta\|_{L^{2}}^{2}
+C(\eta)\|\theta\|_{L^{2}}^{2}.
\end{align}
We point out that this is the only place in the proof where we use the main assumption of the theorem, namely $\beta\geq\frac{4-2\alpha}{4+\alpha}$.
To deal with the last term at the right handside of \eqref{klhbqy3}, we appeal to the following commutator estimate (see \cite[Lemma 2.3]{SWXY})
\begin{eqnarray}\label{tvcbmp6}
\|[\Lambda^{\sigma},f\cdot\nabla]g\|_{L^{p}}\leq C\|\nabla f\|_{L^{r_{1}}}\|\Lambda^{\sigma}g\|_{L^{r_{2}}}
\end{eqnarray}
where $\sigma\in (0,1)$ and $p, r_{1}, r_{2}\in(1, \infty)$ such that $\frac{1}{p}=\frac{1}{r_{1}}+\frac{1}{r_{2}}$.
By means of \eqref{tvcbmp6}, it yields
\begin{align}\label{dfgjkpw8}
\left|\int_{\mathbb{R}^{2}}
[\Lambda^{\frac{(2+\alpha)\beta}{4}}, u \cdot \nabla]\theta\,\,\Lambda^{\frac{(2+\alpha)\beta}{4}}\theta\,dx\right|&\leq C\|[\Lambda^{\frac{(2+\alpha)\beta}{4}}, u \cdot \nabla]\theta\|_{L^{\frac{2(4+\alpha)}{6+\alpha}}}
\|\Lambda^{\frac{(2+\alpha)\beta}{4}}
\theta\|_{L^{\frac{2(4+\alpha)}{2+\alpha}}}
\nonumber\\
&\leq C\|\nabla u\|_{L^{\frac{4+\alpha}{2}}}
\|\Lambda^{\frac{(2+\alpha)\beta}{4}}\theta\|_{L^{\frac{2(4+\alpha)}
{2+\alpha}}}^{2}.
\end{align}
It follows from the direct interpolation inequalities that
\begin{align}\label{bmnhgf1}
\|\nabla u\|_{L^{\frac{4+\alpha}{2}}}&\leq C
\|\Lambda^{\frac{4+2\alpha}{4+\alpha}}u\|_{L^2}
\nonumber\\
&\leq C
\|\Lambda^{\frac{2+\alpha}{6}}u\|_{L^2}^{\frac{3\alpha}{2(4+\alpha)}}
\|\Lambda^{1+\frac{\alpha}{2}}u\|_{L^2}^{\frac{8-\alpha}{2(4+\alpha)}} \nonumber\\
&\leq C \left(\|u\|_{L^2}^{\frac{1}{3}}\|\Lambda^{\frac{2+\alpha}{4}}u\|_{L^2}
^{\frac{2}{3}}\right)^{\frac{3\alpha}{2(4+\alpha)}}
\|\Lambda^{\frac{\alpha}{2}}\omega\|_{L^2}^{\frac{8-\alpha}{2(4+\alpha)}}
 \nonumber\\
&\leq C  \|u\|_{L^2}^{\frac{\alpha}{2(4+\alpha)}}\|\Lambda^{\frac{2+\alpha}{4}}u\|_{L^2}
^{\frac{\alpha}{4+\alpha}}
\|\Lambda^{\frac{\alpha}{2}}\omega\|_{L^2}^{\frac{8-\alpha}{2(4+\alpha)}}
\nonumber\\
&\leq C  \|u\|_{L^2}^{\frac{\alpha}{2(4+\alpha)}} \left(\|\Lambda^{\frac{\alpha}{2}}u\|_{L^2}^{\frac{1}{2}}\|\Lambda u\|_{L^2}^{\frac{1}{2}}\right)
^{\frac{\alpha}{4+\alpha}}
\|\Lambda^{\frac{\alpha}{2}}\omega\|_{L^2}^{\frac{8-\alpha}{2(4+\alpha)}}
\nonumber\\
&\leq C  \|u\|_{L^2}^{\frac{\alpha}{2(4+\alpha)}} \|\Lambda^{\frac{\alpha}{2}}u\|_{L^2}^{\frac{\alpha}{2(4+\alpha)}}\|\omega\|_{L^2}^{\frac{\alpha}{2(4+\alpha)}}
\|\Lambda^{\frac{\alpha}{2}}\omega\|_{L^2}^{\frac{8-\alpha}{2(4+\alpha)}}.
\end{align}
By virtue of the sharp interpolation inequality (see for example \cite[Theorem 2.42]{BCD})
\begin{align}\label{dfexbg6}
\|f\|_{L^{p}}\leq C\|f\|_{\dot{B}_{\infty,\infty}^{-\alpha}}^{\frac{p-2}{p}}
\|f\|_{\dot{B}_{2,2}^{\gamma}}^{\frac{2}{p}},\quad \gamma=\frac{\alpha(p-2)}{2}>0,
\end{align}
we have
\begin{align}\label{bmnhgf2}
\|\Lambda^{\frac{(2+\alpha)\beta}{4}}\theta\|_{L^{\frac{2(4+\alpha)}
{2+\alpha}}}&\leq C \|\Lambda^{\frac{(2+\alpha)\beta}{4}}
\theta\|_{\dot{B}_{\infty,\infty}^{-\frac{(2+\alpha)\beta}{4}}}^{\frac{2}{4+\alpha}}
\|\Lambda^{\frac{(2+\alpha)\beta}{4}}\theta\|_{\dot{B}_{2,2}^{\frac{\beta}{2}}}
^{\frac{2+\alpha}{4+\alpha}}\nonumber\\
&\leq C \|\theta\|_{\dot{B}_{\infty,\infty}^{0}}^{\frac{2}{4+\alpha}}
\|\Lambda^{\frac{(4+\alpha)\beta}{4}}\theta\|_{L^{2}}
^{\frac{2+\alpha}{4+\alpha}}
\nonumber\\
&\leq C \|\theta\|_{L^{\infty}}^{\frac{2}{4+\alpha}}
\|\Lambda^{\frac{(4+\alpha)\beta}{4}}\theta\|_{L^{2}}
^{\frac{2+\alpha}{4+\alpha}}.
\end{align}
Putting \eqref{bmnhgf1} and \eqref{bmnhgf2} into \eqref{dfgjkpw8} yields
\begin{align}\label{klhbqy5}
&\eta\left|\int_{\mathbb{R}^{2}}
[\Lambda^{\frac{(2+\alpha)\beta}{4}}, u \cdot \nabla]\theta\,\,\Lambda^{\frac{(2+\alpha)\beta}{4}}\theta\,dx\right|\nonumber\\&
 \leq C\eta\|u\|_{L^2}^{\frac{\alpha}{2(4+\alpha)}} \|\Lambda^{\frac{\alpha}{2}}u\|_{L^2}^{\frac{\alpha}{2(4+\alpha)}}\|\omega\|_{L^2}^{\frac{\alpha}{2(4+\alpha)}}
\|\Lambda^{\frac{\alpha}{2}}\omega\|_{L^2}^{\frac{8-\alpha}{2(4+\alpha)}}
\|\theta\|_{L^{\infty}}^{\frac{4}{4+\alpha}}
\|\Lambda^{\frac{(4+\alpha)\beta}{4}}\theta\|_{L^{2}}
^{\frac{4+2\alpha}{4+\alpha}}
\nonumber\\
&\leq
\frac{1}{4}\|\Lambda^{ \frac{\alpha}{2}}\omega\|_{L^{2}}^{2}+\frac{\eta}{4}
\|\Lambda^{\frac{(4+\alpha)\beta}{4}}\theta\|_{L^{2}}^{2}
+C\eta^{\frac{8}{\alpha}}\|u\|_{L^2}^{2}
\|\theta\|_{L^{\infty}}^{\frac{16}{\alpha}}
\|\Lambda^{\frac{\alpha}{2}}u\|_{L^2}^{2}\|\omega\|_{L^2}^{2}.
\end{align}
Inserting \eqref{klhbqy4} and \eqref{klhbqy5} into \eqref{klhbqy3}, we have
\begin{align}\label{zdfhpl6}
 &\frac{d}{dt}\left(\|\omega(t)\|_{L^{2}}^{2}
+ \eta \|\Lambda^{\frac{(2+\alpha)\beta}{4}}\theta(t)\|_{L^{2}}^{2}\right)
+ \|\Lambda^{ \frac{\alpha}{2}}\omega\|_{L^{2}}^{2}+
\eta \|\Lambda^{\frac{(4+\alpha)\beta}{4}}\theta\|_{L^{2}}^{2}
 \nonumber\\&\leq C(\eta)\|\theta\|_{L^{2}}^{2}+ C(\eta)\|u\|_{L^2}^{2}
\|\theta\|_{L^{\infty}}^{\frac{16}{\alpha}}
\|\Lambda^{\frac{\alpha}{2}}u\|_{L^2}^{2}\|\omega\|_{L^2}^{2}.
\end{align}
Applying the Gronwall inequality to \eqref{zdfhpl6}, it follows that
\begin{align}
\|\omega(t)\|_{L^{2}}^{2}
+ \|\Lambda^{\frac{(2+\alpha)\beta}{4}}\theta(t)\|_{L^{2}}^{2}
+ \int_{0}^{t}(\|\Lambda^{ \frac{\alpha}{2}}\omega\|_{L^{2}}^{2}+
\|\Lambda^{\frac{(4+\alpha)\beta}{4}}\theta\|_{L^{2}}^{2})(\tau)\,d\tau\leq C(t,\,u_{0},\,\theta_{0}).\nonumber
\end{align}
Hence, the proof of Lemma \ref{addftlem1} is completed.
\end{proof}

\vskip .1in
We point out that the global $L^{2}$-bound of $\omega$ is insufficient for our purpose. This impels us to establish the global $L^{p}$-bound of $\omega$ for some $p>2$. To this end, we carry out an interesting iterative process, which would allow us to improve the global $L^{p}$-bound of $\omega$ for the $p$ as large as possible.
\begin{lemma}\label{addftlem2}
Let $\alpha,\beta\in(0,1)$ satisfy $\beta\geq\frac{4-2\alpha}{4+\alpha}$. If it holds
\begin{eqnarray}\label{vbhkp02}
\|\omega(t)\|_{L^{p_{k}}}^{p_{k}}
+\int_{0}^{t}{
\|\omega(\tau)\|_{L^{\frac{2p_{k}}{2-\alpha}}}^{p_{k}}\,d\tau}\leq
N
\end{eqnarray}
with $p_{k}\in \Big[2,\frac{2(2+\beta)}{(2+\alpha)\beta}\Big)$, then we have
\begin{eqnarray}\label{hgfdrtq11}
\|\omega(t)\|_{L^{p_{k+1}}}^{p_{k+1}}
+\int_{0}^{t}{
\|\omega(\tau)\|_{L^{\frac{2p_{k+1}}{2-\alpha}}}^{p_{k+1}}\,d\tau}\leq
C(t,\,u_{0},\,\theta_{0},\,N),
\end{eqnarray}
where $p_{k+1}\in \Big[2,\frac{2(2+\beta)}{(2+\alpha)\beta}\Big)$ satisfies
\begin{eqnarray}\label{gmawq86}
p_{k+1}=\frac{8\beta p_{k}}{(2-\alpha)(\beta p_{k}+2)}.
\end{eqnarray}
\end{lemma}

\begin{proof}
Firstly, we claim that if \eqref{vbhkp02} holds, then we have
\begin{align}\label{vbhkp016}
 \|\Lambda^{\delta_{k}}\theta(t)\|_{L^{2}}^{2}
+ \int_{0}^{t}
\|\Lambda^{\delta_{k}+\frac{\beta}{2}}\theta(\tau)\|_{L^{2}}^{2}\,d\tau\leq C(t,\,u_{0},\,\theta_{0},\,N),
\end{align}
where
$$\delta_{k}=\frac{\beta}{2}\left(\frac{2+\alpha}{2}p_{k}-1\right)\in (0,\,1).$$
To this end, applying $\Lambda^{\delta_{k}}$ to $(\ref{Bouss})_{2}$ and multiplying the resultant equation by $\Lambda^{\delta_{k}}\theta$, we obtain by integrating over $\mathbb R^2$
$$
\frac{1}{2}\frac{d}{dt} \|\Lambda^{\delta_{k}}\theta(t)\|_{L^{2}}^{2}
+ \|\Lambda^{\delta_{k}+\frac{\beta}{2}}\theta\|_{L^{2}}^{2}
= - \int_{\mathbb{R}^{2}}
[\Lambda^{\delta_{k}}, u \cdot \nabla]\theta\,\,\Lambda^{\delta_{k}}\theta\,dx.$$
Thanks to \eqref{tvcbmp6} and \eqref{dfexbg6}, we infer
\begin{align}
 -\int_{\mathbb{R}^{2}}
[\Lambda^{\delta_{k}}, u \cdot \nabla]\theta\,\,\Lambda^{\delta_{k}}\theta\,dx
&\leq C\|[\Lambda^{\delta_{k}}, u \cdot \nabla]\theta\|_{L^{\frac{2(2+\alpha)p_{k}}{(2+\alpha)p_{k}+2}}}
\|\Lambda^{\delta_{k}}\theta\|_{L^{\frac{2(2+\alpha)p_{k}}{(2+\alpha)p_{k}-2}}}
\nonumber\\
&\leq C\|\nabla u\|_{L^{\frac{(2+\alpha)p_{k}}{2}}}
\|\Lambda^{\delta_{k}}\theta\|_{L^{\frac{2(2+\alpha)p_{k}}{(2+\alpha)p_{k}-2}}}^{2}
\nonumber\\
&\leq C\|\omega\|_{L^{\frac{(2+\alpha)p_{k}}{2}}}
\|\Lambda^{\delta_{k}}\theta\|_{\dot{B}_{\infty,\infty}^{-\delta_{k}}}
^{\frac{4}{(2+\alpha)p_{k}}}
\|\Lambda^{\delta_{k}}
\theta\|_{\dot{B}_{2,2}^{\frac{\beta}{2}}}
^{\frac{2(2+\alpha)p_{k}-4}{(2+\alpha)p_{k}}}
\nonumber\\
&\leq C\|\omega\|_{L^{p_{k}}}^{\frac{\alpha}{2+\alpha}}\|\omega\|_{L^{\frac{2p_{k}}{2-\alpha}}}^{\frac{2}{2+\alpha}}
\|\theta\|_{L^{\infty}}
^{\frac{4}{(2+\alpha)p_{k}}}
\|\Lambda^{\delta_{k}+\frac{\beta}{2}}
\theta\|_{L^{2}}
^{\frac{2(2+\alpha)p_{k}-4}{(2+\alpha)p_{k}}}
\nonumber\\
&\leq
\frac{1}{2}\|\Lambda^{\delta_{k}+\frac{\beta}{2}}\theta\|_{L^{2}}^{2}+
C\|\theta\|_{L^{\infty}}
^{2}\|\omega\|_{L^{p_{k}}}^{\frac{\alpha p_{k}}{2}}
\|\omega\|_{L^{\frac{2p_{k}}{2-\alpha}}}^{p_{k}}.\nonumber
\end{align}
As a result, one deduces
\begin{align}
 \frac{d}{dt} \|\Lambda^{\delta_{k}}\theta(t)\|_{L^{2}}^{2}
+ \|\Lambda^{\delta_{k}+\frac{\beta}{2}}\theta\|_{L^{2}}^{2}
\leq
C\|\theta\|_{L^{\infty}}
^{2}\|\omega\|_{L^{p_{k}}}^{\frac{\alpha p_{k}}{2}}
\|\omega\|_{L^{\frac{2p_{k}}{2-\alpha}}}^{p_{k}}.\nonumber
\end{align}
Using \eqref{vbhkp02} and integrating the above inequality in time, the desired bound \eqref{vbhkp016} holds true. With \eqref{vbhkp016} in hand, we are in the position to show \eqref{hgfdrtq11}. To this end, we take the inner product of \eqref{sdhyrp56} with $|\omega|^{p_{k+1}-2}\omega$ to get
\begin{align}\label{vbhkp017}
 \frac{1}{p_{k+1}}\frac{d}{dt}\|\omega(t)\|_{L^{p_{k+1}}}^{p_{k+1}}
 +c\|\omega\|_{L^{\frac{2p_{k+1}}{2-\alpha}}}^{p_{k+1}}\leq \int_{\mathbb{R}^{2}}
{\partial_{x_{1}}\theta\,\,|\omega|^{p_{k+1}-2}\omega\,dx}.
\end{align}
To bound the righthand side term of \eqref{vbhkp017}, we resort to the following chain rule of fractional derivative in the Sobolev framework (see \cite[Proposition 3.1]{Christw})
\begin{eqnarray}\label{sdsvb999}
 \|\Lambda^{s}(|f|^{m-2}f)\|_{L^{p}}\leq C\|\Lambda^{s}f\|_{L^{q}}\|f\|_{L^{r(m-2)}}^{m-2},
\end{eqnarray}
where $s\in (0,1)$, $m\in (2,\infty)$ and $p, q, r\in(1, \infty)^{3}$ such that $\frac{1}{p}=\frac{1}{q}+\frac{1}{r}$.
Now we choose
\begin{align}\label{schpdt89}
r=\frac{(2+\alpha)\beta p_{k}}{(2+\alpha)\beta p_{k}+\alpha-2}\in (1,2),\qquad \lambda=\frac{(3r-2)p_{k+1}-4r}{\alpha r(p_{k+1}-2)}\in (0,1).
\end{align}
Due to $\beta\geq\frac{4-2\alpha}{4+\alpha}$ and $p_{k}\geq 2$, it is not hard to check $r\in (1,2)$. Moreover, we indeed have $\lambda\in (0,1)$ by taking $p_{k+1}$ as
\begin{align}\label{sdfgwe28}
\frac{4r}{3r-2}<p_{k+1}<\frac{(4-2\alpha)r}{(3-\alpha)r-2}.
\end{align}
Hence, taking advantage of \eqref{sdsvb999} and \eqref{dfexbg6}, we are able to conclude
\begin{align}\label{vbhkp021}
 \left|\int_{\mathbb{R}^{2}}
{\partial_{x_{1}}\theta\,\,|\omega|^{p_{k+1}-2}\omega\,dx}\right|&\leq C\|\Lambda^{1-\frac{\alpha}{2}}\theta\|_{L^{\frac{r}{r-1}}}
\|\Lambda^{\frac{\alpha}{2}}(|\omega|^{p_{k+1}-2}\omega)\|_{L^{r}}\nonumber\\
&\leq C\|\Lambda^{1-\frac{\alpha}{2}}\theta\|_{L^{\frac{r}{r-1}}}
\|\Lambda^{\frac{\alpha}{2}}\omega\|_{L^{2}}
\||\omega|^{p_{k+1}-2}\|_{L^{\frac{2r}
{2-r}}}\nonumber\\
&\leq C
\|\Lambda^{1-\frac{\alpha}{2}}\theta
\|_{\dot{B}_{\infty,\infty}^{-(1-\frac{\alpha}{2})}}^{\frac{2-r}{r}}
\|\Lambda^{1-\frac{\alpha}{2}}\theta
\|_{\dot{B}_{2,2}^{\delta_{k}+\frac{\alpha+\beta}{2}-1}
}^{\frac{2(r-1)}{r}}
\|\Lambda^{\frac{\alpha}{2}}\omega\|_{L^{2}}
\|\omega \|_{L^{\frac{2r(p_{k+1}-2)}
{2-r}}}^{p_{k+1}-2}
\nonumber\\
&\leq C
\|\theta
\|_{\dot{B}_{\infty,\infty}^{0}}^{\frac{2-r}{r}}
\|\Lambda^{\delta_{k}+\frac{\beta}{2}}\theta\|_{L^{2}}^{\frac{2(r-1)}{r}}
\|\Lambda^{\frac{\alpha}{2}}\omega\|_{L^{2}}
\|\omega \|_{L^{\frac{2r(p_{k+1}-2)}
{2-r}}}^{p_{k+1}-2}
\nonumber\\
&\leq C
\|\theta
\|_{L^{\infty}}^{\frac{2-r}{r}}
\|\Lambda^{\delta_{k}+\frac{\beta}{2}}\theta\|_{L^{2}}^{\frac{2(r-1)}{r}}
\|\Lambda^{\frac{\alpha}{2}}\omega\|_{L^{2}}
\|\omega \|_{L^{\frac{2r(p_{k+1}-2)}
{2-r}}}^{p_{k+1}-2}
\nonumber\\
&\leq C
\|\theta
\|_{L^{\infty}}^{\frac{2-r}{r}}
\|\Lambda^{\delta_{k}+\frac{\beta}{2}}\theta\|_{L^{2}}^{\frac{2(r-1)}{r}}
\|\Lambda^{\frac{\alpha}{2}}\omega\|_{L^{2}}
\|\omega \|_{L^{p_{k+1}}}^{(p_{k+1}-2)(1-\lambda)}
\|\omega \|_{L^{\frac{2p_{k+1}}{2-\alpha}}}^{(p_{k+1}-2)\lambda}
\nonumber\\
&\leq \frac{c}{2}\|\omega\|_{L^{\frac{2p_{k+1}}{2-\alpha}}}^{p_{k+1}}+
C(\|\Lambda^{\delta_{k}+\frac{\beta}{2}}\theta\|_{L^{2}}^{2}+
\|\Lambda^{\frac{\alpha}{2}}\omega\|_{L^{2}}^{2})\nonumber\\
& \quad \times
(1+\|\omega \|_{L^{p_{k+1}}}^{p_{k+1}}),
\end{align}
where in the last inequality we have used the fact
\begin{align}\label{schpdt90}
\frac{p_{k+1}}{p_{k+1}-(p_{k+1}-2)\lambda}\frac{3r-2}{r}=2.
\end{align}
Based on \eqref{schpdt90} and $\lambda$ as well as $r$ given by \eqref{schpdt89}, we deduce
\begin{align}\label{xcmkpw89}
p_{k+1}=\frac{8r}{(6+\alpha)r-4-2\alpha}=\frac{8\beta p_{k}}{(2-\alpha)(\beta p_{k}+2)},
\end{align}
which gives \eqref{gmawq86}.
Furthermore, due to $r<2$, we are able to show that the $p_{k+1}$ of \eqref{xcmkpw89} indeed satisfies \eqref{sdfgwe28}.
Putting \eqref{vbhkp021} into \eqref{vbhkp017}, we arrive at
\begin{align}\label{vbhkp022}
 \frac{d}{dt}\|\omega(t)\|_{L^{p_{k+1}}}^{p_{k+1}}
 +c\|\omega\|_{L^{\frac{2p_{k+1}}{2-\alpha}}}^{p_{k+1}}\leq C(\|\Lambda^{\delta_{k}+\frac{\beta}{2}}\theta\|_{L^{2}}^{2}+
\|\Lambda^{\frac{\alpha}{2}}\omega\|_{L^{2}}^{2})
(1+\|\omega \|_{L^{p_{k+1}}}^{p_{k+1}}).
\end{align}
Keeping in mind \eqref{vbhkp016} and \eqref{vbhkp01}, the desired \eqref{hgfdrtq11} follows from \eqref{vbhkp022}.
Consequently, we complete the proof of Lemma \ref{addftlem2}.
\end{proof}

\vskip .1in
By taking the limit, we can show the following improved $L^{p}$-bound of $\omega$, which plays an important role in deriving the boundedness of the vorticity and the gradient of the temperature.
\begin{lemma}\label{addftlem3}
If $\alpha,\beta\in(0,1)$ satisfy $\beta\geq\frac{4-2\alpha}{4+\alpha}$, then it holds
\begin{eqnarray}\label{xcyue5}
\|\omega(t)\|_{L^{p}}^{p}
+\int_{0}^{t}{
\|\omega(\tau)\|_{L^{\frac{2p}{2-\alpha}}}^{p}\,d\tau}\leq
C(t,\,u_{0},\,\theta_{0})
\end{eqnarray}
with any $p$ satisfying
\begin{eqnarray}\label{xcyue6}
2\leq p<\min\left\{\frac{2(\alpha+4\beta-2)}{(2-\alpha)\beta},\ \ \frac{2(2+\beta)}{(2+\alpha)\beta}\right\}.
\end{eqnarray}
\end{lemma}

\begin{proof}
Clearly, it follows from Lemma \ref{addftlem2} that
\begin{eqnarray}
\|\omega(t)\|_{L^{p_{k}}}^{p_{k}}
+\int_{0}^{t}{
\|\omega(\tau)\|_{L^{\frac{2p_{k}}{2-\alpha}}}^{p_{k}}\,d\tau}\leq
C(t,\,u_{0},\,\theta_{0}),\ \ \ k=1,2, \cdot\cdot\cdot \nonumber
\end{eqnarray}
for any $p_{k} \in \Big[2,\frac{2(2+\beta)}{(2+\alpha)\beta}\Big)$ satisfying
\begin{eqnarray}\label{awrety882}
p_{k+1}=\frac{8\beta p_{k}}{(2-\alpha)(\beta p_{k}+2)}.
\end{eqnarray}
Keeping in mind $p_{1}=2$, we derive from \eqref{awrety882} that
\begin{align}
p_{k}=\frac{2(\alpha+4\beta-2)}{(2-\alpha)\beta+[(2+\alpha)\beta+\alpha-2]
\big(\frac{2-\alpha}{4\beta}\big)^{k-1}}.\nonumber
\end{align}
Due to $\beta\geq\frac{4-2\alpha}{4+\alpha}>\frac{2-\alpha}{2+\alpha}>\frac{2-\alpha}{4}$, the sequence $\{p_{k}\}_{k\in \mathbb{N}}$ is increasing and it has the limit
$$\lim_{k\rightarrow\infty}p_{k}=\frac{2(\alpha+4\beta-2)}{(2-\alpha)\beta}.$$
We thus have
\begin{eqnarray}\label{cdsrt1}
2\leq p<\frac{2(\alpha+4\beta-2)}{(2-\alpha)\beta}.
\end{eqnarray}
Recalling the following restriction
$$p_{k}<\frac{2(2+\beta)}{(2+\alpha)\beta},\quad k=1,2,\cdot\cdot\cdot,$$
we further choose $p$ satisfying
\begin{eqnarray}\label{cdsrt2}
 p<\frac{2(2+\beta)}{(2+\alpha)\beta}.
\end{eqnarray}
Concerning \eqref{cdsrt1} and \eqref{cdsrt2}, $p$ should be satisfied \eqref{xcyue6}.
As a result, we finish the proof of Lemma \ref{addftlem3}.
\end{proof}

\vskip .1in
With \eqref{xcyue5} at our disposal, we are ready to complete the proof of Part 1 of Theorem \ref{Th3}. Here we provide two methods.
\vskip .1in

\begin{center}
\textbf{Method 1}
\end{center}
\vskip .1in

Taking advantage of \eqref{xcyue5}, we are able to prove the following key estimate.
\begin{lemma}\label{addftlem4}
If $\alpha,\beta\in(0,1)$ satisfy \eqref{dfg456}, then the following estimate holds
\begin{eqnarray}\label{sdhklp28}
\sup_{0\leq t\leq T}(\|\omega(t)\|_{L^{\infty}}+\|\nabla\theta (t)\|_{L^{\infty}})\leq
C(T,\,u_{0},\,\theta_{0}).
\end{eqnarray}
\end{lemma}

\begin{proof}
The proof is based on the argument
of nonlinear lower bounds for the fractional Laplacian established in \cite{CV}.
We first recall the following pointwise bound (see \eqref{tdfhb12} below)
\begin{eqnarray} \label{abbdghw569}
f(x)\Lambda^{\alpha}f(x)\geq \frac{1}{2}\Lambda^{\alpha}(|f(x)|^{2})+c\frac{|f(x)|^{2+\frac{\alpha p}{2}}}{\|f\|_{L^{p}}^{\frac{\alpha p}{2}}},
\end{eqnarray}
where $p\in [1,\infty)$ and $c=c(\alpha,p)>0$.
Multiplying \eqref{sdhyrp56} by $\omega$ and using \eqref{abbdghw569}, we are able to show that
\begin{align}\label{sdfbnkp15}
\frac{1}{2}\left\{\partial_{t}|\omega|^{2}+(u\cdot \nabla)|\omega|^{2}+\Lambda^{\alpha}|\omega|^{2}\right\}+c\frac{|\omega |^{2+\frac{\alpha p}{2}}}{\|\omega\|_{L^{p}}^{\frac{\alpha p}{2}}}\leq  \|\nabla\theta(t)\|_{L^{\infty}}|\omega|.
\end{align}
For any $t>0$, there exists $\tilde{x}=\tilde{x}(t)$ such that
$$\omega(\tilde{x}(t),t)=\|\omega(t)\|_{L^{\infty}}.$$
Moreover, one can show (see \cite{CC} for example)
$$(\partial_{t}|\omega|)(\tilde{x}(t),t)=\frac{d}{dt}\omega(\tilde{x}(t),t)
=\frac{d}{dt}\|\omega(t)\|_{L^{\infty}}.$$
Let us recall the direct observations that
$$(u\cdot \nabla)|\omega(\tilde{x}(t),t)|^{2}=0,\quad \Lambda^{\alpha}|\omega(\tilde{x}(t),t)|^{2}\geq0.$$
Consequently, letting $x=\tilde{x}(t)$ in \eqref{sdfbnkp15} and using the above facts, we arrive at
\begin{align}
\frac{1}{2}\frac{d}{dt}\|\omega(t)\|_{L^{\infty}}^{2}+c\frac{
\|\omega(t)\|_{L^{\infty}}^{2+\frac{\alpha p}{2}}}{\|\omega\|_{L^{p}}^{\frac{\alpha p}{2} }}
\leq  \|\nabla\theta(t)\|_{L^{\infty}}\|\omega\|_{L^{\infty}},\nonumber
\end{align}
which yields
\begin{align}\label{fgup98}
\frac{d}{dt}\|\omega(t)\|_{L^{\infty}}+c\frac{
\|\omega(t)\|_{L^{\infty}}^{1+\frac{\alpha p}{2}}}{\|\omega\|_{L^{p}}^{\frac{\alpha p}{2} }}
\leq  \|\nabla\theta(t)\|_{L^{\infty}}.
\end{align}
Making use of \eqref{xcyue5}, one gets from \eqref{fgup98} that
\begin{align}\label{dfhkmp78}
\frac{d}{dt}\|\omega(t)\|_{L^{\infty}}+c
\|\omega(t)\|_{L^{\infty}}^{1+\frac{\alpha p}{2}}
\leq  \|\nabla\theta(t)\|_{L^{\infty}},
\end{align}
where here and in what follows, $p$ should satisfy \eqref{xcyue6}.
Resorting to \cite[(6.7)]{CV} or \cite[(2.36)]{YX201502} and arguing as \eqref{dfhkmp78}, we may derive
\begin{eqnarray}
\frac{d}{dt}\|\nabla\theta(t)\|_{L^{\infty}}+c\frac{
\|\nabla\theta(t)\|_{L^{\infty}}^{1+\beta}}{\|\theta_{0}\|_{L^{\infty}}^{\beta}}
\leq \|\nabla u\|_{L^{\infty}}\|\nabla\theta(t)\|_{L^{\infty}},\nonumber
\end{eqnarray}
which yields
\begin{eqnarray}\label{tNew011}
\frac{d}{dt}\|\nabla\theta(t)\|_{L^{\infty}}+c
\|\nabla\theta(t)\|_{L^{\infty}}^{1+\beta}
\leq \|\nabla u\|_{L^{\infty}}\|\nabla\theta(t)\|_{L^{\infty}},
\end{eqnarray}
To control $\|\nabla u\|_{L^{\infty}}$, we resort to the following logarithmic inequality
\begin{align}\label{tNew013}
\|\nabla u(t)\|_{L^{\infty}} \leq C(1+\|\omega(t)\|_{L^{\infty}})\ln
\Big(e+\int_{0}^{t}{\big(1+\|\omega(\tau)\|_{L^{\infty}}+\|\nabla\theta(\tau)
\|_{L^{\infty}}
\big)^{\Gamma}\,d\tau}\Big),
\end{align}
where  $C=C(t,\|u_{0}\|_{H^{s}},\|\theta_{0}\|_{H^{s}})$ and  $\Gamma=\Gamma(\alpha,\beta)$ is a positive constant depending on $\alpha$ and $\beta$.
We point out that the proof of \eqref{tNew013} given by Constantin and Vicol  (see \cite[(6.10)]{CV}) requires the initial data belonging to $H^s(\mathbb{R}^2)$ for $s\geq3$.
Actually, \eqref{tNew013} is still valid for any $s>2$, whose proof is postponed in Appendix \ref{appSec2}. It thus follows from (\ref{tNew011}) that
\begin{align}\label{sdgkh661}
&\frac{d}{dt}\|\nabla\theta(t)\|_{L^{\infty}}+c
\|\nabla\theta(t)\|_{L^{\infty}}^{1+\beta}
\nonumber\\& \leq C\|\nabla\theta(t)\|_{L^{\infty}}(1+\|\omega(t)\|_{L^{\infty}}) \ln
\Big(e+\int_{0}^{t}{\big(1+\|\omega(\tau)\|_{L^{\infty}}+\|\nabla\theta(\tau)
\|_{L^{\infty}}
\big)^{\Gamma}\,d\tau}\Big).
\end{align}
Let $M>0$ be large enough to be fixed hereafter. Assuming the solutions blow up at time $T$, which implies
$$\lim_{t\rightarrow T}\|\nabla\theta(t)\|_{L^{\infty}}=\infty.$$
Then one takes $T_{0}\in (0,\,T)$ as the first time such that $$\|\nabla\theta(T_{0})\|_{L^{\infty}}=M\geq 4\|\nabla\theta_{0}\|_{L^{\infty}}.$$ In view of (\ref{dfhkmp78}), we see that for any $t\in [0,\,T_{0}]$
$$\|\omega(t)\|_{L^{\infty}}\leq\max\Big\{\|\omega_{0}\|_{L^{\infty}},\,\,\, \Big(\frac{M}{c}\Big)^{\frac{2}{2+\alpha p}}\Big\}
=\Big(\frac{M}{c}\Big)^{\frac{2}{2+\alpha p}}\triangleq \widetilde{M},$$
as long as $M$ is large enough in terms of $\|\omega_{0}\|_{L^{\infty}}$, $\alpha$, $p$ and $c$. Therefore, we deduce from \eqref{sdgkh661} that
\begin{align}\label{sdgkh662}
 \frac{d}{dt}\|\nabla\theta(t)\|_{L^{\infty}}+c
\|\nabla\theta(t)\|_{L^{\infty}}^{1+\beta}
  &\leq C\|\nabla\theta(t)\|_{L^{\infty}}\widetilde{M} \ln
\Big(e+\int_{0}^{t}{\big(1+\widetilde{M}+M
\big)^{\Gamma}\,d\tau}\Big)\nonumber\\  &\leq C\|\nabla\theta(t)\|_{L^{\infty}}\widetilde{M} \ln
\Big(e+ T\big(1+\widetilde{M}+M
\big)^{\Gamma} \Big).
\end{align}
Similarly, we obtain from \eqref{sdgkh662} that for any $t\in [0,\,T_{0}]$
\begin{eqnarray}
\|\nabla\theta(t)\|_{L^{\infty}}^{\beta}\leq \max\left\{\|\nabla\theta_{0}\|_{L^{\infty}}^{ \beta},\,\,\frac{C}{c}\widetilde{M} \ln
\Big(e+ T\big(1+\widetilde{M}+M
\big)^{\Gamma} \Big)\right\}.\nonumber
\end{eqnarray}
Keeping in mind $\widetilde{M}\approx M^{\frac{2}{2+\alpha p}}$, we fix $M$ large enough such that
$$\frac{C}{c}\widetilde{M} \ln
\Big(e+ T\big(1+\widetilde{M}+M
\big)^{\Gamma} \Big)\leq \Big(\frac{M}{4}\Big)^{\beta},$$
which is true as long as
$$\beta>\frac{2}{2+\alpha p}\Leftrightarrow p>\frac{2(1-\beta)}{\alpha \beta}.$$
Hence, we can verify that $\|\nabla\theta(T_{0})\|_{L^{\infty}}\leq\frac{M}{4}$, which contradicts the definition of $T_{0}$. It further indicates that $\|\nabla\theta(t)\|_{L^{\infty}}$ never blows up as $t\rightarrow T$ when $T<\infty$. As a result, we obtain
$$
\sup_{0\leq t\leq T}\|\nabla\theta (t)\|_{L^{\infty}}\leq
C(T,\,u_{0},\,\theta_{0}),
$$
which along with \eqref{dfhkmp78} gives
$$\sup_{0\leq t\leq T}\|\omega(t)\|_{L^{\infty}}\leq \|\omega_{0}\|_{L^{\infty}}
+\int_{0}^{T}{\|\nabla \theta(\tau)\|_{L^{\infty}}\,d \tau}\leq
C(T,\,u_{0},\,\theta_{0}).$$
We thus derive the desired bound \eqref{sdhklp28}. Finally, let us show the above $p$ is workable as long as \eqref{dfg456} holds. Actually, it should satisfy
\begin{eqnarray}\label{zfnvwrq89}
\max\left\{2,\ \frac{2(1-\beta)}{\alpha \beta}\right\}<p<\min\left\{\frac{2(\alpha+4\beta-2)}{(2-\alpha)\beta},\ \ \frac{2(2+\beta)}{(2+\alpha)\beta}\right\},
\end{eqnarray}
where
$$\beta\geq\frac{4-2\alpha}{4+\alpha},\quad \alpha\leq \frac{2}{3}.$$
On the one hand, one may check that
\begin{align}
\beta\geq\frac{4-2\alpha}{4+\alpha}>\frac{2-\alpha}{2+\alpha}\Rightarrow
2<\frac{2(\alpha+4\beta-2)}{(2-\alpha)\beta},\nonumber
\end{align}
\begin{align}
\beta>\frac{2+\alpha-\alpha^{2}}{2+3\alpha}\Rightarrow\frac{2(1-\beta)}{\alpha \beta}<\frac{2(\alpha+4\beta-2)}{(2-\alpha)\beta}.\nonumber
\end{align}
On the other hand, we are able to show
$$\beta<\frac{2}{1+\alpha}\Rightarrow 2<\frac{2(2+\beta)}{(2+\alpha)\beta},$$
$$\beta>\frac{2-\alpha}{2(1+\alpha)}\Rightarrow \frac{2(1-\beta)}{\alpha \beta}<\frac{2(2+\beta)}{(2+\alpha)\beta}.$$
Consequently, $\beta$ should obey
$$\max\left\{ \frac{4-2\alpha}{4+\alpha},\ \ \frac{2+\alpha-\alpha^{2}}{2+3\alpha},\ \ \frac{2-\alpha}{2(1+\alpha)}\right\}\equiv \frac{4-2\alpha}{4+\alpha}\leq
\beta<\frac{2}{1+\alpha},$$
which is workable.
Consequently, we complete the proof of Lemma \ref{addftlem4}.
\end{proof}

\begin{proof}[The proof of Part 1 of Theorem \ref{Th3}]
Keeping in mind \eqref{sdhklp28}, the global regularity follows immediately.
In fact, the standard energy method allows us to derive
\begin{align}\label{cvbmr82}
&\frac{d}{dt}(\|\Lambda^{s}u(t)\|_{L^{2}}^{2}+
\|\Lambda^{s}\theta(t)\|_{L^{2}}^{2})+\|\Lambda^{s+\frac{\alpha}{2}}
u\|_{L^{2}}^{2}+\|\Lambda^{s+\frac{\beta}{2}}\theta\|_{L^{2}}^{2}\nonumber\\
&\leq C (1+\|\nabla u\|_{L^{\infty}}+\|\nabla \theta\|_{L^{\infty}})(\|\Lambda^{s}u\|_{L^{2}}^{2}+
\|\Lambda^{s}\theta\|_{L^{2}}^{2}).
\end{align}
Recalling the following logarithmic Sobolev embedding inequality
\begin{eqnarray}\label{sdffgbj98}
\|\nabla u\|_{L^{\infty}}
\leq C\Big(1+\|u\|_{L^{2}}+
\|\omega\|_{L^{\infty}} \ln\big(e+\|
\Lambda^{s}u\|_{L^{2}}\big)\Big), \ \ s>2,
\end{eqnarray}
we deduce from \eqref{cvbmr82} that
\begin{align}\label{cvbmr83}
&\frac{d}{dt}(\|\Lambda^{s}u(t)\|_{L^{2}}^{2}+
\|\Lambda^{s}\theta(t)\|_{L^{2}}^{2})+\|\Lambda^{s+\frac{\alpha}{2}}
u\|_{L^{2}}^{2}+\|\Lambda^{s+\frac{\beta}{2}}\theta\|_{L^{2}}^{2}\nonumber\\
&\leq C (1+\|\omega\|_{L^{\infty}}+\|\nabla \theta\|_{L^{\infty}})\ln(e+\|\Lambda^{s}u\|_{L^{2}}^{2}+
\|\Lambda^{s}\theta\|_{L^{2}}^{2})(\|\Lambda^{s}u\|_{L^{2}}^{2}+
\|\Lambda^{s}\theta\|_{L^{2}}^{2}).
\end{align}
Applying the Log-Gronwall inequality to \eqref{cvbmr83} yields
\begin{eqnarray}
\|\Lambda^{s}u(t)\|_{L^{2}}^{2}+
\|\Lambda^{s}\theta(t)\|_{L^{2}}^{2}+\int_{0}^{t}{(\|\Lambda^{s+\frac{\alpha}{2}}
u\|_{L^{2}}^{2}+\|\Lambda^{s+\frac{\beta}{2}}\theta\|_{L^{2}}^{2})(\tau)\,d\tau}\leq C(t,\,u_{0},\,\theta_{0}).\nonumber
\end{eqnarray}
This ends the proof of Part 1 of Theorem \ref{Th3}.
\end{proof}
\vskip .1in

\begin{center}
\textbf{Method 2}
\end{center}
\vskip .1in
Obviously, \textbf{Method 1} depends heavily on the logarithmic inequality \eqref{tNew013}. Interestingly, our \textbf{Method 2} does not require \eqref{tNew013}. The proof here is based on a variant of nonlinear lower bound for the fractional Laplacian and the direct energy method.
We now show the following estimate which is enough to propagate all the regularities.
\begin{lemma}
If $\alpha,\beta\in(0,1)$ satisfy \eqref{dfg456}, then it holds for sufficiently large $q\in[2,\infty)$
\begin{eqnarray}\label{sdghg251}
\|\omega(t)\|_{L^{q}}+\|\nabla\theta (t)\|_{L^{\beta^{-}q}} \leq
C(t,\,u_{0},\,\theta_{0}),
\end{eqnarray}
where here and below $\beta^{-}>0$ denotes $\beta^{-}=\beta-\epsilon$ for arbitrarily small $\epsilon>0$.
\end{lemma}
\begin{proof}
To start, let us recall the following generalized pointwise bound (see Appendix \ref{appSec2} for its proof)
\begin{eqnarray} \label{tdfhb12}
|f(x)|^{r-2}f(x)\Lambda^{\alpha}f(x)\geq \frac{1}{r}\Lambda^{\alpha}(|f(x)|^{r})+c\frac{|f(x)|^{r+\frac{\alpha p}{2}}}{\|f\|_{L^{p}}^{\frac{\alpha p}{2}}},
\end{eqnarray}
where $r\in (1,\infty), p\in [1,\infty)$ and $c=c(\alpha,r,p)>0$.
Multiplying \eqref{sdhyrp56} by $|\omega|^{p-2}\omega$ and using \eqref{tdfhb12}, we are able to show that
\begin{align}\label{sdghg253}
\frac{1}{q}\left\{\partial_{t}|\omega|^{q}+(u\cdot \nabla)|\omega|^{q}+\Lambda^{\alpha}|\omega|^{q}\right\}+c\frac{|\omega|^{q+\frac{\alpha p}{2}}}{\|\omega\|_{L^{p}}^{\frac{\alpha p}{2}}}\leq& \partial_{x_{1}}\theta|\omega|^{q-2}\omega.
\end{align}
Integrating \eqref{sdghg253} over $\mathbb{R}^2$ and making use of \eqref{xcyue5}, one has
\begin{align} \label{sdghg254}
\frac{d}{dt}\|\omega(t)\|_{L^{q}}^{q}+c
\|\omega\|_{L^{q+\frac{\alpha p}{2}}}^{q+\frac{\alpha p}{2}}
\leq C\left|\int_{\mathbb{R}^{2}}\partial_{x_{1}}\theta|\omega|^{q-2}\omega\,dx\right|
\triangleq \mathcal{Z}_{1},
\end{align}
where $p$ satisfies \eqref{xcyue6}.
It follows from the $\theta$-equation that
\begin{align} \label{sdghg255}
\partial_{t}\nabla\theta+(u \cdot \nabla) \nabla\theta+ \Lambda^{\beta}\nabla\theta=-(\nabla u \cdot \nabla) \theta.
\end{align}
We now recall the following pointwise bound (see Appendix \ref{appSec2} for its proof)
\begin{eqnarray} \label{yyadkb02}
|\nabla f(x)|^{r-2}\nabla f(x)\Lambda^{\beta}\nabla f(x)\geq
\frac{1}{r}\Lambda^{\beta}(|\nabla f(x)|^{r})+c\frac{|\nabla f(x)|^{r+\frac{\beta p}{p+2}}}{\|f\|_{L^{p}}^{\frac{\beta p}{p+2}}},
\end{eqnarray}
where $r\in (1,\infty), p\in [1,\infty]$ and $c=c(\alpha,r,p)>0$.
Multiplying \eqref{sdghg255} by $|\nabla\theta|^{r-2}\nabla\theta$ and taking advantage of \eqref{yyadkb02}, we infer
\begin{align}\label{sdghg256}
\frac{1}{r}\left\{\partial_{t}|\nabla\theta|^{r}+(u\cdot \nabla)|\nabla\theta|^{r}+\Lambda^{\beta}|\nabla\theta|^{r}\right\}
+c\frac{|\nabla\theta|^{r+\beta}}{\|\theta\|_{L^{\infty}}^{\beta}}\leq&-(\nabla u \cdot \nabla) \theta|\nabla\theta|^{r-2}\nabla\theta.
\end{align}
Integrating \eqref{sdghg256} over $\mathbb{R}^2$, we have
\begin{align} \label{sdghg257}
\frac{d}{dt}\|\nabla\theta(t)\|_{L^{r}}^{r}+c
\|\nabla\theta\|_{L^{r+\beta}}^{r+\beta}
\leq & C\left|\int_{\mathbb{R}^{2}}(\nabla u \cdot \nabla) \theta|\nabla\theta|^{r-2}\nabla\theta\,dx\right|\triangleq \mathcal{Z}_{2}.
\end{align}
Putting \eqref{sdghg254} and \eqref{sdghg257} together yields
\begin{align} \label{sdghg258}
\frac{d}{dt}(\|\omega(t)\|_{L^{q}}^{q}+\|\nabla\theta(t)\|_{L^{r}}^{r})+c
\|\omega\|_{L^{q+\frac{\alpha p}{2}}}^{q+\frac{\alpha p}{2}}+c
\|\nabla\theta\|_{L^{r+\beta}}^{r+\beta}
\leq \mathcal{Z}_{1}+\mathcal{Z}_{2}.
\end{align}
If we take $q$ and $r$ as
\begin{align} \label{sdghg259}
q>\frac{r+\beta}{\beta}-\frac{\alpha p}{2},
\end{align}
then we obtain
\begin{align}
\mathcal{Z}_{2}&\leq C\|\nabla u\|_{L^{\frac{r+\beta}{\beta}}}\|\nabla\theta\|_{L^{r+\beta}}^{r}
\nonumber\\&\leq C\|\omega\|_{L^{\frac{r+\beta}{\beta}}}\|\nabla\theta\|_{L^{r+\beta}}^{r}\nonumber\\&\leq C\|\omega\|_{L^{2}}^{1-\frac{(r-\beta)(2q+\alpha p)}{(r+\beta)(2q+\alpha p-4)}}\|\omega\|_{L^{q+\frac{\alpha p}{2}}}^{\frac{(r-\beta)(2q+\alpha p)}{(r+\beta)(2q+\alpha p-4)}}\|\nabla\theta\|_{L^{r+\beta}}^{r}\nonumber\\&\leq C \|\omega\|_{L^{q+\frac{\alpha p}{2}}}^{\frac{(r-\beta)(2q+\alpha p)}{(r+\beta)(2q+\alpha p-4)}}\|\nabla\theta\|_{L^{r+\beta}}^{r}
\nonumber\\&\leq C+
\frac{c}{4}
\|\omega\|_{L^{q+\frac{\alpha p}{2}}}^{q+\frac{\alpha p}{2}}+\frac{c}{4}
\|\nabla\theta\|_{L^{r+\beta}}^{r+\beta}.\nonumber
\end{align}
We further take $q$ and $r$ as
\begin{align} \label{sdghg2510}
q\leq \frac{(\alpha p+2)r-\alpha p}{2},
\end{align}
then one obtains
\begin{align}
\mathcal{Z}_{1}&\leq C\|\nabla \theta\|_{L^{r}}\|\omega\|_{L^{\frac{r(q-1)}{r-1}}}^{q-1}
\nonumber\\&\leq C \|\nabla \theta\|_{L^{r}}\|\omega\|_{L^{2}}^{q-1-\frac{[(q-3)r+2](2q+\alpha p)}{r(2q+\alpha p-4)}}\|\omega\|_{L^{q+\frac{\alpha p}{2}}}^{\frac{[(q-3)r+2](2q+\alpha p)}{r(2q+\alpha p-4)}}
\nonumber\\&\leq C \|\nabla \theta\|_{L^{r}} \|\omega\|_{L^{q+\frac{\alpha p}{2}}}^{\frac{[(q-3)r+2](2q+\alpha p)}{r(2q+\alpha p-4)}}
\nonumber\\&\leq
\frac{c}{4}
\|\omega\|_{L^{q+\frac{\alpha p}{2}}}^{q+\frac{\alpha p}{2}} +C(1+\|\nabla\theta\|_{L^{r}}^{r}).\nonumber
\end{align}
Putting the above estimates into \eqref{sdghg258} yields
\begin{align}
\frac{d}{dt}(\|\omega(t)\|_{L^{q}}^{q}+\|\nabla\theta(t)\|_{L^{r}}^{r})+c
\|\omega\|_{L^{q+\frac{\alpha p}{2}}}^{q+\frac{\alpha p}{2}}+c
\|\nabla\theta\|_{L^{r+\beta}}^{r+\beta}
\leq C(1+\|\nabla\theta\|_{L^{r}}^{r}).\nonumber
\end{align}
This together with the Gronwall inequality yields
\begin{eqnarray}\label{sdghg2511}
\|\omega(t)\|_{L^{q}}+\|\nabla\theta (t)\|_{L^{r}} \leq
C(t,\,u_{0},\,\theta_{0}),
\end{eqnarray}
where $q$ and $r$ satisfy \eqref{sdghg259} and \eqref{sdghg2510}. To ensure the workable of $q$ between \eqref{sdghg259} and \eqref{sdghg2510}, it requires
\begin{eqnarray}\label{sdghg2512}
p> \frac{2(1-\beta)}{\alpha \beta},
\end{eqnarray}
which is the same as the lower bound of $p$ stated in \eqref{zfnvwrq89}.
Noticing \eqref{sdghg2512}, we take
$$r=(\beta-\epsilon)q$$
with arbitrarily small $\epsilon\in (0,\beta)$, then there exists $q_{0}=q_{0}(\alpha,\beta,\epsilon)$ such that
$$q\geq q_{0}.$$
Therefore, we derive from \eqref{sdghg2511} that
\begin{eqnarray}
\|\omega(t)\|_{L^{q}}+\|\nabla\theta (t)\|_{L^{(\beta-\epsilon)q}} \leq
C(t,\,u_{0},\,\theta_{0}),\quad q\geq q_{0}.\nonumber
\end{eqnarray}
As a result, we complete the proof of \eqref{sdghg251}.
\end{proof}

\begin{proof}[The proof of Part 1 of Theorem \ref{Th3}]
By the standard energy method, we derive
\begin{align}\label{sdghg2513}
& \frac{1}{2}\frac{d}{dt}(\|\Lambda^{s}u(t)\|_{L^{2}}^{2}+
\|\Lambda^{s}\theta(t)\|_{L^{2}}^{2})+\|\Lambda^{s+\frac{\alpha}{2}}
u\|_{L^{2}}^{2}+\|\Lambda^{s+\frac{\beta}{2}}
\theta\|_{L^{2}}^{2}\nonumber\\
&=  \int_{\mathbb{R}^{2}}{\Lambda^{s}(\theta e_{2})
\Lambda^{s}u\,dx}-\int_{\mathbb{R}^{2}}{[\Lambda^{s},
u\cdot\nabla]u\cdot \Lambda^{s}u\,dx}-
\int_{\mathbb{R}^{2}}{[\Lambda^{s}, u\cdot\nabla]\theta
\Lambda^{s}\theta\,dx}\nonumber\\
&\triangleq L_{1}+L_{2}+L_{3}.
\end{align}
Obviously, we have
 \begin{eqnarray}
|L_{1}|\leq \|\Lambda^{s}\theta\|_{L^{2}}\|\Lambda^{s}u\|_{L^{2}} \leq C(\|\Lambda^{s}u\|_{L^{2}}^{2}+
\|\Lambda^{s}\theta\|_{L^{2}}^{2}).\nonumber
\end{eqnarray}
By the classical commutator estimate, it follows that
\begin{align}
|L_{2}| \leq&  \|[\Lambda^{s},
u\cdot\nabla]u\|_{L^{\frac{2q}{q+1}}}\|\Lambda^{s}u\|_{L^{\frac{2q}{q-1}}}
\nonumber\\  \leq& C
\|\nabla u\|_{L^{q}}\|\Lambda^{s}u\|_{L^{\frac{2q}{q-1}}}^{2}
\nonumber\\  \leq& C
\|\omega\|_{L^{q}}\|\Lambda^{s}u\|_{L^{2}}^{2-\frac{4}{\alpha q}} \|\Lambda^{s+\frac{\alpha}{2}}u\|_{L^{2}}^{\frac{4}{\alpha q}}
\nonumber\\  \leq&
\frac{1}{4}\|\Lambda^{s+\frac{\alpha}{2}}
u\|_{L^{2}}^{2}+C\|\omega\|_{L^{q}}^{\frac{\alpha q}{\alpha q-2}}\|\Lambda^{s}u\|_{L^{2}}^{2},\nonumber
\end{align}
where $q>\frac{2}{\alpha}$. Similarly, one also gets
\begin{align}
|L_{3}| \leq&  \|[\Lambda^{s},
u\cdot\nabla]\theta\|_{L^{\frac{2q}{q+1}}}\|\Lambda^{s}\theta\|_{L^{\frac{2q}{q-1}}}
\nonumber\\  \leq& C
\left(\|\nabla u\|_{L^{q}}\|\Lambda^{s}\theta\|_{L^{\frac{2q}{q-1}}}+
\|\nabla \theta\|_{L^{\beta^{-}q}}\|\Lambda^{s}u\|_{L^{\frac{2\beta^{-}
q}{\beta^{-}(q+1)-2}}}\right)\|\Lambda^{s}\theta\|_{L^{\frac{2q}{q-1}}}
\nonumber\\  \leq& C
\|\omega\|_{L^{q}}\|\Lambda^{s}\theta\|_{L^{\frac{2q}{q-1}}}^{2} +C\|\nabla \theta\|_{L^{\beta^{-}q}}\|\Lambda^{s}u\|_{L^{\frac{2\beta^{-}
q}{\beta^{-}(q+1)-2}}} \|\Lambda^{s}\theta\|_{L^{\frac{2q}{q-1}}}
\nonumber\\  \leq& C
\|\omega\|_{L^{q}}\|\Lambda^{s}\theta\|_{L^{2}}^{2-\frac{4}{\beta q}} \|\Lambda^{s+\frac{\beta}{2}}\theta\|_{L^{2}}^{\frac{4}{\beta q}}\nonumber\\&+
C
\|\nabla \theta\|_{L^{\beta^{-}q}}
\|\Lambda^{s}u\|_{L^{2}}^{1-\frac{2(2-\beta^{-})}
{\alpha\beta^{-}q}} \|\Lambda^{s+\frac{\alpha}{2}}u\|_{L^{2}}^{\frac{2(2-\beta^{-})}
{\alpha\beta^{-}q}}
\|\Lambda^{s}\theta\|_{L^{2}}^{1-\frac{2}{\beta q}} \|\Lambda^{s+\frac{\beta}{2}}\theta\|_{L^{2}}^{\frac{2}{\beta q}}
\nonumber\\  \leq&
\frac{1}{4}\|\Lambda^{s+\frac{\alpha}{2}}
u\|_{L^{2}}^{2}+\frac{1}{2}\|\Lambda^{s+\frac{\beta}{2}}
\theta\|_{L^{2}}^{2}\nonumber\\&+C\left(\|\omega\|_{L^{q}}^{\frac{\beta q}{\beta q-2}}+\|\nabla \theta\|_{L^{\beta^{-}q}}^{\frac{\alpha\beta\beta^{-}q}{\alpha\beta^{-}(\beta q-
1)-(2-\beta^{-})\beta}}\right)(\|\Lambda^{s}u\|_{L^{2}}^{2}+
\|\Lambda^{s}\theta\|_{L^{2}}^{2}),\nonumber
\end{align}
where $q$ further satisfies
$$q>\max\left\{\frac{2}{\beta},\quad \frac{2(2-\beta^{-})}{\alpha\beta^{-}}\right\}.$$
Putting all the above estimates into \eqref{sdghg2513} shows
\begin{align} \label{sdghg2514}
&  \frac{d}{dt}(\|\Lambda^{s}u(t)\|_{L^{2}}^{2}+
\|\Lambda^{s}\theta(t)\|_{L^{2}}^{2})+\|\Lambda^{s+\frac{\alpha}{2}}
u\|_{L^{2}}^{2}+\|\Lambda^{s+\frac{\beta}{2}}
\theta\|_{L^{2}}^{2}\nonumber\\
&\leq C\left(\|\omega\|_{L^{q}}^{\frac{\alpha q}{\alpha q-2}}+\|\omega\|_{L^{q}}^{\frac{\beta q}{\beta q-2}}++\|\nabla \theta\|_{L^{\beta^{-}q}}^{\frac{\alpha\beta\beta^{-}q}{\alpha\beta^{-}(\beta q-
1)-(2-\beta^{-})\beta}}\right)(\|\Lambda^{s}u\|_{L^{2}}^{2}+
\|\Lambda^{s}\theta\|_{L^{2}}^{2}),
\end{align}
here $q$ satisfies
$$q>\max\left\{\frac{2}{\alpha},\quad \frac{2}{\beta},\quad \frac{2(2-\beta^{-})}{\alpha\beta^{-}}\right\}.$$
Using \eqref{sdghg251} and applying the Gronwall inequality to \eqref{sdghg2514}, we thus obtain
\begin{eqnarray}
\|\Lambda^{s}u(t)\|_{L^{2}}^{2}+
\|\Lambda^{s}\theta(t)\|_{L^{2}}^{2}+\int_{0}^{t}{(\|\Lambda^{s+\frac{\alpha}{2}}
u\|_{L^{2}}^{2}+\|\Lambda^{s+\frac{\beta}{2}}\theta\|_{L^{2}}^{2})(\tau)\,d\tau}\leq C(t,\,u_{0},\,\theta_{0}).\nonumber
\end{eqnarray}
This completes the proof of Part 1 of Theorem \ref{Th3}.
\end{proof}

\vskip .2in
\subsection{The proof of Part 2}
This subsection is devoted to the proof of $\alpha$ and $\beta$ satisfying \eqref{dfgp66}. Here we point out that all the estimates obtained of
$\mbox{Part 2}$ are valid for all $\beta>\max\{\frac{4-6\alpha}{\alpha},\ \frac{2-\alpha-2\alpha^{2}}{2\alpha}\}$ with $\alpha\leq\frac{2}{3}$. In this case (mainly the case $\alpha>\beta$), we still expect that the following combined quantity $G$ (see \cite{JMWZ,MX,SW,SWXY,wuxuxueye,Yemmas}) plays a certain role in deriving the regularity of the solution
$$
G=\omega-\mathcal {R}_{\alpha}\theta,\qquad \mathcal {R}_{\alpha}\triangleq\partial_{x_{1}}\Lambda^{-\alpha}.
$$
It follows from \eqref{sdhyrp56} and the $\theta$-equation that
\begin{eqnarray}\label{t305}
\partial_{t}G+(u\cdot\nabla)G+\Lambda^{\alpha}G=[\mathcal {R}_{\alpha},\,u\cdot\nabla]\theta+\Lambda^{\beta-\alpha}\partial_{x_{1}}\theta.
\end{eqnarray}
By the Biot-Savart law, we have
\begin{eqnarray*}
u=\nabla^{\perp}\Delta^{-1}\omega
=\nabla^{\perp}\Delta^{-1}(G+\mathcal {R}_{\alpha}\theta)=\nabla^{\perp}\Delta^{-1}G+\nabla^{\perp}\Delta^{-1}\mathcal {R}_{\alpha}\theta\triangleq u_{G}+u_{\theta}.
\end{eqnarray*}
To handle the term $[\mathcal {R}_{\alpha},\,u\cdot\nabla]\theta$ at the righthand side of \eqref{t305}, we need the following commutator estimate (see \cite[Lemma 4.1]{SWXY}).
\begin{lemma}
Let $r\in[1,\infty]$ and $p,\,p_{1},\,p_{2}\in (1, \infty)$ satisfy $\frac{1}{p}=\frac{1}{p_{1}}+\frac{1}{p_{2}}$. Assume that $\epsilon\in[0,1)$ and $\alpha\in (0,1)$ satisfy $s\in(-1, \alpha-\epsilon)$. If $f$ is a divergence-free vector field, then it holds true
\begin{align}\label{uk3rtr03}
\|[\mathcal{R}_{\alpha},f\cdot\nabla]g\|_{\dot{B}_{p,r}^{s}}\leq
C \|\Lambda^{1-\epsilon} f\|_{L^{p_{1}}}\|g\|_{
{\dot{B}}_{p_{2},r}^{s+1+\epsilon-\alpha}}.
\end{align}
\end{lemma}

\begin{rem}
After checking the proof of \cite[Lemma 4.1]{SWXY}, it still holds
\begin{align}\label{xcdksdr88}
\|[\mathcal{R}_{\alpha},f\cdot\nabla]g\|_{\dot{B}_{\infty,1}^{0}}\leq
C \|\nabla f\|_{L^{\infty}}\|g\|_{
{\dot{B}}_{\infty,1}^{1-\alpha}}.
\end{align}
\end{rem}

\vskip .1in
Based on \eqref{t305} and \eqref{uk3rtr03}, we are able to show the following estimate.
\begin{lemma}
If $\beta>\max\{1-\alpha,\ \frac{4-6\alpha}{\alpha} \}$ with $\alpha\leq \frac{2}{3}$,
then it holds
\begin{eqnarray}\label{sfdsjkp2}
\|G(t)\|_{L^{2}}^{2}+\|\Lambda^{\frac{(2+\alpha)\beta}{4}}\theta(t)\|_{L^{2}}^{2}
+\int_{0}^{t}{( \|\Lambda^{\frac{\alpha}{2}}G \|_{L^{2}}^{2}+
\|\Lambda^{\frac{(4+\alpha)\beta}{4}}\theta\|_{L^{2}}^{2})(\tau)\,d\tau}\leq
C(t,u_{0},\theta_{0}).
\end{eqnarray}
\end{lemma}
\begin{proof}
Testing (\ref{t305}) by $G$ and integrating in the space variable, we have
\begin{align}\label{AZyyy1}
\frac{1}{2}\frac{d}{dt}\|G(t)\|_{L^{2}}^{2}
+\|\Lambda^{\frac{\alpha}{2}}G\|_{L^{2}}^{2} =&\int_{\mathbb{R}^{2}}
{[\mathcal
{R}_{\alpha},\,u\cdot\nabla]\theta\,\,G\,dx}+\int_{\mathbb{R}^{2}}
{\Lambda^{\beta-\alpha}\partial_{x_{1}}\theta\,\, G\,dx}\nonumber\\ =&\int_{\mathbb{R}^{2}}
{[\mathcal
{R}_{\alpha},\,u_{G}\cdot\nabla]\theta\,\,G\,dx}+\int_{\mathbb{R}^{2}}
{[\mathcal
{R}_{\alpha},\,u_{\theta}\cdot\nabla]\theta\,\,G\,dx}\nonumber\\&
+\int_{\mathbb{R}^{2}}
{\Lambda^{\beta-\alpha}\partial_{x_{1}}\theta\,\,G\,dx}.
\end{align}
Combining \eqref{AZyyy1} and \eqref{ksdh95}, we have
\begin{align} \label{vbmk445}
&\frac{1}{2}\frac{d}{dt}\left(\|G(t)\|_{L^{2}}^{2}+ \|\Lambda^{\frac{(2+\alpha)\beta}{4}}\theta(t)\|_{L^{2}}^{2}\right)
+ \|\Lambda^{\frac{\alpha}{2}}G \|_{L^{2}}^{2}+\|\Lambda^{\frac{(4+\alpha)\beta}{4}}\theta\|_{L^{2}}^{2}
\nonumber\\ &=\int_{\mathbb{R}^{2}}
{[\mathcal
{R}_{\alpha},\,u_{G}\cdot\nabla]\theta\,\,G\,dx}+\int_{\mathbb{R}^{2}}
{[\mathcal
{R}_{\alpha},\,u_{\theta}\cdot\nabla]\theta\,\,G\,dx}
+\int_{\mathbb{R}^{2}}
{\Lambda^{\beta-\alpha}\partial_{x_{1}}\theta\,\,G\,dx} \nonumber\\& \quad-  \int_{\mathbb{R}^{2}}
[\Lambda^{\frac{(2+\alpha)\beta}{4}}, u \cdot \nabla]\theta\,\,\Lambda^{\frac{(2+\alpha)\beta}{4}}\theta\,dx
\nonumber\\ &\triangleq \sum_{i=1}^{4}A_{i}.
\end{align}
As a start, we handle the term $A_{1}$ which is divided into two cases, namely
$$\beta>\frac{4-4\alpha-3\alpha^{2}}{\alpha}$$
and
$$\frac{4-6\alpha}{\alpha}<\beta\leq\frac{4-4\alpha-3\alpha^{2}}{\alpha}.$$
We begin with the former case.
In view of \eqref{uk3rtr03} and \eqref{dfexbg6}, we derive for any sufficiently small $\epsilon>0$ that
\begin{align}\label{asfbmu12}
A_{1}&\leq C\|[\mathcal{R}_{\alpha},u_{G}\cdot\nabla]\theta
\|_{\dot{B}_{2,2}^{-\frac{\alpha}{2}}}\|G\|_{\dot{B}_{2,2}^{\frac{\alpha}{2}}}
\nonumber\\&\leq C \|\Lambda u_{G}\|_{L^{p}}\|\theta\|_{
{\dot{B}}_{\frac{2p}{p-2},2}^{1-\frac{3\alpha}{2}}} \|\Lambda^{\frac{\alpha}{2}}G\|_{L^{2}}
\nonumber\\&\leq C \|G\|_{L^{p}}(\|\theta\|_{L^{2}}+\|\Lambda^{1-\frac{3\alpha}{2}+\epsilon}\theta\|_{L^{\frac{2p}{p-2}}})
\|\Lambda^{\frac{\alpha}{2}}G\|_{L^{2}}  \nonumber
 \\&\leq C \|G\|_{L^{2}}^{1-\frac{2(p-2)}{\alpha p}}\|\Lambda^{\frac{\alpha}{2}}G\|_{L^{2}}^{\frac{2(p-2)}{\alpha p}}\left(1+\|\Lambda^{1-\frac{3\alpha}{2}
+\epsilon}\theta\|_{\dot{B}_{\infty,\infty}^{-(1-\frac{3\alpha}{2}+\epsilon)}}
^{\frac{2}{p}}
\|\Lambda^{1-\frac{3\alpha}{2}
+\epsilon}\theta\|_{\dot{B}_{2,2}^{\frac{\beta}{2}-(1-\frac{3\alpha}{2}+\epsilon)}}
^{\frac{p-2}{p}}
\right)\nonumber\\&\quad\times
\|\Lambda^{\frac{\alpha}{2}}G\|_{L^{2}}
\nonumber\\&\leq C \|G\|_{L^{2}}^{1-\frac{2(p-2)}{\alpha p}}\|\Lambda^{\frac{\alpha}{2}}G\|_{L^{2}}^{1+\frac{2(p-2)}{\alpha p}}\left(1+\|\theta\|_{L^{\infty}}
^{\frac{2}{p}}
\|\Lambda^{\frac{\beta}{2}
}\theta\|_{L^{2}}
^{\frac{p-2}{p}}
\right)\nonumber\\&\leq C \|G\|_{L^{2}}^{1-\frac{2(p-2)}{\alpha p}}\|\Lambda^{\frac{\alpha}{2}}G\|_{L^{2}}^{1+\frac{2(p-2)}{\alpha p}}\left(1+
\|\Lambda^{\frac{\beta}{2}
}\theta\|_{L^{2}}
^{\frac{p-2}{p}}
\right)\nonumber\\
&\leq
\frac{1}{16}\|\Lambda^{\frac{\alpha}{2}}G\|_{L^{2}}^{2}+C(1+\|\Lambda^{\frac{\beta}{2}
}\theta\|_{L^{2}}^{2})(1+\|G\|_{L^{2}}^{2}),
\end{align}
where
\begin{align} \label{asdsfb89}
p=2+\frac{2(2-3\alpha+2\epsilon)}{3\alpha+\beta-2-2\epsilon}\leq 2+\alpha.
\end{align}
According to the arbitrariness of $\epsilon>0$, we deduce from \eqref{asdsfb89} that
\begin{align}
\beta>\frac{4-4\alpha-3\alpha^{2}}{\alpha}.\nonumber
\end{align}
The later case is a bit more tricky. To this end, we take
$$\kappa=\frac{2[(2+\alpha)(2-3\alpha+2\epsilon)-\alpha\beta]}{\alpha(2+\alpha)
\beta+4(3\alpha-2-2\epsilon)}\in (0,1),$$
$$\widetilde{s}=\frac{\beta}{2}+\frac{(2+\alpha)\beta\kappa}{4},\quad p=\frac{4\widetilde{s}}{2\widetilde{s}+3\alpha-2-2\epsilon}$$
for any sufficiently small $\epsilon>0$, then via direct computations we obtain
$$p=2+\frac{\alpha}{1+\kappa}.$$
In view of the above fixed $\kappa, \widetilde{s}$ and $p$, we deduce from \eqref{uk3rtr03} and \eqref{dfexbg6} that
\begin{align}\label{addsry897}
A_{1}&\leq C\|[\mathcal{R}_{\alpha},u_{G}\cdot\nabla]\theta
\|_{\dot{B}_{2,2}^{-\frac{\alpha}{2}}}\|G\|_{\dot{B}_{2,2}^{\frac{\alpha}{2}}}
\nonumber\\&\leq C \|\Lambda u_{G}\|_{L^{p}}\|\theta\|_{
{\dot{B}}_{\frac{2p}{p-2},2}^{1-\frac{3\alpha}{2}}} \|\Lambda^{\frac{\alpha}{2}}G\|_{L^{2}}
\nonumber\\&\leq  C \|G\|_{L^{p}}(\|\theta\|_{L^{2}}+\|\Lambda^{1-\frac{3\alpha}{2}+\epsilon}\theta\|_{L^{\frac{2p}{p-2}}})
\|\Lambda^{\frac{\alpha}{2}}G\|_{L^{2}}
\nonumber\\&\leq C
\|\Lambda^{1-\frac{2}{p}}G\|_{L^{2}}(\|\theta\|_{L^{2}}+\|\Lambda^{1-\frac{3\alpha}{2}+\epsilon}\theta\|_{L^{\frac{2p}{p-2}}})
\|\Lambda^{\frac{\alpha}{2}}G\|_{L^{2}}
\nonumber\\&\leq C\|\Lambda^{\frac{\alpha}{2}-1}G\|_{L^2}^{\frac{(p-2)\kappa}{p}} \|\Lambda^{\frac{\alpha[2-(p-2)\kappa]}{2[p-(p-2)\kappa]}}G\|_{L^{2}}
^{1-\frac{(p-2)\kappa}{ p}}\nonumber\\&\quad\times\left(1+\|\Lambda^{1-\frac{3\alpha}{2}
+\epsilon}\theta\|_{\dot{B}_{\infty,\infty}^{-(1-\frac{3\alpha}{2}+\epsilon)}}
^{\frac{2}{p}}
\|\Lambda^{1-\frac{3\alpha}{2}
+\epsilon}\theta\|_{\dot{B}_{2,2}^{\widetilde{s}-(1-\frac{3\alpha}{2}+\epsilon)}}
^{\frac{p-2}{p}}
\right)
\|\Lambda^{\frac{\alpha}{2}}G\|_{L^{2}}
\nonumber\\&\leq C\|\Lambda^{\frac{\alpha}{2}-1}G\|_{L^2}^{\frac{(p-2)\kappa}{p}} \|G\|_{L^{2}}^{\frac{p-2}{ p}}\|\Lambda^{\frac{\alpha}{2}}G\|_{L^{2}}^{\frac{2-(p-2)\kappa}{ p}}\left(1+\|\theta\|_{L^{\infty}}
^{\frac{2}{p}}
\|\Lambda^{\widetilde{s}
}\theta\|_{L^{2}}
^{\frac{p-2}{p}}
\right)\|\Lambda^{\frac{\alpha}{2}}G\|_{L^{2}}
\nonumber\\&\leq C\|\Lambda^{\frac{\alpha}{2}-1}G\|_{L^2}^{\frac{(p-2)\kappa}{p}} \|G\|_{L^{2}}^{\frac{p-2}{ p}}\|\Lambda^{\frac{\alpha}{2}}G\|_{L^{2}}^{1+\frac{2-(p-2)\kappa}{ p}}\left(1+
\|\Lambda^{\widetilde{s}
}\theta\|_{L^{2}}
^{\frac{p-2}{p}}
\right)
\nonumber\\&\leq C\left(\|\Lambda^{\frac{\alpha}{2}}u\|_{L^2}+
\|\theta\|_{L^{\frac{4}{2+\alpha}}}\right)^{\frac{(p-2)\kappa}{p}} \|G\|_{L^{2}}^{\frac{p-2}{ p}}\|\Lambda^{\frac{\alpha}{2}}G\|_{L^{2}}^{1+\frac{2-(p-2)\kappa}{ p}}\nonumber\\&\quad\times\left(1+
\|\Lambda^{\frac{\beta}{2}
}\theta\|_{L^{2}}
^{\frac{(p-2)(1-\kappa)}{p}}\|\Lambda^{\frac{(4+\alpha)\beta}{4}
}\theta\|_{L^{2}}
^{\frac{(p-2)\kappa}{p}}
\right)\nonumber\\
&\leq
\frac{1}{16}\|\Lambda^{\frac{\alpha}{2}}G\|_{L^{2}}^{2}+\frac{1}{16}
\|\Lambda^{\frac{(4+\alpha)\beta}{4}}\theta\|_{L^{2}}^{2}
+C(1+\|\Lambda^{\frac{\alpha}{2}}u\|_{L^2}^{2}+\|\Lambda^{\frac{\beta}{2}
}\theta\|_{L^{2}}^{2})(1+\|G\|_{L^{2}}^{2}),
\end{align}
where we have used the following fact
\begin{align}\label{sdfghp28}
\|\Lambda^{\frac{\alpha}{2}-1}G\|_{L^2}&\leq  \|\Lambda^{\frac{\alpha}{2}-1}\omega\|_{L^2}+
\|\Lambda^{\frac{\alpha}{2}-1}\mathcal{R}_{\alpha}\theta\|_{L^2} \nonumber\\
&\leq  C\left(\|\Lambda^{\frac{\alpha}{2}}u\|_{L^2}+
\|\Lambda^{-\frac{\alpha}{2}}\theta\|_{L^2}\right)
\nonumber\\
&\leq C
\left(\|\Lambda^{\frac{\alpha}{2}}u\|_{L^2}+
\|\theta\|_{L^{\frac{4}{2+\alpha}}}\right)\nonumber\\
&\leq C
\left(1+\|\Lambda^{\frac{\alpha}{2}}u\|_{L^2}\right).
\end{align}
In the last inequality of \eqref{sdfghp28}, we have used
$$\|\theta\|_{L^{\frac{4}{2+\alpha}}}\leq \|\theta_{0}\|_{L^{\frac{4}{2+\alpha}}}\leq C$$
due to $\theta_{0}\in L^{\frac{4}{2+\alpha}}(\mathbb{R}^{2})$.
In order to ensure $\kappa\in (0,1)$, keeping in mind the arbitrariness of $\epsilon>0$, we have
\begin{align}
\beta>\frac{4-6\alpha}{\alpha}.\nonumber
\end{align}
Concerning \eqref{asfbmu12} and \eqref{addsry897}, when $\beta>\frac{4-6\alpha}{\alpha}$, we see that
\begin{align}\label{fgdfdnmk67}
A_{1}
\leq
\frac{1}{16}\|\Lambda^{\frac{\alpha}{2}}G\|_{L^{2}}^{2}+\frac{1}{16}
\|\Lambda^{\frac{(4+\alpha)\beta}{4}}\theta\|_{L^{2}}^{2}
+C(1+\|\Lambda^{\frac{\alpha}{2}}u\|_{L^2}^{2}+\|\Lambda^{\frac{\beta}{2}
}\theta\|_{L^{2}}^{2})(1+\|G\|_{L^{2}}^{2}).
\end{align}
Via direct computations, we infer that
\begin{equation*}
\left\{\aligned
&\frac{4-6\alpha}{\alpha}\geq \frac{2(4-5\alpha)}{4+\alpha},\quad \mbox{if}\ \ \ 0<\alpha\leq\frac{7-\sqrt{33}}{2}\approx 0.6277,\\
& 1-\alpha>\frac{2(4-5\alpha)}{4+\alpha},\quad \ \ \mbox{if}\ \ \  \frac{7-\sqrt{33}}{2}<\alpha\leq\frac{2}{3},
\endaligned\right.
\end{equation*}
which yields
$$\max\left\{1-\alpha,\ \frac{4-6\alpha}{\alpha}\right\}\geq \frac{2(4-5\alpha)}{4+\alpha}.$$
This further implies
$$\beta>\frac{2(4-5\alpha)}{4+\alpha}.$$
Therefore, it allows us to take $\delta$ as
$$\delta=\frac{(4+\alpha)\beta+2(5\alpha-4)}{12}>0,$$
then by means of \eqref{uk3rtr03} and \eqref{dfexbg6}, we achieve
\begin{align}\label{asfbmu13}
A_{2}&\leq C\|[\mathcal{R}_{\alpha},u_{\theta}\cdot\nabla]\theta
\|_{\dot{B}_{2,2}^{-\frac{\alpha}{2}}}\|G\|_{\dot{B}_{2,2}^{\frac{\alpha}{2}}}
\nonumber\\&\leq C \|\Lambda^{1-\frac{\alpha}{4}} u_{\theta}\|_{L^{4}}\|\theta\|_{
{\dot{B}}_{4,2}^{1-\frac{5\alpha}{4}}} \|\Lambda^{\frac{\alpha}{2}}G\|_{L^{2}}
\nonumber\\&\leq C \|\Lambda^{1-\frac{5\alpha}{4}} \theta\|_{L^{4}}\|\theta\|_{
{\dot{B}}_{4,2}^{1-\frac{5\alpha}{4}}}
\|\Lambda^{\frac{\alpha}{2}}G\|_{L^{2}}
\nonumber\\&\leq C (\|\theta\|_{L^{2}}+\|\Lambda^{1-\frac{5\alpha}{4}+\delta} \theta\|_{L^{4}})^{2} \|\Lambda^{\frac{\alpha}{2}}G\|_{L^{2}}
\nonumber\\&\leq C\left(1+\|\Lambda^{1-\frac{5\alpha}{4}
+\delta}\theta\|_{\dot{B}_{\infty,\infty}^{-(1-\frac{5\alpha}{4}
+\delta)}}
^{\frac{1}{2}}
\|\Lambda^{1-\frac{5\alpha}{4}
+\delta}\theta\|_{\dot{B}_{2,2}^{\frac{(4+\alpha)\beta}{4}-(1-\frac{5\alpha}{4}
+2\delta)}}
^{\frac{1}{2}}
\right)^{2}
\|\Lambda^{\frac{\alpha}{2}}G\|_{L^{2}}
\nonumber\\&\leq C \left(1+\|\theta\|_{L^{\infty}}
\|\Lambda^{\frac{(4+\alpha)\beta}{4}-\delta
}\theta\|_{L^{2}}
\right)\|\Lambda^{\frac{\alpha}{2}}G\|_{L^{2}}\nonumber\\&\leq C \left(1+\|\theta\|_{L^{\infty}}\| \theta\|_{L^{2}}^{\frac{\delta}{{(4+\alpha)\beta}}}
\|\Lambda^{\frac{(4+\alpha)\beta}{4}
}\theta\|_{L^{2}}^{\frac{{(4+\alpha)\beta}-\delta}{{(4+\alpha)\beta}}}
\right)\|\Lambda^{\frac{\alpha}{2}}G\|_{L^{2}}\nonumber\\&\leq C \left(1+
\|\Lambda^{\frac{(4+\alpha)\beta}{4}
}\theta\|_{L^{2}}^{\frac{{(4+\alpha)\beta}-\delta}{{(4+\alpha)\beta}}}
\right)\|\Lambda^{\frac{\alpha}{2}}G\|_{L^{2}}\nonumber\\
&\leq
\frac{1}{16}\|\Lambda^{\frac{\alpha}{2}}G\|_{L^{2}}^{2}+\frac{1}{16}
\|\Lambda^{\frac{(4+\alpha)\beta}{4}}\theta\|_{L^{2}}^{2}+C.
\end{align}
Due to $\beta>\frac{4-6\alpha}{\alpha}$, it follows from the interpolation that
\begin{align}\label{asfbdfsfd46}
A_{3}&\leq C\|\Lambda^{1+\beta-\frac{3\alpha}{2}}\theta\|_{L^{2}}\|\Lambda^{\frac{\alpha}{2}}G\|_{L^{2}}
\nonumber\\&\leq   C\| \theta\|_{L^{2}} ^{1-\frac{2(2+2\beta-3\alpha)}{(4+\alpha)\beta}} \|\Lambda^{\frac{(4+\alpha)\beta}{4}}\theta\|_{L^{2}}^{\frac{2(2+2\beta-3\alpha)}{(4+\alpha)\beta}}
\|\Lambda^{\frac{\alpha}{2}}G\|_{L^{2}}
\nonumber\\&\leq
\frac{1}{16}\|\Lambda^{\frac{\alpha}{2}}G\|_{L^{2}}^{2}+\frac{1}{16}
\|\Lambda^{\frac{(4+\alpha)\beta}{4}}\theta\|_{L^{2}}^{2}+C.
\end{align}
We divide $A_{4}$ into two parts, namely
\begin{align}
A_{4}&=-  \int_{\mathbb{R}^{2}}
[\Lambda^{\frac{(2+\alpha)\beta}{4}}, u_{G} \cdot \nabla]\theta\,\,\Lambda^{\frac{(2+\alpha)\beta}{4}}\theta\,dx-   \int_{\mathbb{R}^{2}}
[\Lambda^{\frac{(2+\alpha)\beta}{4}}, u_{\theta} \cdot \nabla]\theta\,\,\Lambda^{\frac{(2+\alpha)\beta}{4}}\theta\,dx\nonumber\\
&\triangleq A_{41}+A_{42}.\nonumber
\end{align}
Taking advantage of \eqref{bmnhgf2} and \eqref{sdfghp28}, the term $A_{41}$ can be bounded by
\begin{align}\label{advbxs11}
A_{41}&\leq C\|[\Lambda^{\frac{(2+\alpha)\beta}{4}}, u_{G}\cdot \nabla]\theta\|_{L^{\frac{2(4+\alpha)}{6+\alpha}}}
\|\Lambda^{\frac{(2+\alpha)\beta}{4}}
\theta\|_{L^{\frac{2(4+\alpha)}{2+\alpha}}}
\nonumber\\
&\leq C\|\nabla u_{G}\|_{L^{\frac{4+\alpha}{2}}}
\|\Lambda^{\frac{(2+\alpha)\beta}{4}}\theta\|_{L^{\frac{2(4+\alpha)}
{2+\alpha}}}^{2}
\nonumber\\
&\leq
\|\Lambda^{\frac{4+2\alpha}{4+\alpha}}u_{G}\|_{L^2}\|\theta\|_{L^{\infty}}^{\frac{4}{4+\alpha}}
\|\Lambda^{\frac{(4+\alpha)\beta}{4}}\theta\|_{L^{2}}
^{\frac{4+2\alpha}{4+\alpha}}
\nonumber\\
&\leq C
\|\Lambda^{\frac{2+\alpha}{4}}u_{G}\|_{L^2}^{\frac{2\alpha}{4+\alpha}}
\|\Lambda^{1+\frac{\alpha}{2}}u_{G}\|_{L^2}^{\frac{4-\alpha}{4+\alpha}}
\|\Lambda^{\frac{(4+\alpha)\beta}{4}}\theta\|_{L^{2}}
^{\frac{4+2\alpha}{4+\alpha}}\nonumber\\
&\leq C \left(\|\Lambda^{\frac{\alpha}{2}}u_{G}\|_{L^2}^{\frac{1}{2}}\|\Lambda u_{G}\|_{L^2}^{\frac{1}{2}}\right)^{\frac{2\alpha}{4+\alpha}}
\|\Lambda^{1+\frac{\alpha}{2}}u_{G}\|_{L^2}^{\frac{4-\alpha}{4+\alpha}}
\|\Lambda^{\frac{(4+\alpha)\beta}{4}}\theta\|_{L^{2}}
^{\frac{4+2\alpha}{4+\alpha}}
 \nonumber\\
&\leq C  \|\Lambda^{\frac{\alpha}{2}-1}G\|_{L^2}^{\frac{\alpha}{4+\alpha}}\|G\|_{L^2}^{\frac{\alpha}{4+\alpha}}
\|\Lambda^{\frac{\alpha}{2}}G\|_{L^2}^{\frac{4-\alpha}{4+\alpha}}
\|\Lambda^{\frac{(4+\alpha)\beta}{4}}\theta\|_{L^{2}}
^{\frac{4+2\alpha}{4+\alpha}}
 \nonumber\\
&\leq C  \left(\|\Lambda^{\frac{\alpha}{2}}u\|_{L^2}+
\|\theta\|_{L^{\frac{4}{2+\alpha}}}\right)^{\frac{\alpha}{4+\alpha}}\|G\|_{L^2}^{\frac{\alpha}{4+\alpha}}
\|\Lambda^{\frac{\alpha}{2}}G\|_{L^2}^{\frac{4-\alpha}{4+\alpha}}
\|\Lambda^{\frac{(4+\alpha)\beta}{4}}\theta\|_{L^{2}}
^{\frac{4+2\alpha}{4+\alpha}}\nonumber\\&\leq
\frac{1}{16}\|\Lambda^{\frac{\alpha}{2}}G\|_{L^{2}}^{2}+\frac{1}{16}
\|\Lambda^{\frac{(4+\alpha)\beta}{4}}\theta\|_{L^{2}}^{2}+C \left(1+\|\Lambda^{\frac{\alpha}{2}}u\|_{L^2}^{2}\right)\|G\|_{L^2}^{2}.
\end{align}
Denoting
$$\gamma=\frac{(2+\alpha)\beta+2(1-\alpha)}{4},$$
and invoking \eqref{tvcbmp6} as well as \eqref{dfexbg6}, one has
\begin{align}\label{advbxs12}
A_{42}&\leq C\|[\Lambda^{\frac{(2+\alpha)\beta}{4}}, u_{\theta}\cdot \nabla]\theta\|_{L^{\frac{4\gamma}{2\gamma+1-\alpha}}}
\|\Lambda^{\frac{(2+\alpha)\beta}{4}}
\theta\|_{L^{\frac{4\gamma}{2\gamma+\alpha-1}}}
\nonumber\\
&\leq C\|\nabla u_{\theta}\|_{L^{\frac{2\gamma}{1-\alpha}}}\|\Lambda^{\frac{(2+\alpha)\beta}{4}}
\theta\|_{L^{\frac{4\gamma}{2\gamma+\alpha-1}}}^{2}
\nonumber\\
&\leq C\|\Lambda^{1-\alpha}\theta\|_{L^{\frac{2\gamma}{1-\alpha}}}
\|\Lambda^{\frac{(2+\alpha)\beta}{4}}
\theta\|_{L^{\frac{4\gamma}{2\gamma+\alpha-1}}}^{2}
\nonumber\\
&\leq C
\|\Lambda^{1-\alpha}\theta\|_{\dot{B}_{\infty,\infty}^{-(1-\alpha)}}^{1-\frac{1-\alpha}{\gamma}}
\|\Lambda^{1-\alpha}\theta\|_{\dot{B}_{2,2}^{\gamma-(1-\alpha)}}^{\frac{1-\alpha}{\gamma}}
\|\Lambda^{\frac{(2+\alpha)\beta}{4}}\theta\|_{\dot{B}_{\infty,\infty}
^{-\frac{(2+\alpha)\beta}{4}}}^{\frac{1-\alpha}{\gamma}}
\|\Lambda^{\frac{(2+\alpha)\beta}{4}}\theta\|_{\dot{B}_{2,2}^{\gamma-
\frac{(2+\alpha)\beta}{4}}}
^{2-\frac{1-\alpha}{\gamma}}
\nonumber\\
&\leq C\|\theta\|_{L^{\infty}}\|\Lambda^{\gamma}\theta\|_{L^{2}}^{2}
\nonumber\\
&\leq C\|\theta\|_{L^{\infty}}\|\theta\|_{L^{2}}^{2-\frac{8\gamma}{(4+\alpha)\beta}}
\|\Lambda^{\frac{(4+\alpha)\beta}{4}}\theta\|_{L^{2}}
^{\frac{8\gamma}{(4+\alpha)\beta}}
\nonumber\\
&\leq
\frac{1}{16}
\|\Lambda^{\frac{(4+\alpha)\beta}{4}}\theta\|_{L^{2}}^{2}+C,
\end{align}
where we have used $\beta>1-\alpha$.
It follows from \eqref{advbxs11} and \eqref{advbxs12} that
\begin{align}\label{cvghu81}
A_{4}
\leq
\frac{1}{16}\|\Lambda^{\frac{\alpha}{2}}G\|_{L^{2}}^{2}+\frac{1}{8}
\|\Lambda^{\frac{(4+\alpha)\beta}{4}}\theta\|_{L^{2}}^{2}+C \left(1+\|\Lambda^{\frac{\alpha}{2}}u\|_{L^2}^{2}\right)(1+\|G\|_{L^2}^{2}).
\end{align}
Plugging \eqref{fgdfdnmk67},  \eqref{asfbmu13}, \eqref{asfbdfsfd46} and \eqref{cvghu81} into \eqref{vbmk445}, it yields
\begin{align}
& \frac{d}{dt}\left(\|G(t)\|_{L^{2}}^{2}+ \|\Lambda^{\frac{(2+\alpha)\beta}{4}}\theta(t)\|_{L^{2}}^{2}\right)
+ \|\Lambda^{\frac{\alpha}{2}}G \|_{L^{2}}^{2}+\|\Lambda^{\frac{(4+\alpha)\beta}{4}}\theta\|_{L^{2}}^{2}
 \nonumber\\ & \leq C(1+\|\Lambda^{\frac{\alpha}{2}}u\|_{L^2}^{2}+\|\Lambda^{\frac{\beta}{2}
}\theta\|_{L^{2}}^{2})(1+\|G\|_{L^{2}}^{2}),\nonumber
\end{align}
which along with the Gronwall inequality implies \eqref{sfdsjkp2}. Consequently, this ends the proof of the lemma.
\end{proof}

\vskip .1in
Inspired by \cite[Proposition 3.3]{SWXY} or the above Lemma \ref{addftlem2}, we are able to carry out the following iterative procedure involving the combined quantity $G$.
\begin{lemma}\label{Las42}
Let $\beta>\max\{1-\alpha,\ \frac{4-6\alpha}{\alpha}\}$ with $\alpha\leq\frac{2}{3}$.
If it holds
\begin{eqnarray}\label{dfklp95}
\|G(t)\|_{L^{m_{k}}}^{m_{k}}+
\int_{0}^{t}{\|G(\tau)\|_{L^{\frac{2m_{k}}{2-\alpha}}}^{m_{k}}\,d\tau}\leq
M,
\end{eqnarray}
then we have
\begin{eqnarray}\label{sss2}
\|G(t)\|_{L^{m_{k+1}}}^{m_{k+1}}+
\int_{0}^{t}{\|G(\tau)\|_{L^{\frac{2m_{k+1}}{2-\alpha}}}^{m_{k+1}}\,d\tau}\leq
C(t,\,M,\,u_{0},\,\theta_{0}),
\end{eqnarray}
where
\begin{eqnarray}\label{serty8}
m_{k+1}=\frac{8\beta m_{k}}{2(2\beta+2-3\alpha)+(2-\alpha)\beta m_{k}}.
\end{eqnarray}
Furthermore, we may restrict $m_{k}$ and $m_{k+1}$ to the range
\begin{eqnarray}
2\leq m_{k},\,m_{k+1}<\min\left\{ \frac{2(2+\beta)}{(2+\alpha)\beta},\,\,\frac{2(2\beta+3\alpha-2)}{(2-\alpha)\beta}
\right\}.\nonumber
\end{eqnarray}
\end{lemma}

\begin{proof}
We first point out that $\alpha>\frac{2}{3}$ plays a key role in deriving \cite[Proposition 3.3]{SWXY}. Consequently, some new ideas and estimates are strongly required to overcome the difficulty caused by $\alpha\leq\frac{2}{3}$. Notice that in this case, it follows from \eqref{dfklp95} that (see (3.11) of \cite{SWXY})
\begin{eqnarray}\label{dfjkpl99}
\|\Lambda^{\frac{\beta[(2+\alpha)m_{k}-2]}{4}}\theta(t)\|_{L^{2}}^{2}+
\int_{0}^{t}{\|\Lambda^{\frac{(2+\alpha)\beta m_{k}}{4}}\theta(\tau)\|_{L^{2}}^{2}\,d\tau}
\leq
C(t,\,M,\,u_{0},\,\theta_{0}).
\end{eqnarray}
After checking the proof of (3.11) of \cite{SWXY}, \eqref{dfjkpl99} is valid as long as \eqref{dfklp95} and $\alpha+\beta>1$ hold true.
Taking the inner product of (\ref{t305}) with
$|G|^{m-2}G$ and using the fact
\begin{eqnarray}
\int_{\mathbb{R}^{2}}
(\Lambda^{\alpha}G)|G|^{m-2}G\,dx\geq \widetilde{C}\|\Lambda^{\frac{\alpha}{2}}|G|^{\frac{m}{2}}\|_{L^{2}}^{2}\geq \widetilde{C}\|G\|_{L^{\frac{2m}{2-\alpha}}}^{m},\nonumber
\end{eqnarray}
we derive that
\begin{align}\label{sdhp856}
 \frac{1}{m}\frac{d}{dt}\|G(t)\|_{L^{m}}^{m}+\widetilde{C}
 \|G\|_{L^{\frac{2m}{2-\alpha}}}^{m} \leq&\int_{\mathbb{R}^{2}}
{\Lambda^{\beta-\alpha}\partial_{x_{1}}\theta\,\,|G|^{m-2}G\,dx}
\nonumber\\&+\int_{\mathbb{R}^{2}}
{[\mathcal {R}_{\alpha},\,u_{\theta}\cdot\nabla]\theta\,\,|G|^{m-2}G\,dx}\nonumber\\& +\int_{\mathbb{R}^{2}}
{[\mathcal {R}_{\alpha},\,u_{G}\cdot\nabla]\theta\,\,|G|^{m-2}G\,dx}\nonumber\\
 \triangleq& K_{1}+K_{2}+K_{3},
\end{align}
where $\widetilde{C}>0$ is an absolute constant.
We choose $r_{1}\in (2,\infty)$, $\lambda_{1}\in [0,1]$ and $m>2$ as follows
$$r_{1}=\frac{2(2+\alpha)\beta m_{k}}{(2+\alpha)\beta m_{k}-2(2\beta+2-3\alpha)},\quad
\lambda_{1}=\frac{2[(m-2)r_{1}-m]}{\alpha(m-2)r_{1}},$$
\begin{align}
\frac{2r_{1}}{r_{1}-1}\leq m \leq \frac{4r_{1}}{2r_{1}-(2+\alpha)} \equiv\frac{8\beta m_{k}}{2(2\beta+2-3\alpha)+(2-\alpha)\beta m_{k}},\nonumber
\end{align}
then we deduce from \eqref{sdsvb999} and \eqref{dfexbg6} that
\begin{align}
K_{1}&\leq C\|\Lambda^{1+\beta-\frac{3\alpha}{2}}\theta\|_{L^{\frac{2r_{1}}{r_{1}-2}}}
\|\Lambda^{\frac{\alpha}{2}}(|G|^{m-2}G)\|_{L^{\frac{2r_{1}}{r_{1}+2}}}\nonumber\\
&\leq C\|\Lambda^{1+\beta-\frac{3\alpha}{2}}\theta\|_{L^{\frac{2r_{1}}{r_{1}-2}}}
\|\Lambda^{\frac{\alpha}{2}}G\|_{L^{2}}\|G\|_{L^{(m-2)r_{1}}}^{m-2}
\nonumber\\
&\leq C
\|\Lambda^{1+\beta-\frac{3\alpha}{2}}\theta\|_{\dot{B}_{\infty,\infty}
^{-(1+\beta-\frac{3\alpha}{2})}}^{\frac{2}{r_{1}}}
\|\Lambda^{1+\beta-\frac{3\alpha}{2}}\theta\|_{\dot{B}_{2,2}^{\frac{(2+\alpha)\beta m_{k}}{4}-(1+\beta-\frac{3\alpha}{2})}}^{\frac{r_{1}-2}{r_{1}}}
\nonumber\\& \quad \times\|\Lambda^{\frac{\alpha}{2}}G\|_{L^{2}}
\|G\|_{L^{m}}^{(m-2)(1-\lambda_{1})}
\|G\|_{L^{\frac{2m}{2-\alpha}}}^{(m-2)\lambda_{1}}
\nonumber\\
&\leq C
\|\theta\|_{L^{\infty}}^{\frac{2}{r_{1}}}
\|\Lambda^{\frac{(2+\alpha)\beta m_{k}}{4}}\theta\|_{L^{2}}^{\frac{r_{1}-2}{r_{1}}}
\|\Lambda^{\frac{\alpha}{2}}G\|_{L^{2}}
\|G\|_{L^{m}}^{(m-2)(1-\lambda_{1})}
\|G\|_{L^{\frac{2m}{2-\alpha}}}^{(m-2)\lambda_{1}}
\nonumber\\
&\leq C
\|\Lambda^{\frac{(2+\alpha)\beta m_{k}}{4}}\theta\|_{L^{2}}^{\frac{r_{1}-2}{r_{1}}}
\|\Lambda^{\frac{\alpha}{2}}G\|_{L^{2}}
\|G\|_{L^{m}}^{(m-2)(1-\lambda_{1})}
\|G\|_{L^{\frac{2m}{2-\alpha}}}^{(m-2)\lambda_{1}}
\nonumber\\
&\leq
\frac{\widetilde{C}}{16}
 \|G\|_{L^{\frac{2m}{2-\alpha}}}^{m}+C\left(1+\|\Lambda^{\frac{(2+\alpha)\beta m_{k}}{4}}\theta\|_{L^{2}}^{2}+\|\Lambda^{\frac{\alpha}{2}}G\|_{L^{2}}^{2}\right)
 (1+\|G\|_{L^{m}}^{m}).\nonumber
\end{align}
For the case $\frac{2r_{1}+2}{r_{1}}\leq m\leq \frac{2r_{1}}{r_{1}-1}$, it suffices to modify the proof as
$$\|G\|_{L^{(m-2)r_{1}}}^{m-2}\leq C (\|G\|_{L^{2}}+\|G\|_{L^{m}})^{m-2}\leq C(1+\|G\|_{L^{m}}^{m}).$$
Notice that the case $2<m<\frac{2r_{1}+2}{r_{1}}\equiv \frac{3(2+\alpha)\beta m_{k}-2(2\beta+2-3\alpha)}{(2+\alpha)\beta m_{k}}$ is actually more easier to handle. To this end, we take $\Pi_{1}>2$ as
$$\Pi_{1}=\frac{2[(2+\alpha)\beta m_{k}-2(2\beta+2-3\alpha)+4\delta_{1}]}{(2+\alpha)\beta m_{k}-2(2\beta+2-3\alpha)},$$
where
$$ \frac{2\beta+2-3\alpha}{2}\leq \delta_{1} \leq  \min\left\{ \frac{(2+\alpha)\beta m_{k}+2(2\beta+2-3\alpha)}{8},\quad \frac{2\beta+4-3\alpha}{2}\right\}.$$
We therefore obtain
\begin{align}
K_{1}&\leq C\|\Lambda^{1+\beta-\frac{3\alpha}{2}}\theta\|_{L^{\frac{2\Pi_{1}}{\Pi_{1}-2}}}
\|\Lambda^{\frac{\alpha}{2}}(|G|^{m-2}G)\|_{L^{\frac{2\Pi_{1}}{\Pi_{1}+2}}}\nonumber\\
&\leq C\|\Lambda^{1+\beta-\frac{3\alpha}{2}}\theta\|_{L^{\frac{2\Pi_{1}}{\Pi_{1}-2}}}
\|\Lambda^{\frac{\alpha}{2}}G\|_{L^{2}}\|G\|_{L^{(m-2)\Pi_{1}}}^{m-2}
\nonumber\\
&\leq C
\|\Lambda^{1+\beta-\frac{3\alpha}{2}}\theta\|_{\dot{B}_{\infty,\infty}
^{-\delta_{1}}}^{\frac{2}{\Pi_{1}}}
\|\Lambda^{1+\beta-\frac{3\alpha}{2}}\theta\|_{\dot{B}_{2,2}^{\frac{(2+\alpha)
\beta m_{k}}{4}-(1+\beta-\frac{3\alpha}{2})}}^{\frac{\Pi_{1}-2}{\Pi_{1}}}
 \|\Lambda^{\frac{\alpha}{2}}G\|_{L^{2}}(\|G\|_{L^{2}}+\|G\|_{L^{m}})^{m-2}
\nonumber\\
&\leq C
(\|\theta\|_{L^{2}}+\|\theta\|_{L^{\infty}})^{\frac{2}{\Pi_{1}}}
\|\Lambda^{\frac{(2+\alpha)\beta m_{k}}{4}}\theta\|_{L^{2}}^{\frac{\Pi_{1}-2}{\Pi_{1}}}
\|\Lambda^{\frac{\alpha}{2}}G\|_{L^{2}}
(1+\|G\|_{L^{m}}^{m})
\nonumber\\
&\leq C\left(1+\|\Lambda^{\frac{(2+\alpha)\beta m_{k}}{4}}\theta\|_{L^{2}}^{2}+
\|\Lambda^{\frac{\alpha}{2}}G\|_{L^{2}}^{2}\right)(1+\|G\|_{L^{m}}^{m}).\nonumber
\end{align}
As a result, we deduce
\begin{align}\label{sdhpfg457}
K_{1} \leq
\frac{\widetilde{C}}{16}
 \|G\|_{L^{\frac{2m}{2-\alpha}}}^{m}+C\left(1+\|\Lambda^{\frac{(2+\alpha)\beta m_{k}}{4}}\theta\|_{L^{2}}^{2}+\|\Lambda^{\frac{\alpha}{2}}G\|_{L^{2}}^{2}\right)
 (1+\|G\|_{L^{m}}^{m}),
\end{align}
where $m$ satisfies
\begin{align}\label{sndde11}
2< m \leq  \frac{8\beta m_{k}}{2(2\beta+2-3\alpha)+(2-\alpha)\beta m_{k}}.
\end{align}
To bound $K_{2}$, we fix $r_{2}\in (1,2)$, $\lambda_{2}\in [0,1]$ and $m>2$ as follows
$$r_{2}=\frac{(2+\alpha)\beta m_{k}}{5\alpha-4-4\epsilon+(2+\alpha)\beta m_{k}},\qquad
\lambda_{2}=\frac{(3m-4)r_{2}-2m}{\alpha(m-2)r_{2}},
$$
\begin{align}
\frac{4r_{2}}{3r_{2}-2}\leq m \leq \frac{8r_{2}}{(6+\alpha)r_{2}-2(2+\alpha)} \equiv\frac{8\beta m_{k}}{8-10\alpha+8\epsilon+(2-\alpha)\beta m_{k}}\nonumber
\end{align}
for any sufficiently small $\epsilon>0$, then it follows from \eqref{uk3rtr03} and \eqref{sdsvb999} that
\begin{align}\label{sdhpfg45811}
K_{2}&\leq C\|\Lambda^{-\frac{\alpha}{2}}([\mathcal{R}_{\alpha},u_{\theta}\cdot\nabla]\theta)
\|_{L^{\frac{r_{2}}{r_{2}-1}}}\|\Lambda^{\frac{\alpha}{2}}(|G|^{m-2}G)\|_{L^{r_{2}}}
\nonumber\\&\leq
C\|[\mathcal{R}_{\alpha},u_{\theta}\cdot\nabla]\theta
\|_{\dot{B}_{\frac{r_{2}}{r_{2}-1},1}^{-\frac{\alpha}{2}}}
\|\Lambda^{\frac{\alpha}{2}}(|G|^{m-2}G)\|_{L^{r_{2}}}
\nonumber\\&\leq C \|\Lambda^{1-\frac{\alpha}{4}+\epsilon} u_{\theta}\|_{L^{\frac{2r_{2}}{r_{2}-1}}}\|\theta\|_{
{\dot{B}}_{\frac{2r_{2}}{r_{2}-1},1}^{1-\frac{5\alpha}{4}-\epsilon}} \|\Lambda^{\frac{\alpha}{2}}G\|_{L^{2}}\|G\|_{L^{\frac{2(m-2)r_{2}}{2-r_{2}}}}^{m-2}
\nonumber\\&\leq C \|\Lambda^{1-\frac{5\alpha}{4}+\epsilon} \theta\|_{L^{\frac{2r_{2}}{r_{2}-1}}}\|\theta\|_{
{{B}}_{\frac{2r_{2}}{r_{2}-1},1}^{1-\frac{5\alpha}{4}-\epsilon}}
\|\Lambda^{\frac{\alpha}{2}}G\|_{L^{2}}\|G\|_{L^{\frac{2(m-2)r_{2}}{2-r_{2}}}}^{m-2}
\nonumber\\&\leq C \|\Lambda^{1-\frac{5\alpha}{4}+\epsilon} \theta\|_{L^{\frac{2r_{2}}{r_{2}-1}}}(\|\theta\|_{L^{2}}+\|\Lambda^{1-\frac{5\alpha}{4}+\epsilon} \theta\|_{L^{\frac{2r_{2}}{r_{2}-1}}})
\|\Lambda^{\frac{\alpha}{2}}G\|_{L^{2}}\|G\|_{L^{\frac{2(m-2)r_{2}}{2-r_{2}}}}^{m-2}
\nonumber\\&\leq C (1+\|\Lambda^{1-\frac{5\alpha}{4}+\epsilon} \theta\|_{L^{\frac{2r_{2}}{r_{2}-1}}}^{2})
\|\Lambda^{\frac{\alpha}{2}}G\|_{L^{2}}\|G\|_{L^{\frac{2(m-2)r_{2}}{2-r_{2}}}}^{m-2}
\nonumber\\&\leq C\left(1+\|\Lambda^{1-\frac{5\alpha}{4}
+\epsilon}\theta\|_{\dot{B}_{\infty,\infty}^{-(1-\frac{5\alpha}{4}
+\epsilon)}}
^{\frac{2}{r_{2}}}
\|\Lambda^{1-\frac{5\alpha}{4}
+\epsilon}\theta\|_{\dot{B}_{2,2}^{\frac{(2+\alpha)\beta m_{k}}{4}-(1-\frac{5\alpha}{4}
+\epsilon)}}
^{\frac{2(r_{2}-1)}{r_{2}}}
\right)
\|\Lambda^{\frac{\alpha}{2}}G\|_{L^{2}}
\nonumber\\&\quad \times
\|G\|_{L^{m}}^{(m-2)(1-\lambda_{2})}
\|G\|_{L^{\frac{2m}{2-\alpha}}}^{(m-2)\lambda_{2}}
\nonumber\\&\leq C\left(1+\|\theta\|_{L^{\infty}}
^{\frac{2}{r_{2}}}
\|\Lambda^{\frac{(2+\alpha)\beta m_{k}}{4}}\theta\|_{L^{2}}
^{\frac{2(r_{2}-1)}{r_{2}}}
\right)
\|\Lambda^{\frac{\alpha}{2}}G\|_{L^{2}}
\|G\|_{L^{m}}^{(m-2)(1-\lambda_{2})}
\|G\|_{L^{\frac{2m}{2-\alpha}}}^{(m-2)\lambda_{2}}
\nonumber\\&\leq C\left(1+
\|\Lambda^{\frac{(2+\alpha)\beta m_{k}}{4}}\theta\|_{L^{2}}
^{\frac{2(r_{2}-1)}{r_{2}}}
\right)
\|\Lambda^{\frac{\alpha}{2}}G\|_{L^{2}}
\|G\|_{L^{m}}^{(m-2)(1-\lambda_{2})}
\|G\|_{L^{\frac{2m}{2-\alpha}}}^{(m-2)\lambda_{2}}
\nonumber\\
&\leq
\frac{\widetilde{C}}{16}
 \|G\|_{L^{\frac{2m}{2-\alpha}}}^{m}+C\left(1+\|\Lambda^{\frac{(2+\alpha)\beta m_{k}}{4}}\theta\|_{L^{2}}^{2}+\|\Lambda^{\frac{\alpha}{2}}G\|_{L^{2}}^{2}\right)
 (1+\|G\|_{L^{m}}^{m}).
\end{align}
For the case $\frac{r_{2}+2}{r_{2}}\leq m\leq \frac{4r_{2}}{3r_{2}-2}$, it suffices to modify one part of \eqref{sdhpfg45811} as
$$\|G\|_{L^{\frac{2(m-2)r_{2}}{2-r_{2}}}}^{m-2}\leq C (\|G\|_{L^{2}}+\|G\|_{L^{m}})^{m-2}\leq C(1+\|G\|_{L^{m}}^{m}),$$
which implies the validness of \eqref{sdhpfg45811} in this case. For the case $2<m<\frac{r_{2}+2}{r_{2}}$, we appeal to the following variant version of \eqref{tvcbmp6}, whose proof is the same one as for \eqref{tvcbmp6}
\begin{eqnarray}\label{sdfghq66}
\|[\mathcal{R}_{\alpha},f\cdot\nabla]g\|_{L^{q}}\leq C\|\nabla f\|_{L^{q_{1}}}\|\Lambda^{1-\alpha}g\|_{L^{q_{2}}},\quad \frac{1}{q}=\frac{1}{q_{1}}+\frac{1}{q_{2}}.
\end{eqnarray}
In terms of the case $2<m<\frac{r_{2}+2}{r_{2}}$, we take $\Pi_{2}>1$ as
$$\Pi_{2}=\frac{(2+\alpha)\beta m_{k}+4(\delta_{2}+\alpha-1)}{(2+\alpha)\beta m_{k}+4(\alpha-1)},$$
where $\Lambda_{1}\leq \delta_{2} \leq \Lambda_{2}$ given by
$$\Lambda_{1}=\max\left\{\frac{(4-5\alpha)[(2+\alpha)\beta m_{k}+4\alpha-4]}{4[(2+\alpha)\beta m_{k}+5\alpha-4]},\quad 1-\alpha\right\},$$
$$\Lambda_{2}= \min\left\{\frac{(2+\alpha)\beta m_{k}[(2+\alpha)\beta m_{k}+4\alpha-4]}{8[(2+\alpha)\beta m_{k}+5\alpha-4]},\quad 2-\alpha\right\}.$$
Thus we deduce from \eqref{sdfghq66} that
\begin{align}\label{sdhpfasdg98}
K_{2}&\leq C\|[\mathcal{R}_{\alpha},u_{\theta}\cdot\nabla]\theta
\|_{L^{\frac{\Pi_{2}}{\Pi_{2}-1}}}\||G|^{m-2}G\|_{L^{\Pi_{2}}}
\nonumber\\&\leq
C\|\nabla u_{\theta}
\|_{L^{\frac{2\Pi_{2}}{\Pi_{2}-1}}}\|\Lambda^{1-\alpha}\theta
\|_{L^{\frac{2\Pi_{2}}{\Pi_{2}-1}}}
\|G\|_{L^{\Pi_{2}(m-1)}}^{m-1}
\nonumber\\&\leq C\|\Lambda^{1-\alpha}\theta
\|_{L^{\frac{2\Pi_{2}}{\Pi_{2}-1}}}^{2}
\|G\|_{L^{\Pi_{2}(m-1)}}^{m-1}
\nonumber\\&\leq C
\|\Lambda^{1-\alpha}\theta\|_{\dot{B}_{\infty,\infty}^{-\delta_{2}}}
^{\frac{2}{\Pi_{2}}}
\|\Lambda^{1-\alpha}\theta\|_{\dot{B}_{2,2}^{\frac{(2+\alpha)\beta m_{k}}{4}-(1-\alpha)}}
^{\frac{2(\Pi_{2}-1)}{\Pi_{2}}}(\|G\|_{L^{2}}^{m-1}+\|G\|_{L^{m}}^{m-1})\nonumber\\&\leq C
(\|\theta\|_{L^{2}}+\|\theta\|_{L^{\infty}})
^{\frac{2}{\Pi_{2}}}
\|\Lambda^{\frac{(2+\alpha)\beta m_{k}}{4}}\theta\|_{L^{2}}
^{\frac{2(\Pi_{2}-1)}{\Pi_{2}}}(\|G\|_{L^{2}}^{m-1}+\|G\|_{L^{m}}^{m-1})\nonumber\\&\leq C
(1+\|\Lambda^{\frac{(2+\alpha)\beta m_{k}}{4}}\theta\|_{L^{2}}
^{2})(1+\|G\|_{L^{m}}^{m}).
\end{align}
Therefore, we get from \eqref{sdhpfg45811} and \eqref{sdhpfasdg98} that
\begin{align}\label{sdhpfg458}
K_{2}
\leq
\frac{\widetilde{C}}{16}
 \|G\|_{L^{\frac{2m}{2-\alpha}}}^{m}+C\left(1+\|\Lambda^{\frac{(2+\alpha)\beta m_{k}}{4}}\theta\|_{L^{2}}^{2}+\|\Lambda^{\frac{\alpha}{2}}G\|_{L^{2}}^{2}\right)
 (1+\|G\|_{L^{m}}^{m}),
\end{align}
where $m$ should be satisfied
\begin{align}\label{sndde12}
2< m \leq  \frac{8\beta m_{k}}{8-10\alpha+8\epsilon+(2-\alpha)\beta m_{k}}
\end{align}
To deal with $K_{3}$, we take $r_{3}\in (2,\infty)$, $\lambda_{3}\in [0,1]$ and $m>2$ as follows
$$r_{3}=\frac{(2+\alpha)\beta m_{k}}{2-3\alpha+2\epsilon},\quad
\lambda_{3}=\frac{(m-2)r_{3}+2m}{\alpha(m-1)r_{3}},\qquad\qquad\qquad$$
\begin{align}\label{sndde13}
m\leq\frac{4r_{3}}{(2-\alpha)r_{3}+2(2+\alpha)}\equiv\frac{4\beta m_{k}}{4-6\alpha+4\epsilon+(2-\alpha)\beta m_{k}}
\end{align}
for any sufficiently small $\epsilon>0$. By virtue of \eqref{uk3rtr03} and \eqref{sdsvb999} again, we obtain
\begin{align}\label{sdhpfg459}
K_{3}&\leq C\|\Lambda^{-\frac{\alpha}{2}}([\mathcal{R}_{\alpha},u_{G}\cdot\nabla]\theta)
\|_{L^{\frac{2(m-1)r_{3}}{r_{3}+2(m-2)}}}
\|\Lambda^{\frac{\alpha}{2}}(|G|^{m-2}G)\|_{L^{\frac{2(m-1)r_{3}}{(2m-3)r_{3}-2(m-2)}}}
\nonumber\\&\leq
C\|[\mathcal{R}_{\alpha},u_{G}\cdot\nabla]\theta
\|_{\dot{B}_{\frac{2(m-1)r_{3}}{r_{3}+2(m-2)},1}^{-\frac{\alpha}{2}}}
\|\Lambda^{\frac{\alpha}{2}}(|G|^{m-2}G)\|_{L^{\frac{2(m-1)r_{3}}{(2m-3)r_{3}-2(m-2)}}}
\nonumber\\&\leq C \|\Lambda u_{G}\|_{L^{\frac{2(m-1)r_{3}}{r_{3}-2}}}\|\theta\|_{
{\dot{B}}_{r_{3},1}^{1-\frac{3\alpha}{2}}} \|\Lambda^{\frac{\alpha}{2}}G\|_{L^{2}}\|G\|_{L^{\frac{2(m-1)r_{3}}{r_{3}-2}}}^{m-2}
\nonumber\\&\leq C  \left(\|\theta\|_{L^{2}}+\|\Lambda^{1-\frac{3\alpha}{2}+\epsilon}\theta\|_{L^{r_{3}}}\right) \|\Lambda^{\frac{\alpha}{2}}G\|_{L^{2}}\|G\|_{L^{\frac{2(m-1)r_{3}}{r_{3}-2}}}^{m-1}
\nonumber\\&\leq C\left(1+\|\Lambda^{1-\frac{3\alpha}{2}
+\epsilon}\theta\|_{\dot{B}_{\infty,\infty}^{-(1-\frac{3\alpha}{2}
+\epsilon)}}
^{1-\frac{2}{r_{3}}}
\|\Lambda^{1-\frac{3\alpha}{2}
+\epsilon}\theta\|_{\dot{B}_{2,2}^{\frac{(2+\alpha)\beta m_{k}}{4}-(1-\frac{3\alpha}{2}
+\epsilon)}}
^{\frac{2}{r_{3}}}
\right)
\|\Lambda^{\frac{\alpha}{2}}G\|_{L^{2}}
\nonumber\\&\quad \times
\|G\|_{L^{m}}^{(m-1)(1-\lambda_{3})}
\|G\|_{L^{\frac{2m}{2-\alpha}}}^{(m-1)\lambda_{3}}
\nonumber\\&\leq C\left(1+\|\theta\|_{L^{\infty}}
^{1-\frac{2}{r_{3}}}
\|\Lambda^{\frac{(2+\alpha)\beta m_{k}}{4}}\theta\|_{L^{2}}
^{\frac{2}{r_{3}}}
\right)\|\Lambda^{\frac{\alpha}{2}}G\|_{L^{2}}\|G\|_{L^{m}}^{(m-1)(1-\lambda_{3})}
\|G\|_{L^{\frac{2m}{2-\alpha}}}^{(m-1)\lambda_{3}}
\nonumber\\&\leq C\left(1+
\|\Lambda^{\frac{(2+\alpha)\beta m_{k}}{4}}\theta\|_{L^{2}}
^{\frac{2}{r_{3}}}
\right)\|\Lambda^{\frac{\alpha}{2}}G\|_{L^{2}}\|G\|_{L^{m}}^{(m-1)(1-\lambda_{3})}
\|G\|_{L^{\frac{2m}{2-\alpha}}}^{(m-1)\lambda_{3}}
\nonumber\\
&\leq
\frac{\widetilde{C}}{16}
 \|G\|_{L^{\frac{2m}{2-\alpha}}}^{m}+C\left(1+\|\Lambda^{\frac{(2+\alpha)\beta m_{k}}{4}}\theta\|_{L^{2}}^{2}+\|\Lambda^{\frac{\alpha}{2}}G\|_{L^{2}}^{2}\right)
 (1+\|G\|_{L^{m}}^{m}).
\end{align}
Inserting \eqref{sdhpfg457}, \eqref{sdhpfg458} and \eqref{sdhpfg459} into \eqref{sdhp856}, this enables us to obtain
\begin{align}\label{sdhp8510}
 \frac{d}{dt}\|G(t)\|_{L^{m}}^{m}+\widetilde{C}
 \|G\|_{L^{\frac{2m}{2-\alpha}}}^{m} \leq C\left(1+\|\Lambda^{\frac{(2+\alpha)\beta m_{k}}{4}}\theta\|_{L^{2}}^{2}+\|\Lambda^{\frac{\alpha}{2}}G\|_{L^{2}}^{2}\right)
 (1+\|G\|_{L^{m}}^{m}),
\end{align}
where $m$ should satisfy \eqref{sndde11}, \eqref{sndde12} and \eqref{sndde13}.
Note that the arbitrariness of $\epsilon>0$, we have that
\begin{align}
2< m\leq\frac{8\beta m_{k}}{2(2\beta+2-3\alpha)+(2-\alpha)\beta m_{k}},\nonumber
\end{align}
\begin{align}
2<m< \min\left\{\frac{8\beta m_{k}}{8-10\alpha+(2-\alpha)\beta m_{k}},\quad \frac{4\beta m_{k}}{4-6\alpha+(2-\alpha)\beta m_{k}} \right\}.\nonumber
\end{align}
Due to $\beta>1-\alpha$ and $m_{k}<\frac{2(2\beta+3\alpha-2)}{(2-\alpha)\beta}$, it is not hard to check
$$\frac{8\beta m_{k}}{8-10\alpha+(2-\alpha)\beta m_{k}}>\frac{8\beta m_{k}}{2(2\beta+2-3\alpha)+(2-\alpha)\beta m_{k}},$$
$$\frac{4\beta m_{k}}{4-6\alpha+(2-\alpha)\beta m_{k}}>\frac{8\beta m_{k}}{2(2\beta+2-3\alpha)+(2-\alpha)\beta m_{k}}.$$
As a result, $m$ should satisfy
\begin{align}
2< m\leq\frac{8\beta m_{k}}{2(2\beta+2-3\alpha)+(2-\alpha)\beta m_{k}}.\nonumber
\end{align}
We thus deduce from \eqref{sdhp8510} that
\begin{align}
\frac{d}{dt}\|G(t)\|_{L^{m_{k+1}}}^{m_{k+1}}+
\|G\|_{L^{\frac{2m_{k+1}}{2-\alpha}}}^{m_{k+1}}
\leq& C\left(1+\|\Lambda^{\frac{(2+\alpha)\beta m_{k}}{4}}\theta\|_{L^{2}}^{2}+\|\Lambda^{\frac{\alpha}{2}}G\|_{L^{2}}^{2}\right)
(1+\|G\|_{L^{m_{k+1}}}^{m_{k+1}}),\nonumber
\end{align}
where
$$m_{k+1}=\frac{8\beta m_{k}}{2(2\beta+2-3\alpha)+(2-\alpha)\beta m_{k}}.$$
The desired \eqref{sss2} follows from \eqref{dfjkpl99} and the Gronwall inequality immediately.
\end{proof}

\vskip .1in
Lemma \ref{Las42} allows us to show the following estimate.
\begin{lemma}
If $\beta>\max\{1-\alpha,\ \frac{4-6\alpha }{\alpha}\}$ with $\alpha\leq\frac{2}{3}$,
then it holds
\begin{eqnarray}\label{cvhlp6}
\|G(t)\|_{L^{m}}^{m}+
\int_{0}^{t}{\|G(\tau)\|_{L^{\frac{2m}{2-\alpha}}}^{m}\,d\tau}\leq
C(t,\,u_{0},\,\theta_{0}),
\end{eqnarray}
where $m$ satisfies
\begin{eqnarray}\label{dfdsgkp6}
2\leq m <\min\left\{ \frac{2(2+\beta)}{(2+\alpha)\beta},\,\,\frac{2(2\beta+3\alpha-2)}{(2-\alpha)\beta}
\right\}.
\end{eqnarray}
\end{lemma}

\begin{proof}
It follows from \eqref{serty8} that
\begin{align}
m_{k}=\frac{2(2\beta+3\alpha-2)}{(2-\alpha)\beta
+(\alpha\beta+3\alpha-2)\Big(\frac{2\beta+2-3\alpha}
{4\beta}\Big)^{k-1}},\quad k\geq1.\nonumber
\end{align}
Due to $\beta>\frac{4-6\alpha}{\alpha}\geq\frac{2-3\alpha}{\alpha}$ for $\alpha\leq \frac{2}{3}$, the sequence $\{m_{k}\}_{k\in \mathbb{N}}$ is increasing and it has the limit
$$\lim_{k\rightarrow\infty}m_{k}=\frac{2(2\beta+3\alpha-2)}
{(2-\alpha)\beta}.$$
It implies that $m$ should be satisfied
\begin{eqnarray}\label{asdfvbv25}
2\leq m<\frac{2(2\beta+3\alpha-2)}{(2-\alpha)\beta}.
\end{eqnarray}
Keeping in mind the following restriction
$$m_{k}<\frac{2(2+\beta)}{(2+\alpha)\beta},\quad k=1,2,\cdot\cdot\cdot,$$
we further choose $m$ satisfying
\begin{eqnarray}\label{asdfvbv26}
 m<\frac{2(2+\beta)}{(2+\alpha)\beta}.
\end{eqnarray}
Due to \eqref{asdfvbv25} and \eqref{asdfvbv26}, we thus derive that $m$ obeys  \eqref{dfdsgkp6}.
Consequently, the desired bound \eqref{cvhlp6} holds.
\end{proof}

\vskip .1in
Naturally, one may appeal to \cite[Proposition 3.5]{SWXY} to obtain more higher regularity estimate of the combined quantity $G$, which strongly requires the key condition
$$\sup_{0\leq t\leq T}\|G(t)\|_{L^{q}}<\infty, \qquad q>\frac{2}{\alpha}.$$
However, when $\alpha\leq\frac{2}{3}$, it holds
$$\min\left\{ \frac{2(2+\beta)}{(2+\alpha)\beta},\,\,\frac{2(2\beta+3\alpha-2)}{(2-\alpha)\beta}
\right\}\leq\frac{2}{\alpha}, $$
which yields the invalidness of \cite[Proposition 3.5]{SWXY} for this case. To overcome this difficulty caused by $\alpha\leq\frac{2}{3}$, we again appeal to the lower bounds for the fractional Laplacian. We also provide two methods.

\vskip .1in

\begin{center}
\textbf{Method 1}
\end{center}
\vskip .1in

We now prove the following $L^{\infty}$-estimate of both $\omega$ and the gradient of $\theta$.
\begin{lemma}
If $\beta>\max\{\frac{4-6\alpha}{\alpha},\ \frac{2-\alpha-2\alpha^{2}}{2\alpha}\}$ with $\alpha\in (\sqrt{13}-3,\frac{2}{3}]$,
then it holds
\begin{eqnarray}\label{tyukle12}
\|\omega(t)\|_{L^{\infty}}+\|\nabla\theta (t)\|_{L^{\infty}} \leq
C(t,\,u_{0},\,\theta_{0}).
\end{eqnarray}
\end{lemma}

\begin{proof}
Before proving this lemma, it is worthwhile to mention that it suffices to consider the case $\alpha>\beta>\max\{\frac{4-6\alpha}{\alpha},\ \frac{2-\alpha-2\alpha^{2}}{2\alpha}\}$
as the case $\beta\geq\alpha$ belongs to Part 1 due to the simple fact $\alpha>\frac{4-2\alpha}{4+\alpha}$ with $\alpha\in (\sqrt{13}-3,\frac{2}{3}]$.
Arguing as for proving \eqref{fgup98}, we get by applying \eqref{abbdghw569} to \eqref{t305} that
\begin{align}
\frac{d}{dt}\|G(t)\|_{L^{\infty}}+c\frac{
\|G(t)\|_{L^{\infty}}^{1+\frac{\alpha m}{2}}}{\|G\|_{L^{m}}^{\frac{\alpha m}{2} }}
&\leq  C\|[\mathcal {R}_{\alpha},\,u\cdot\nabla]\theta\|_{L^{\infty}}+C
\|\Lambda^{\beta-\alpha}\partial_{x_{1}}\theta\|_{L^{\infty}},\nonumber
\end{align}
which along with \eqref{xcdksdr88} yields
\begin{align}\label{fasaw68}
\frac{d}{dt}\|G(t)\|_{L^{\infty}}+c\frac{
\|G(t)\|_{L^{\infty}}^{1+\frac{\alpha m}{2}}}{\|G\|_{L^{m}}^{\frac{\alpha m}{2} }}
&\leq  C\|[\mathcal {R}_{\alpha},\,u\cdot\nabla]\theta\|_{\dot{B}_{\infty,1}^{0}}+C
\|\Lambda^{\beta-\alpha}\partial_{x_{1}}\theta\|_{\dot{B}_{\infty,1}^{0}}\nonumber\\
&\leq  C\|\nabla u\|_{L^{\infty}}\|\theta\|_{\dot{B}_{\infty,1}^{1-\alpha}}+C
\|\theta\|_{\dot{B}_{\infty,1}^{1+\beta-\alpha}}\nonumber\\
&\leq  C\|\nabla u\|_{L^{\infty}}\|\theta\|_{\dot{B}_{\infty,\infty}^{0}}^{\alpha}
\|\theta\|_{\dot{B}_{\infty,\infty}^{1}}^{1-\alpha}+C
\|\theta\|_{\dot{B}_{\infty,\infty}^{0}}^{\alpha-\beta}
\|\theta\|_{\dot{B}_{\infty,\infty}^{1}}^{1+\beta-\alpha}\nonumber\\
&\leq  C\|\nabla u\|_{L^{\infty}}\|\theta\|_{L^{\infty}}^{\alpha}
\|\nabla\theta\|_{L^{\infty}}^{1-\alpha}+C
\|\theta\|_{L^{\infty}}^{\alpha-\beta}
\|\nabla\theta\|_{L^{\infty}}^{1+\beta-\alpha}\nonumber\\
&\leq  C\|\nabla u\|_{L^{\infty}}
\|\nabla\theta\|_{L^{\infty}}^{1-\alpha}+C
\|\nabla\theta\|_{L^{\infty}}^{1+\beta-\alpha}.
\end{align}
In light of \eqref{cvhlp6}, we derive from \eqref{fasaw68} that
 \begin{align}\label{bghkp68}
\frac{d}{dt}\|G(t)\|_{L^{\infty}}+c
\|G(t)\|_{L^{\infty}}^{1+\frac{\alpha m}{2}}\leq
C\|\nabla u\|_{L^{\infty}}
\|\nabla\theta\|_{L^{\infty}}^{1-\alpha}+C
\|\nabla\theta\|_{L^{\infty}}^{1+\beta-\alpha},
\end{align}
where $m$ satisfies \eqref{dfdsgkp6}.
Recalling \eqref{tNew011}, we get
 \begin{eqnarray}\label{bghkp69}
\frac{d}{dt}\|\nabla\theta(t)\|_{L^{\infty}}+c
\|\nabla\theta(t)\|_{L^{\infty}}^{1+\beta}
\leq C\|\nabla u\|_{L^{\infty}}\|\nabla\theta(t)\|_{L^{\infty}}.
\end{eqnarray}
Direct computation yields
\begin{align}\label{bdffhp98}
\|\omega\|_{L^{\infty}}&\leq
 \|G\|_{L^{\infty}}+\|\mathcal {R}_{\alpha}\theta\|_{L^{\infty}}
\nonumber\\
&\leq
C(\|G\|_{L^{\infty}}+\|\mathcal {R}_{\alpha}\theta\|_{\dot{B}_{\infty,1}^{0}})
\nonumber\\
&\leq
C(\|G\|_{L^{\infty}}+\|\theta\|_{\dot{B}_{\infty,1}^{1-\alpha}})
\nonumber\\
&\leq
C(\|G\|_{L^{\infty}}+\|\theta\|_{L^{\infty}}^{\alpha}
\|\nabla\theta\|_{L^{\infty}}^{1-\alpha})
\nonumber\\
&\leq
C(\|G\|_{L^{\infty}}+
\|\nabla\theta\|_{L^{\infty}}^{1-\alpha}).
\end{align}
Applying \eqref{tNew013} to \eqref{bghkp68} and noting \eqref{bdffhp98}, it follows
\begin{align}\label{bghkp70}
 \frac{d}{dt}\|G(t)\|_{L^{\infty}}+c
\|G(t)\|_{L^{\infty}}^{1+\frac{\alpha m}{2}}\leq&
C(1+\|\omega\|_{L^{\infty}})
\|\nabla\theta\|_{L^{\infty}}^{1-\alpha}D(t)+C
\|\nabla\theta\|_{L^{\infty}}^{1+\beta-\alpha}\nonumber\\
\leq&
C\|G\|_{L^{\infty}}\|\nabla\theta\|_{L^{\infty}}^{1-\alpha}D(t)\nonumber\\
&+C
(\|\nabla\theta\|_{L^{\infty}}^{1-\alpha}+\|\nabla\theta\|_{L^{\infty}}^{2-2\alpha})D(t)+C
\|\nabla\theta\|_{L^{\infty}}^{1+\beta-\alpha}.
\end{align}
where the increasing function $D(t)$ is given by
$$D(t)=C\ln
\Big(e+\int_{0}^{t}{\big(1+\|G(\tau)\|_{L^{\infty}}+\|\nabla\theta(\tau)
\|_{L^{\infty}}
\big)^{\Gamma}\,d\tau}\Big).$$
As a result, we get from \eqref{bghkp70} that
\begin{align}\label{bghkp71}
 \frac{d}{dt}\|G(t)\|_{L^{\infty}}+c
\|G(t)\|_{L^{\infty}}^{1+\frac{\alpha m}{2}}\leq&
C\|G(t)\|_{L^{\infty}}\|\nabla\theta(t)\|_{L^{\infty}}^{1-\alpha}D(t)+C
\|\nabla\theta(t)\|_{L^{\infty}}^{1+\beta-\alpha}\nonumber\\&+C
(\|\nabla\theta(t)\|_{L^{\infty}}^{1-\alpha}
+\|\nabla\theta(t)\|_{L^{\infty}}^{2-2\alpha})D(t).
\end{align}
Similar arguments may be employed to \eqref{bghkp69} to derive
\begin{align}\label{bghkp72}
\frac{d}{dt}\|\nabla\theta(t)\|_{L^{\infty}}+c
\|\nabla\theta(t)\|_{L^{\infty}}^{1+\beta}
\leq & C\|G(t)\|_{L^{\infty}}\|\nabla\theta(t)\|_{L^{\infty}}D(t) \nonumber\\&+C
(\|\nabla\theta(t)\|_{L^{\infty}} +\|\nabla\theta(t)\|_{L^{\infty}}^{2-\alpha})D(t).
\end{align}
We divide the proof into two cases, namely
\begin{align}
\mbox{Case 1}: \frac{c}{2}\|\nabla\theta(t)\|_{L^{\infty}}^{\beta}\geq C\|G(t)\|_{L^{\infty}}D(t),\qquad
\mbox{Case 2}: \frac{c}{2}\|\nabla\theta(t)\|_{L^{\infty}}^{\beta}< C\|G(t)\|_{L^{\infty}}D(t).\nonumber
 \end{align}
Let us begin with $\mbox{Case 1}$. On the one hand, if it holds
$$\frac{c}{2}\|\nabla\theta(t)\|_{L^{\infty}}^{1+\beta}\geq C
(\|\nabla\theta(t)\|_{L^{\infty}}+\|\nabla\theta(t)\|_{L^{\infty}}^{2-\alpha})D(t),$$
then \eqref{bghkp72} reduces to
$$\frac{d}{dt}\|\nabla\theta(t)\|_{L^{\infty}}\leq 0,$$
which implies
$$\|\nabla\theta(t)\|_{L^{\infty}}\leq C.$$
Of course, it allows us to show
\begin{eqnarray}\label{KEY2}
\|\omega(t)\|_{L^{\infty}}\leq \|\omega_{0}\|_{L^{\infty}}+\int_{0}^{t}{\|\nabla \theta(\tau)\|_{L^{\infty}}\,d \tau}\leq C,
\end{eqnarray}
which yields
\begin{align}
\|\omega(t)\|_{L^{\infty}}+\|\nabla\theta (t)\|_{L^{\infty}} \leq
C.\nonumber
\end{align}
On the other hand, if it holds
$$\frac{c}{2}\|\nabla\theta(t)\|_{L^{\infty}}^{1+\beta}<C
(\|\nabla\theta(t)\|_{L^{\infty}}+\|\nabla\theta(t)\|_{L^{\infty}}^{2-\alpha})D(t),$$
then we discover that
\begin{align} \label{dfuyt6983}
\frac{c}{2}\|\nabla\theta(t)\|_{L^{\infty}}^{1+\beta}<C
(1+\|\nabla\theta(t)\|_{L^{\infty}}^{2-\alpha})D(t).
\end{align}
Moreover, regarding $\mbox{Case 1}$, we find
\begin{align} \label{dfuyt6984}
D(t)&\leq C\ln
\Big(e+\int_{0}^{t}{\big(1+\|\nabla\theta(\tau)\|_{L^{\infty}}^{\beta}+\|\nabla\theta(\tau)
\|_{L^{\infty}}
\big)^{\Gamma}\,d\tau}\Big)\nonumber\\ &\leq C\ln
\Big(e+\int_{0}^{t}{\big(1+\|\nabla\theta(\tau)
\|_{L^{\infty}}
\big)^{\Gamma}\,d\tau}\Big).
\end{align}
Denoting
$$\Theta(t)=\sup_{0\leq \tau\leq t}\|\nabla\theta(\tau)
\|_{L^{\infty}},$$
and noting the increasing property of $D(t)$, we thus deduce from \eqref{dfuyt6983} that
\begin{align}
\frac{c}{2}\Theta(t)^{1+\beta}\leq C\left(1+\Theta(t)^{2-\alpha}\right)
D(t),\nonumber
\end{align}
which along with \eqref{dfuyt6984} ensures
\begin{align} \label{dfuyt6986}
\frac{c}{2}\Theta(t)^{1+\beta}&\leq C\left(1+\Theta(t)^{2-\alpha}\right)
\ln
\Big(e+\int_{0}^{t}{\big(1+\|\nabla\theta(\tau)
\|_{L^{\infty}}
\big)^{\Gamma}\,d\tau}\Big)\nonumber\\
&\leq C\left(1+\Theta(t)^{2-\alpha}\right)
\ln
\Big(e+\int_{0}^{t}{\big(1+\Theta(\tau)
\big)^{\Gamma}\,d\tau}\Big)\nonumber\\
&\leq C
\ln
\Big(e+\int_{0}^{t}{\big(1+\Theta(t)
\big)^{\Gamma}\,d\tau}\Big)\nonumber\\
&\leq C\left(1+\Theta(t)^{2-\alpha}\right)
\ln
\big(e+\Theta(t)
\big),
\end{align}
where we used the increasing property of $\Theta(t)$.
It thus follows from \eqref{dfuyt6986} that
\begin{align} \label{sdkp1q}
\frac{c}{2}\Theta(t)^{1+\beta}\leq C\left(1+\Theta(t)^{2-\alpha}\right)
\ln
\big(e+\Theta(t)
\big).
\end{align}
Note that $1+\beta>2-\alpha$ due to $\beta>1-\alpha$, we infer from \eqref{sdkp1q} that
$$\Theta(t)\leq C,$$
which implies
$$\|\nabla\theta(t)\|_{L^{\infty}}\leq C.$$
This along with \eqref{KEY2} gives \eqref{tyukle12}.
Let us now move to $\mbox{Case 2}$. We divide this case into two sub-cases
\begin{align}
\mbox{Case 21}: D(t)\leq C\|G\|_{L^{\infty}}^{\epsilon},\qquad
\mbox{Case 22}:D(t)> C\|G\|_{L^{\infty}}^{\epsilon},\nonumber
 \end{align}
where $\epsilon=\epsilon(\alpha,\beta)>0$ is suitable small. For $\mbox{Case 21}$, we deduce from \eqref{bghkp71} that
\begin{align} \label{bghkp73}
 \frac{d}{dt}\|G(t)\|_{L^{\infty}}+c
\|G(t)\|_{L^{\infty}}^{1+\frac{\alpha m}{2}}\leq&
C\|G\|_{L^{\infty}}^{1+\frac{1-\alpha}{\beta}+\frac{\epsilon}{\beta}+\epsilon}
+C\|G\|_{L^{\infty}}^{\frac{2-2\alpha+\epsilon}{\beta}+\epsilon}\nonumber\\&+C
\|G\|_{L^{\infty}}^{\frac{(1+\epsilon)(1+\beta-\alpha)}{\beta}}+C
\|G\|_{L^{\infty}}^{\frac{1-\alpha+(1+\beta-\alpha)\epsilon}{\beta}}\nonumber\\
\leq& C+
C\|G\|_{L^{\infty}}^{1+\frac{1-\alpha}{\beta}+\frac{\epsilon}{\beta}+\epsilon}
+C\|G\|_{L^{\infty}}^{\frac{2-2\alpha+\epsilon}{\beta}+\epsilon}\nonumber\\&+C
\|G\|_{L^{\infty}}^{\frac{(1+\epsilon)(1+\beta-\alpha)}{\beta}}.
\end{align}
If it holds
\begin{align} \label{bghkp74}
m>\frac{2}{\alpha}\max\left\{\frac{1-\alpha}{\beta},\ \ \frac{2-2\alpha-\beta}{\beta}\right\}\equiv \frac{2(1-\alpha)}{\alpha\beta},
\end{align}
then by taking $\epsilon>0$ small enough, we derive from \eqref{bghkp73} that
$$\frac{d}{dt}\|G(t)\|_{L^{\infty}}+\frac{c}{2}
\|G(t)\|_{L^{\infty}}^{1+\frac{\alpha m}{2}}\leq C,$$
which gives
\begin{align}\label{bghkdf8}
\|G(t)\|_{L^{\infty}} \leq C.
\end{align}
Moreover, we have
\begin{align}\label{bghkdf9}
D(t) \leq C\|G\|_{L^{\infty}}^{\epsilon}\leq C.
\end{align}
Concerning \eqref{bghkdf8} and \eqref{bghkdf9}, we derive from $\mbox{Case 2}$ that
\begin{align}
\frac{c}{2}\|\nabla\theta(t)\|_{L^{\infty}}^{\beta}< C\|G(t)\|_{L^{\infty}}D(t)\leq C,\nonumber
\end{align}
which yields
\begin{align}\label{bgsd67}
\|\nabla\theta(t)\|_{L^{\infty}}\leq C.
\end{align}
Combining \eqref{bgsd67} and \eqref{KEY2} implies \eqref{tyukle12}.
For $\mbox{Case 22}$, we note that
$$\frac{c}{2}\|\nabla\theta(t)\|_{L^{\infty}}^{\beta}< C\|G\|_{L^{\infty}}D(t)< CD(t)^{1+\frac{1}{\epsilon}},$$
which shows
\begin{align}
D(t)&=C\ln
\Big(e+\int_{0}^{t}{\big(1+\|G(\tau)\|_{L^{\infty}}+\|\nabla\theta(\tau)
\|_{L^{\infty}}
\big)^{\Gamma}\,d\tau}\Big)\nonumber\\
&\leq C\ln
\Big(e+\int_{0}^{t}{\big(1+D(\tau)^{\frac{1}{\epsilon}}
+D(\tau)^{\frac{\epsilon+1}{\beta\epsilon}}
\big)^{\Gamma}\,d\tau}\Big)
\nonumber\\
&\leq C\ln
\Big(e+\int_{0}^{t}{\big(1+D(\tau)^{\frac{\epsilon+1}{\beta\epsilon}}
\big)^{\Gamma}\,d\tau}\Big)\nonumber\\
&\leq C\ln
\Big(e+\int_{0}^{t}{\big(1+D(t)^{\frac{\epsilon+1}{\beta\epsilon}}
\big)^{\Gamma}\,d\tau}\Big)\nonumber\\
&\leq C\ln
\big(e+ D(t) \big).\nonumber
\end{align}
As a result, one gets
$$
D(t)\leq C\ln
\big(e+ D(t) \big),
$$
which of course implies
$$D(t)\leq C.$$
Furthermore, it yields
$$\|G(t)\|_{L^{\infty}} \leq C.$$
Keeping in mind \eqref{bgsd67}, the desired estimate \eqref{tyukle12} is still valid for $\mbox{Case 22}$.
Consequently, \eqref{tyukle12} is valid for each case. Finally, let us check the workable of $m$. In fact, combining \eqref{dfdsgkp6} and \eqref{bghkp74}, it should be satisfied
\begin{eqnarray}\label{dgvzpkq11}
\max\left\{2,\ \ \frac{2(1-\alpha)}{\alpha\beta} \right\}< m <\min\left\{ \frac{2(2+\beta)}{(2+\alpha)\beta},\,\,\frac{2(2\beta+3\alpha-2)}{(2-\alpha)\beta}
\right\}.
\end{eqnarray}
Hence, the following condition ensures the workable of $m$
$$\beta>\max\left\{\frac{2-3\alpha}{\alpha},\ \ \frac{2-3\alpha-\alpha^{2}}{\alpha},\ \ \frac{2-\alpha-2\alpha^{2}}{2\alpha} \right\}\equiv\frac{2-\alpha-2\alpha^{2}}{2\alpha}.$$
Now collecting all the restrictions on $\beta$ leads to
$$\beta>\max\left\{1-\alpha,\ \frac{4-6\alpha }{\alpha},\ \frac{2-\alpha-2\alpha^{2}}{2\alpha}\right\}.$$
When $\alpha\in (\sqrt{13}-3,\frac{2}{3}]$, we infer
$$1-\alpha\leq \frac{2-\alpha-2\alpha^{2}}{2\alpha},\qquad \frac{4-6\alpha}{\alpha}< \frac{4-2\alpha}{4+\alpha},\qquad \frac{2-\alpha-2\alpha^{2}}{2\alpha}< \frac{4-2\alpha}{4+\alpha}.$$
Finally, $\beta$ should be satisfied
$$\beta>\max\left\{\frac{4-6\alpha }{\alpha},\ \frac{2-\alpha-2\alpha^{2}}{2\alpha}\right\}.$$
We thus complete the proof of the lemma.
\end{proof}

\vskip .1in
\begin{proof}[The proof of Part 2 of Theorem \ref{Th3}]
Keeping in mind the proof of Part 1 and \eqref{tyukle12}, we are able to finish the proof of Part 2 of Theorem \ref{Th3}.
\end{proof}

\vskip .1in

\begin{center}
\textbf{Method 2}
\end{center}
\vskip .1in
The proof here is also based on a variant of nonlinear lower bound for the fractional Laplacian and the direct energy method.
\begin{lemma}
If $\beta>\max\{\frac{4-6\alpha}{\alpha},\ \frac{2-\alpha-2\alpha^{2}}{2\alpha}\}$ with $\alpha\in (\sqrt{13}-3,\frac{2}{3}]$,
then it holds for sufficiently large $p\in[2,\infty)$
\begin{eqnarray}\label{yyadkb01}
\|\omega(t)\|_{L^{p}}+\|\nabla\theta (t)\|_{L^{\beta^{-}p}} \leq
C(t,\,u_{0},\,\theta_{0}).
\end{eqnarray}
\end{lemma}

\begin{proof}
Multiplying \eqref{t305} by $|G|^{p-2}G$ and making use of \eqref{tdfhb12}, we arrive at the pointwise estimate
\begin{align}\label{yyadkb03}
\frac{1}{p}\left\{\partial_{t}|G|^{p}+(u\cdot \nabla)|G|^{p}+\Lambda^{\alpha}|G|^{p}\right\}+c\frac{|G|^{p+\frac{\alpha m}{2}}}{\|G\|_{L^{m}}^{\frac{\alpha m}{2}}}\leq&  [\mathcal {R}_{\alpha},\,u\cdot\nabla]\theta|G|^{p-2}G
\nonumber\\&+\Lambda^{\beta-\alpha}\partial_{x_{1}}\theta|G|^{p-2}G.
\end{align}
Integrating \eqref{yyadkb03} over $\mathbb{R}^2$ and invoking \eqref{cvhlp6}, it yields
\begin{align} \label{yyadkb04}
\frac{d}{dt}\|G(t)\|_{L^{p}}^{p}+c
\|G\|_{L^{p+\frac{\alpha m}{2}}}^{p+\frac{\alpha m}{2}}
\leq &\underbrace{C\int_{\mathbb{R}^{2}}[\mathcal {R}_{\alpha},\,u\cdot\nabla]\theta|G|^{p-2}G\,dx}_{I_{1}}
\nonumber\\&
+\underbrace{C\int_{\mathbb{R}^{2}}\Lambda^{\beta-\alpha}\partial_{x_{1}}\theta|G|^{p-2}G\,dx}
_{I_{2}},
\end{align}
where $m$ satisfies \eqref{dfdsgkp6}.
Recalling \eqref{sdghg257}, we have
\begin{align} \label{yyadkb07}
\frac{d}{dt}\|\nabla\theta(t)\|_{L^{r}}^{r}+c
\|\nabla\theta\|_{L^{r+\beta}}^{r+\beta}
\leq &\underbrace{-C\int_{\mathbb{R}^{2}}(\nabla u \cdot \nabla) \theta|\nabla\theta|^{r-2}\nabla\theta\,dx}_{I_{3}}.
\end{align}
Putting \eqref{yyadkb04} and \eqref{yyadkb07} together yields
\begin{align} \label{yyadkb08}
\frac{d}{dt}(\|G(t)\|_{L^{p}}^{p}+\|\nabla\theta(t)\|_{L^{r}}^{r})+c
\|G\|_{L^{p+\frac{\alpha m}{2}}}^{p+\frac{\alpha m}{2}}+c
\|\nabla\theta\|_{L^{r+\beta}}^{r+\beta}
\leq I_{1}+I_{2}+I_{3}.
\end{align}
Due to $u=u_{G}+u_{\theta}$, we get
$$I_{3}=-C\int_{\mathbb{R}^{2}}(\nabla u_{G} \cdot \nabla) \theta|\nabla\theta|^{r-2}\nabla\theta\,dx-C\int_{\mathbb{R}^{2}}(\nabla u_{\theta} \cdot \nabla) \theta|\nabla\theta|^{r-2}\nabla\theta\,dx\triangleq I_{31}+I_{32}.$$
If we take $p$ and $r$ as
\begin{align} \label{yyadkb09}
p>\frac{r+\beta}{\beta}-\frac{\alpha m}{2},
\end{align}
then we obtain
\begin{align}
I_{31}&\leq C\|\nabla u_{G}\|_{L^{\frac{r+\beta}{\beta}}}\|\nabla\theta\|_{L^{r+\beta}}^{r}
\nonumber\\&\leq C\|G\|_{L^{\frac{r+\beta}{\beta}}}\|\nabla\theta\|_{L^{r+\beta}}^{r}\nonumber\\&\leq C\|G\|_{L^{2}}^{1-\vartheta}\|G\|_{L^{p+\frac{\alpha m}{2}}}^{\vartheta}\|\nabla\theta\|_{L^{r+\beta}}^{r}\nonumber\\&\leq C \|G\|_{L^{p+\frac{\alpha m}{2}}}^{\vartheta}\|\nabla\theta\|_{L^{r+\beta}}^{r}
\nonumber\\&\leq C+
\frac{c}{16}
\|G\|_{L^{p+\frac{\alpha m}{2}}}^{p+\frac{\alpha m}{2}}+\frac{c}{16}
\|\nabla\theta\|_{L^{r+\beta}}^{r+\beta} \nonumber
\end{align}
with $\vartheta\in (0,1)$.
According to $\beta>1-\alpha$ and \eqref{dfexbg6}, one may derive
\begin{align}
I_{32}&\leq C\|\nabla u_{\theta}\|_{L^{\frac{r+\beta}{\beta}}}\|\nabla\theta\|_{L^{r+\beta}}^{r}
\nonumber\\&\leq C\|\Lambda^{1-\alpha}\theta\|_{L^{\frac{r+\beta}{\beta}}}
\|\nabla\theta\|_{L^{r+\beta}}^{r}\nonumber\\&\leq C
\|\Lambda^{1-\alpha}\theta\|_{\dot{B}_{\infty,\infty}^{-(1-\alpha)}}^{1-\beta}
\|\Lambda^{1-\alpha}\theta\|_{\dot{B}_{r+\beta,r+\beta}^{\frac{(1-\alpha)
(1-\beta)}{\beta}}}^{\beta}
\|\nabla\theta\|_{L^{r+\beta}}^{r}\nonumber\\&\leq C
\|\theta\|_{L^{\infty}}^{1-\beta}
\|\theta\|_{\dot{B}_{r+\beta,r+\beta}^{\frac{1-\alpha
}{\beta}}}^{\beta}
\|\nabla\theta\|_{L^{r+\beta}}^{r}\nonumber\\&\leq C
\|\theta\|_{L^{\infty}}^{1-\beta}
\|\theta\|_{\dot{B}_{r+\beta,\infty}^{0}}^{\alpha+\beta-1}
\|\theta\|_{\dot{B}_{r+\beta,\infty}^{1}}^{1-\alpha}
\|\nabla\theta\|_{L^{r+\beta}}^{r}\nonumber\\&\leq C
\|\theta\|_{L^{\infty}}^{1-\beta}
\|\theta\|_{L^{r+\beta}}^{\alpha+\beta-1}
\|\nabla\theta\|_{L^{r+\beta}}^{r+1-\alpha}\nonumber\\&\leq C
\|\nabla\theta\|_{L^{r+\beta}}^{r+1-\alpha}\nonumber\\&\leq C +\frac{c}{16}
\|\nabla\theta\|_{L^{r+\beta}}^{r+\beta}. \nonumber
\end{align}
Therefore, one derives
\begin{align} \label{yyadkb10}
I_{3} \leq  C+
\frac{c}{8}
\|G\|_{L^{p+\frac{\alpha m}{2}}}^{p+\frac{\alpha m}{2}}+\frac{c}{8}
\|\nabla\theta\|_{L^{r+\beta}}^{r+\beta}.
\end{align}
Regarding $I_{2}$, we have
\begin{align}\label{svkpwev11}
I_{2}&\leq C\|\Lambda^{1+\beta-\alpha}\theta\|_{L^{\infty}}\|G\|_{L^{p-1}}^{p-1}
\nonumber\\&\leq C
\|\theta\|_{L^{\infty}}^{1-\mu_{22}}\|\nabla\theta\|_{L^{r+\beta}}^{\mu_{22}}
\|G\|_{L^{2}}^{(p-1)(1-\mu_{21})}\|G\|_{L^{p+\frac{\alpha m}{2}}}^{(p-1)\mu_{21}}
\nonumber\\&\leq C \|\nabla\theta\|_{L^{r+\beta}}^{\mu_{22}}
 \|G\|_{L^{p+\frac{\alpha m}{2}}}^{(p-1)\mu_{21}}
\nonumber\\&\leq C+
\frac{c}{8}
\|G\|_{L^{p+\frac{\alpha m}{2}}}^{p+\frac{\alpha m}{2}}+\frac{c}{8}
\|\nabla\theta\|_{L^{r+\beta}}^{r+\beta},
\end{align}
where $\mu_{21},\mu_{22}\in (0,1)$ are given by
$$\mu_{21}=\frac{(p-3)(2p+\alpha m)}{(p-1)(2p+\alpha m-4)},\quad \mu_{22}=\frac{(1+\beta-\alpha)(r+\beta)}{r+\beta-2}.$$
To guarantee \eqref{svkpwev11}, it requires
$$\frac{(p-1)\mu_{21}}{p+\frac{\alpha m}{2}}+\frac{\mu_{22}}{r+\beta}<1,$$
which is equal to
\begin{align} \label{svkpwev12}
(1+\beta-\alpha)(2p+\alpha m-4)<(r+\beta-2)(\alpha m+2).
\end{align}
Keeping in mind $u=u_{G}+u_{\theta}$, we see that
$$I_{1}=C\int_{\mathbb{R}^{2}}[\mathcal {R}_{\alpha},\,u_{\theta}\cdot\nabla]\theta|G|^{p-2}G\,dx+C\int_{\mathbb{R}^{2}}[\mathcal {R}_{\alpha},\,u_{G}\cdot\nabla]\theta|G|^{p-2}G\,dx\triangleq I_{11}+I_{12}.$$
In view of \eqref{sdfghq66} and \eqref{dfexbg6}, one derives
\begin{align}
I_{11}&\leq C\|[\mathcal {R}_{\alpha},\,u_{\theta}\cdot\nabla]\theta\|_{L^{\frac{r+\beta}{2(1-\alpha)}}}
\|G\|_{L^{\frac{(r+\beta)(p-1)}{3\alpha+\beta-2}}}^{p-1}
\nonumber\\&\leq C\|\Lambda^{1-\alpha}\theta\|_{L^{\frac{r+\beta}{1-\alpha}}}^{2}
\|G\|_{L^{\frac{(r+\beta)(p-1)}{r+2\alpha+\beta-2}}}^{p-1}\nonumber\\&\leq C
\|\Lambda^{1-\alpha}\theta\|_{\dot{B}_{\infty,\infty}^{-(1-\alpha)}}^{2\alpha}
\|\Lambda^{1-\alpha}\theta\|_{\dot{B}_{r+\beta,r+\beta}^{\alpha}}^{2(1-\alpha)}
\|G\|_{L^{p}}^{(p-1)(1-\mu_{11})}
\|G\|_{L^{p+\frac{\alpha m}{2}}}^{(p-1)\mu_{11}}
\nonumber\\&\leq C
\|\theta\|_{L^{\infty}}^{2\alpha}
\|\nabla\theta\|_{L^{r+\beta}}^{2(1-\alpha)}
\|G\|_{L^{p}}^{(p-1)(1-\mu_{11})}
\|G\|_{L^{p+\frac{\alpha m}{2}}}^{(p-1)\mu_{11}}\nonumber\\&\leq C
\|\nabla\theta\|_{L^{r+\beta}}^{2(1-\alpha)}
\|G\|_{L^{p}}^{(p-1)(1-\mu_{11})}
\|G\|_{L^{p+\frac{\alpha m}{2}}}^{(p-1)\mu_{11}}\nonumber\\&\leq
\frac{c}{16}
\|G\|_{L^{p+\frac{\alpha m}{2}}}^{p+\frac{\alpha m}{2}}+\frac{c}{16}
\|\nabla\theta\|_{L^{r+\beta}}^{r+\beta}+C\|G\|_{L^{p}}^{p}, \nonumber
\end{align}
where $\mu_{11}\in(0,1)$ is given by
$$\mu_{11}=\frac{[2(1-\alpha)p-r-\beta](2p+\alpha m)}{\alpha m(r+\beta)(p-1)}.$$
We point out that to ensure $\mu_{11}\in(0,1)$, it requires
\begin{align} \label{yyadkb11}
2(1-\alpha)p-r-\beta>0,\quad \alpha m (r+\beta)(p-1)>[2(1-\alpha)p-r-\beta](2p+\alpha m).
\end{align}
To deal with $I_{12}$, we take $\lambda_{12}>1$ and $\mu_{12}\in (0,1)$ satisfying
$$\lambda_{12}=\frac{r+\beta}{r+\alpha+\beta-1}, \qquad \frac{1}{\lambda_{12}p}=\frac{1-\mu_{12}}{p}
+\frac{\mu_{12}}{p+\frac{\alpha m}{2}},$$
then by means of \eqref{sdfghq66} and \eqref{dfexbg6}, we have
\begin{align}\label{svkpwev13}
I_{12}&\leq C\|[\mathcal {R}_{\alpha},\,u_{G}\cdot\nabla]\theta\|_{L^{\frac{\lambda_{12}p}{(\lambda_{12}-1)p+1}}}
\|G\|_{L^{\lambda_{12}p}}^{p-1}
\nonumber\\&\leq C\|\Lambda^{1-\alpha}\theta\|_{L^{\frac{\lambda_{12}}{\lambda_{12}-1}}}
\|G\|_{L^{\lambda_{12}p}}^{p}\nonumber\\&\leq C
\|\Lambda^{1-\alpha}\theta\|_{\dot{B}_{\infty,\infty}^{-(1-\alpha)}}^{\alpha}
\|\Lambda^{1-\alpha}\theta\|_{\dot{B}_{r+\beta,r+\beta}^{\alpha}}^{1-\alpha}
\|G\|_{L^{p}}^{p(1-\mu_{12})}
\|G\|_{L^{p+\frac{\alpha m}{2}}}^{p\mu_{12}}
\nonumber\\&\leq C
\|\theta\|_{L^{\infty}}^{\alpha}
\|\nabla\theta\|_{L^{r+\beta}}^{1-\alpha}
\|G\|_{L^{p}}^{p(1-\mu_{12})}
\|G\|_{L^{p+\frac{\alpha m}{2}}}^{p\mu_{12}}\nonumber\\&\leq C
\|\nabla\theta\|_{L^{r+\beta}}^{1-\alpha}
\|G\|_{L^{p}}^{p(1-\mu_{12})}
\|G\|_{L^{p+\frac{\alpha m}{2}}}^{p\mu_{12}}\nonumber\\&\leq
\frac{c}{16}
\|G\|_{L^{p+\frac{\alpha m}{2}}}^{p+\frac{\alpha m}{2}}+\frac{c}{16}
\|\nabla\theta\|_{L^{r+\beta}}^{r+\beta}+C\|G\|_{L^{p}}^{p}.
\end{align}
To ensure \eqref{svkpwev13}, one needs
$$\frac{p\mu_{12}}{p+\frac{\alpha m}{2}}+\frac{1-\alpha}{r+\beta}<1,$$
or equivalently
\begin{align} \label{svkpwev14}
2(1-\alpha)p<\alpha m(r+\alpha+\beta-1).
\end{align}
We thus get
\begin{align} \label{svkpwev15}
I_{1}\leq \frac{c}{8}
\|G\|_{L^{p+\frac{\alpha m}{2}}}^{p+\frac{\alpha m}{2}}+\frac{c}{8}
\|\nabla\theta\|_{L^{r+\beta}}^{r+\beta}+C\|G\|_{L^{p}}^{p}.
\end{align}
Inserting \eqref{yyadkb10}, \eqref{svkpwev11} and \eqref{svkpwev15} into \eqref{yyadkb08} gives
\begin{align}\label{svkpwev16}
\frac{d}{dt}(\|G(t)\|_{L^{p}}^{p}+\|\nabla\theta(t)\|_{L^{r}}^{r})
\leq C(1+\|G\|_{L^{p}}^{p}),
\end{align}
where $p$ and $r$ satisfy \eqref{yyadkb09}, \eqref{svkpwev12}, \eqref{yyadkb11} and \eqref{svkpwev14}.
Now we take
$$r=(\beta-\epsilon)p$$
for sufficiently small $\epsilon\in (0,\beta)$, then to ensure the validity of such $p$, the lower bound of $\beta$ obeys
\begin{eqnarray}
m>\frac{2(1-\alpha)}{\alpha \beta},\nonumber
\end{eqnarray}
which is consistent with \eqref{dgvzpkq11}.
Moreover, $p$ satisfies $p\geq p_{0}=p_{0}(\alpha,\beta,\epsilon)$.
Therefore, we derive from \eqref{svkpwev16} that
\begin{align}
\frac{d}{dt}(\|G(t)\|_{L^{p}}^{p}+\|\nabla\theta(t)\|_{L^{(\beta-\epsilon)p}}^{(\beta-\epsilon)p})
\leq C(1+\|G\|_{L^{p}}^{p}),\nonumber
\end{align}
which implies
\begin{eqnarray}
\|G(t)\|_{L^{p}}+\|\nabla\theta (t)\|_{L^{(\beta-\epsilon)p}} \leq
C(t,\,u_{0},\,\theta_{0}),\quad p\geq p_{0},\nonumber
\end{eqnarray}
or equivalently
\begin{eqnarray}
\|G(t)\|_{L^{p}}+\|\nabla\theta (t)\|_{L^{\beta^{-}p}} \leq
C(t,\,u_{0},\,\theta_{0}),\quad p\geq p_{0}.\nonumber
\end{eqnarray}
If we further take
$$p\geq\frac{2(1-\beta^{-})}{\alpha\beta^{-}},$$
then it holds
\begin{align}\label{ssgpuy18}
\|\omega\|_{L^{p}}&\leq \|G\|_{L^{p}}+C\|\Lambda^{1-\alpha}\theta\|_{L^{p}}
\nonumber\\&\leq \|G\|_{L^{p}}+C\|\theta\|_{L^{p}}+C\|\nabla\theta \|_{L^{\beta^{-}p}}\nonumber\\& \leq
C(t,\,u_{0},\,\theta_{0}).
\end{align}
Consequently, one has
\begin{eqnarray}
\|\omega(t)\|_{L^{p}} \leq
C(t,\,u_{0},\,\theta_{0}),\quad  \,p\geq\max\left\{p_{0},\,\, \frac{2(1-\beta^{-})}{\alpha\beta^{-}}\right\}.\nonumber
\end{eqnarray}
Therefore, we obtain the desired bound \eqref{yyadkb01}.
\end{proof}

\vskip .1in
\begin{proof}[The proof of Part 2 of Theorem \ref{Th3}]
Recalling \textbf{Method 2} of Part 1 and \eqref{yyadkb01}, we are able to finish the proof of Part 2 of Theorem \ref{Th3}.
\end{proof}

\vskip .2in
\appendix

\section{The proof of \eqref{tdfhb12}, \eqref{yyadkb02} and \eqref{tNew013}}  \label{appSec2}
In this appendix, we prove the inequalities \eqref{tdfhb12}, \eqref{yyadkb02} and \eqref{tNew013} which have been used in the proof of the main result.

\begin{proof}[The proof of \eqref{tdfhb12} and \eqref{yyadkb02}]
The proof is inspired by \cite{CV}. For $\alpha\in (0,2)$, we have
$$\Lambda^{\alpha}f(x)=C(\alpha)\,\,\mbox{p.v.}\int_{\mathbb{R}^{2}}{\frac{f(x)-f(x+y)}
{|y|^{2+\alpha}}\,dy},$$
which yields
\begin{align}\label{dfqxz11}
&|f(x)|^{r-2}f(x)\Lambda^{\alpha}f(x)\nonumber\\&=C(\alpha)\,\,\mbox{p.v.}
\int_{\mathbb{R}^{2}}{\frac{|f(x)|^{r}-|f(x)|^{r-2}f(x)f(x+y)}
{|y|^{2+\alpha}}\,dy}\nonumber\\&=\frac{C(\alpha)}{r}\,\,\mbox{p.v.}
\int_{\mathbb{R}^{2}}{\frac{r|f(x)|^{r}-r|f(x)|^{r-2}f(x)f(x+y)}
{|y|^{2+\alpha}}\,dy}\nonumber\\&=\frac{C(\alpha)}{r}\,\,\mbox{p.v.}
\int_{\mathbb{R}^{2}}{\frac{|f(x)|^{r}-|f(x+y)|^{r}}
{|y|^{2+\alpha}}\,dy}\nonumber\\&\quad+\frac{C(\alpha)}{r}\,\,\mbox{p.v.}
\int_{\mathbb{R}^{2}}{\frac{(r-1)|f(x)|^{r}-r|f(x)|^{r-2}f(x)f(x+y)+|f(x+y)|^{r}}
{|y|^{2+\alpha}}\,dy}
\nonumber\\&=
\frac{1}{r}\Lambda^{\alpha}(|f(x)|^{r})+\mathcal{A},
\end{align}
where $\mathcal{A}$ is given by
$$\mathcal{A}=\frac{C(\alpha)}{r}\,\,\mbox{p.v.}
\int_{\mathbb{R}^{2}}{\frac{(r-1)|f(x)|^{r}-r|f(x)|^{r-2}f(x)f(x+y)+|f(x+y)|^{r}}
{|y|^{2+\alpha}}\,dy}.$$
Notice that
\begin{align}
|f(x)|^{r-2}f(x)f(x+y)\leq |f(x)|^{r-1}|f(x+y)|
\leq \frac{r-1}{r}|f(x)|^{r}+\frac{1}{r}|f(x+y)|^{r},\nonumber
\end{align}
we thus have
\begin{align}\label{dfqxz12}
(r-1)|f(x)|^{r}-r|f(x)|^{r-2}f(x)f(x+y)+|f(x+y)|^{r}\geq0.
\end{align}
Let $\chi$ be a smooth radially cutoff function that vanishes on $|x|\leq 1$ and is identically $1$
for $|x|\geq 2$. Keeping in mind \eqref{dfqxz12}, we can conclude
\begin{align} \label{dfqxz13}
\mathcal{A}&\geq
\frac{C(\alpha)}{r}\int_{\mathbb{R}^{2}}{\frac{(r-1)|f(x)|^{r}-r|f(x)|^{r-2}f(x)f(x+y)+|f(x+y)|^{r}}
{|y|^{2+\alpha}}\chi(\frac{|y|}{R})\,dy}\nonumber\\
&\geq
\frac{C(\alpha)}{r}\int_{\mathbb{R}^{2}}{\frac{(r-1)|f(x)|^{r}-r|f(x)|^{r-2}f(x)f(x+y)}
{|y|^{2+\alpha}}\chi(\frac{|y|}{R})\,dy}
\nonumber\\
&=
C(\alpha)\frac{r-1}{r}\int_{\mathbb{R}^{2}}{\frac{|f(x)|^{r}}
{|y|^{2+\alpha}}\chi(\frac{|y|}{R})\,dy}
-C(\alpha)\int_{\mathbb{R}^{2}}{\frac{|f(x)|^{r-2}f(x)f(x+y)}
{|y|^{2+\alpha}}\chi(\frac{|y|}{R})\,dy}
\nonumber\\
&\geq
C(\alpha)\frac{r-1}{r}\int_{|y|\geq 2R}{\frac{|f(x)|^{r}}
{|y|^{2+\alpha}}\chi(\frac{|y|}{R})\,dy}
-C(\alpha)\int_{\mathbb{R}^{2}}{\frac{|f(x)|^{r-1}|f(x+y)|}
{|y|^{2+\alpha}}\chi(\frac{|y|}{R})\,dy}
\nonumber\\
&=
C(\alpha)\frac{r-1}{r}|f(x)|^{r}\int_{|y|\geq 2R}{\frac{1}
{|y|^{2+\alpha}} \,dy}
-C(\alpha)|f(x)|^{r-1}\int_{|y|\geq R}{\frac{|f(x+y)|}
{|y|^{2+\alpha}}\chi(\frac{|y|}{R})\,dy}
\nonumber\\
&\geq
c_{1}\frac{r-1}{r}|f(x)|^{r}R^{-\alpha}
-C(\alpha)|f(x)|^{r-1}\|f\|_{L^{p}}\left(\int_{|y|\geq R}{\frac{1}
{|y|^{\frac{(2+\alpha)p}{p-1}}}\,dy}\right)^{\frac{p-1}{p}}\nonumber\\
&\geq
c_{1}\frac{r-1}{r}|f(x)|^{r}R^{-\alpha}
-C_{2}|f(x)|^{r-1}\|f\|_{L^{p}}R^{-(\alpha+\frac{2}{p})},
\end{align}
where $c_{1}$ and $C_{2}$ depend only on $\alpha$ and $p$. Now taking
$$R=\left(\frac{2C_{2} r\|f\|_{L^{p}}}{c_{1}(r-1)|f(x)|}\right)^{\frac{p}{2}},$$
we derive from \eqref{dfqxz13} that
\begin{align}  \label{dfqxz14}
\mathcal{A}&\geq c\frac{|f(x)|^{r+\frac{\alpha p}{2}}}{\|f\|_{L^{p}}^{\frac{\alpha p}{2}}},
\end{align}
where
$$c=\frac{c_{1}^{1+\frac{\alpha p}{2}}}{2^{1+\frac{\alpha p}{2}}C_{2}^{\frac{\alpha p}{2}}}\Big(\frac{r-1}{r}\Big)^{1+\frac{\alpha p}{2}}.$$
Inserting \eqref{dfqxz14} into \eqref{dfqxz11} implies the desired inequality \eqref{tdfhb12} immediately.\\

Inspired the proof of \eqref{tdfhb12}, we are able to show \eqref{yyadkb02}.
In view of \eqref{dfqxz11}, one has
\begin{align}\label{asdwp1}
|\nabla f(x)|^{r-2}\nabla f(x)\Lambda^{\beta}\nabla f(x)=
\frac{1}{r}\Lambda^{\beta}(|\nabla f(x)|^{r})+\mathcal{B},
\end{align}
where $\mathcal{B}$ is given by
$$\mathcal{B}=\frac{C(\beta)}{r}\,\,\mbox{p.v.}
\int_{\mathbb{R}^{2}}{\frac{(r-1)|\nabla f(x)|^{r}-r|\nabla f(x)|^{r-2}\nabla f(x)\nabla f(x+y)+|\nabla f(x+y)|^{r}}
{|y|^{2+\beta }}\,dy}.$$
Arguing as \eqref{dfqxz12}, we obtain
\begin{align}
(r-1)|\nabla f(x)|^{r}-r|\nabla f(x)|^{r-2}\nabla f(x)\nabla f(x+y)+|\nabla f(x+y)|^{r}\geq0.\nonumber
\end{align}
Let $\chi$ be given above. It is not hard to see that
\begin{align} \label{asdwp2}
\mathcal{B}&\geq
\frac{C}{r}\int_{\mathbb{R}^{2}}{\frac{(r-1)|\nabla f(x)|^{r}-r|\nabla f(x)|^{r-2}\nabla f(x)\nabla f(x+y)+|\nabla f(x+y)|^{r}}
{|y|^{2+\beta }}\chi(\frac{|y|}{R})\,dy}\nonumber\\
&\geq
\frac{C}{r}\int_{\mathbb{R}^{2}}{\frac{(r-1)|\nabla f(x)|^{r}-r|\nabla f(x)|^{r-2}\nabla f(x)\nabla f(x+y)}
{|y|^{2+\beta }}\chi(\frac{|y|}{R})\,dy}
\nonumber\\
&=
C\frac{r-1}{r}\int_{\mathbb{R}^{2}}{\frac{|\nabla f(x)|^{r}}
{|y|^{2+\beta }}\chi(\frac{|y|}{R})\,dy}
-C \int_{\mathbb{R}^{2}}{\frac{|\nabla f(x)|^{r-2}\nabla f(x)\nabla f(x+y)}
{|y|^{2+\beta }}\chi(\frac{|y|}{R})\,dy}
\nonumber\\
&=
C\frac{r-1}{r}|\nabla f(x)|^{r}\int_{|y|\geq 2R}{\frac{1}
{|y|^{2+\beta }}\chi(\frac{|y|}{R})\,dy}
-C |\nabla f(x)|^{r-2}\nabla f(x)\int_{\mathbb{R}^{2}}{\frac{\nabla_{y} f(x+y)}
{|y|^{2+\beta }}\chi(\frac{|y|}{R})\,dy}
\nonumber\\
&=
C\frac{r-1}{r}|\nabla f(x)|^{r}\int_{|y|\geq 2R}{\frac{1}
{|y|^{2+\beta }}\chi(\frac{|y|}{R})\,dy}
-C |\nabla f(x)|^{r-2}\nabla f(x)\int_{\mathbb{R}^{2}}{f(x+y)\nabla_{y}\Big(\frac{\chi(\frac{|y|}{R})}
{|y|^{2+\beta }}\Big)\,dy}
\nonumber\\
&\geq
C\frac{r-1}{r}|\nabla f(x)|^{r}\int_{|y|\geq 2R}{\frac{1}
{|y|^{2+\beta }} \,dy}
-C |\nabla f(x)|^{r-1}\int_{|y|\geq R}{|f(x+y)|
\Big|\nabla_{y}\Big(\frac{\chi(\frac{|y|}{R})}
{|y|^{2+\beta }}\Big)\Big|\,dy}
\nonumber\\
&\geq
c_{1}\frac{r-1}{r}|\nabla f(x)|^{r}R^{-\beta}
-C|\nabla f(x)|^{r-1}\|f\|_{L^{p}}\left(\int_{|y|\geq R}{\frac{1}
{|y|^{\frac{(3+\beta)p}{p-1}}}\,dy}\right)^{\frac{p-1}{p}}\nonumber\\
&\geq
c_{1}\frac{r-1}{r}|\nabla f(x)|^{r}R^{-\beta}
-C_{2}|\nabla f(x)|^{r-1}\|f\|_{L^{p}}R^{-(1+\beta+\frac{2}{p})},
\end{align}
where $c_{1}$ and $C_{2}$ depend only on $\alpha$ and $p$. Taking
$$R=\left(\frac{2C_{2} r\|f\|_{L^{p}}}{c_{1}(r-1)|\nabla f(x)|}\right)^{\frac{p}{p+2}},$$
we derive from \eqref{asdwp2} that
\begin{align}
\mathcal{B}&\geq c\frac{|\nabla f(x)|^{r+\frac{\beta p}{p+2}}}{\|f\|_{L^{p}}^{\frac{\beta p}{p+2}}}.\nonumber
\end{align}
This along with \eqref{asdwp1} implies
\begin{align}
|\nabla f(x)|^{r-2}\nabla f(x)\Lambda^{\beta}\nabla f(x)\geq
\frac{1}{r}\Lambda^{\beta}(|\nabla f(x)|^{r})+c\frac{|\nabla f(x)|^{r+\frac{\beta p}{p+2}}}{\|f\|_{L^{p}}^{\frac{\beta p}{p+2}}}.\nonumber
\end{align}
As a result, we prove the desired \eqref{yyadkb02}.
\end{proof}

\begin{proof}[The proof of \eqref{tNew013}]
Here we prove that if $(u_{0}, \theta_{0})
\in H^{s}(\mathbb{R}^{2})\times H^{s}(\mathbb{R}^{2})$ with $s>2$, then the solution $(u,\theta)$ of (\ref{Bouss}) with $\alpha>0,\beta>0$ admits the following logarithmic inequality
\begin{align}\label{BBTE1}
\|\nabla u(t)\|_{L^{\infty}} \leq C(1+\|\omega(t)\|_{L^{\infty}})\ln
\Big(e+\int_{0}^{t}{\big(1+\|\omega(\tau)\|_{L^{\infty}}+\|\nabla\theta(\tau)
\|_{L^{\infty}}
\big)^{\Gamma}\,d\tau}\Big),
\end{align}
where $C=C(t,\|u_{0}\|_{H^{s}},\|\theta_{0}\|_{H^{s}})$ and $\Gamma=\Gamma(\alpha,\beta)>1$ is a positive constant depending on $\alpha$ and $\beta$.
We apply $\Lambda^{s}$ to (\ref{Bouss}) and multiply the resultant equations with $(\Lambda^{s}u,\Lambda^{s}\theta)$ to get
\begin{align}\label{BBTE2}
&\frac{d}{dt}(\|\Lambda^{s}u(t)\|_{L^{2}}^{2}+
\|\Lambda^{s}\theta(t)\|_{L^{2}}^{2})+\|\Lambda^{s+\frac{\alpha}{2}}
u\|_{L^{2}}^{2}+\|\Lambda^{s+\frac{\beta}{2}}\theta\|_{L^{2}}^{2}\nonumber\\
&=\int_{\mathbb{R}^{2}}
\Lambda^{s}u e_{2}\cdot \Lambda^{s}\theta\,dx-\int_{\mathbb{R}^{2}}
[\Lambda^{s}, u \cdot
\nabla]u \cdot\Lambda^{s}u\,dx-\int_{\mathbb{R}^{2}}
[\Lambda^{s}, u \cdot
\nabla]\theta \Lambda^{s}\theta\,dx.
\end{align}
Obviously, one gets
\begin{align}\label{BBTE3}
\int_{\mathbb{R}^{2}}
\Lambda^{s}u e_{2}\cdot \Lambda^{s}\theta\,dx &\leq C\|\Lambda^{s}u \|_{L^{2}}
\|\Lambda^{s}\theta\|_{L^{2}}\nonumber\\&\leq C\|u\|_{L^{2}}^{1-\frac{s}{s+\frac{\alpha}{2}}}\|\Lambda^{s+\frac{\alpha}{2}}
u\|_{L^{2}}^{\frac{s}{s+\frac{\alpha}{2}}}
\|\theta\|_{L^{2}}^{1-\frac{s}{s+\frac{\beta}{2}}}
\|\Lambda^{s+\frac{\beta}{2}}\theta\|_{L^{2}}^{\frac{s}{s+\frac{\beta}{2}}}
\nonumber\\&\leq
\frac{1}{4}\|\Lambda^{s+\frac{\alpha}{2}}
u\|_{L^{2}}^{2}+\frac{1}{4}\|\Lambda^{s+\frac{\beta}{2}}\theta\|_{L^{2}}^{2}
+C(\|u \|_{L^{2}}^{2}+\|\theta\|_{L^{2}}^{2})\nonumber\\&\leq C+
\frac{1}{4}\|\Lambda^{s+\frac{\alpha}{2}}
u\|_{L^{2}}^{2}+\frac{1}{4}\|\Lambda^{s+\frac{\beta}{2}}\theta\|_{L^{2}}^{2}.
\end{align}
From now on, we fix $p\geq2$ and $\sigma>0$ as follows
$$p=1+\frac{s-1}{\sigma},\qquad 0<\sigma<\min\left\{s-1,\ \frac{\alpha}{2},\ \frac{\beta}{2}\right\}.$$
By appealing to \eqref{dfexbg6} and the direct interpolation inequalities, we may have
\begin{align}\label{BBTE4}
\int_{\mathbb{R}^{2}}
[\Lambda^{s}, u \cdot
\nabla]u \cdot\Lambda^{s}u\,dx&\leq C\|[\Lambda^{s}, u \cdot
\nabla]u\|_{L^{\frac{2p}{p+1}}}\|\Lambda^{s}u\|_{L^{\frac{2p}{p-1}}}\nonumber\\
&\leq C\|\nabla u\|_{L^{p}}\|\Lambda^{s}u\|_{L^{\frac{2p}{p-1}}}^{2}
\nonumber\\
&\leq C\|\omega\|_{L^{p}}
\|\Lambda^{s}u\|_{\dot{B}_{\infty,\infty}^{-(s-1)}}^{\frac{2}{p}}
\|\Lambda^{s}u\|_{\dot{B}_{2,2}^{\sigma}}^{\frac{2(p-1)}{p}}\nonumber\\
&\leq C\|\omega\|_{L^{p}}
\|\nabla u\|_{\dot{B}_{\infty,\infty}^{0}}^{\frac{2}{p}}
\|\Lambda^{s+\sigma}u\|_{L^{2}}^{\frac{2(p-1)}{p}}
\nonumber\\
&\leq C\|\omega\|_{L^{2}}^{\frac{2}{p}}\|\omega\|_{L^{\infty}}^{\frac{p-2}{p}}
\|\omega\|_{L^{\infty}}^{\frac{2}{p}}
\|\Lambda^{s+\sigma}u\|_{L^{2}}^{\frac{2(p-1)}{p}}
\nonumber\\
&\leq C\left(\|u\|_{L^{2}}^{1-\frac{1}{s+\frac{\alpha}{2}}}
\|\Lambda^{s+\frac{\alpha}{2}}u\|_{L^{2}}
^{\frac{1}{s+\frac{\alpha}{2}}}
\right)^{\frac{2}{p}}\|\omega\|_{L^{\infty}}\nonumber\\
& \quad\times
\left(\|u\|_{L^{2}}^{1-\frac{s+\sigma}{s+\frac{\alpha}{2}}}
\|\Lambda^{s+\frac{\alpha}{2}}u\|_{L^{2}}
^{\frac{s+\sigma}{s+\frac{\alpha}{2}}}
\right)^{\frac{2(p-1)}{p}}
\nonumber\\
&\leq C
\|\Lambda^{s+\frac{\alpha}{2}}u\|_{L^{2}}
^{\frac{1}{s+\frac{\alpha}{2}}\frac{2}{p}+\frac{s+\sigma}{s+\frac{\alpha}{2}}
\frac{2(p-1)}{p}}
 \|\omega\|_{L^{\infty}}
\nonumber\\
&\leq
\frac{1}{8}\|\Lambda^{s+\frac{\alpha}{2}}
u\|_{L^{2}}^{2}+C\|\omega\|_{L^{\infty}}^{\Gamma_{1}(\alpha,\beta)},
\end{align}
where $\Gamma_{1}(\alpha,\beta)$ is given by
$$\Gamma_{1}(\alpha,\beta)=\frac{2}{2-\Lambda_{1}},\qquad\Lambda_{1} \triangleq \frac{1}{s+\frac{\alpha}{2}}\frac{2}{p}+\frac{s+\sigma}{s+\frac{\alpha}{2}}
\frac{2(p-1)}{p}\in (0,2).$$
Similar to \eqref{BBTE4}, one derives
\begin{align}\label{BBTE5}
&\int_{\mathbb{R}^{2}}
[\Lambda^{s}, u \cdot
\nabla]\theta  \Lambda^{s}\theta\,dx\nonumber\\ &\leq C\|[\Lambda^{s}, u \cdot
\nabla]\theta\|_{L^{\frac{2p}{p+1}}}\|\Lambda^{s}\theta\|_{L^{\frac{2p}{p-1}}}
\nonumber\\
&\leq C\left(\|\nabla u\|_{L^{p}}\|\Lambda^{s}\theta\|_{L^{\frac{2p}{p-1}}} +
\|\nabla \theta\|_{L^{p}}\|\Lambda^{s}u\|_{L^{\frac{2p}{p-1}}}\right)
\|\Lambda^{s}\theta\|_{L^{\frac{2p}{p-1}}}
\nonumber\\
&\leq C\|\omega\|_{L^{p}} \|\Lambda^{s}\theta\|_{L^{\frac{2p}{p-1}}}^{2}+ C
\|\nabla \theta\|_{L^{p}} \|\Lambda^{s}u\|_{L^{\frac{2p}{p-1}}}\|\Lambda^{s}\theta\|_{L^{\frac{2p}{p-1}}}
\nonumber\\
&\leq C\|\omega\|_{L^{2}}^{\frac{2}{p}}\|\omega\|_{L^{\infty}}^{\frac{p-2}{p}}
\|\Lambda^{s}\theta\|_{\dot{B}_{\infty,\infty}^{-(s-1)}}^{\frac{2}{p}}
\|\Lambda^{s}\theta\|_{\dot{B}_{2,2}^{\sigma}}^{\frac{2(p-1)}{p}}
\nonumber\\ &\quad+C
\|\nabla \theta\|_{L^{2}}^{\frac{2}{p}}\|\nabla \theta\|_{L^{\infty}}^{\frac{p-2}{p}}
\|\Lambda^{s}u\|_{\dot{B}_{\infty,\infty}^{-(s-1)}}^{\frac{1}{p}}
\|\Lambda^{s}u\|_{\dot{B}_{2,2}^{\sigma}}^{\frac{p-1}{p}}
\|\Lambda^{s}\theta\|_{\dot{B}_{\infty,\infty}^{-(s-1)}}^{\frac{1}{p}}
\|\Lambda^{s}\theta\|_{\dot{B}_{2,2}^{\sigma}}^{\frac{p-1}{p}}
\nonumber\\
&\leq C\Big(\|u\|_{L^{2}}^{1-\frac{1}{s+\frac{\alpha}{2}}}
\|\Lambda^{s+\frac{\alpha}{2}}u\|_{L^{2}}
^{\frac{1}{s+\frac{\alpha}{2}}}\Big)^{\frac{2}{p}}\|\omega\|_{L^{\infty}}^{\frac{p-2}{p}}
\|\nabla\theta\|_{L^{\infty}}^{\frac{2}{p}}
\|\Lambda^{s+\sigma}\theta\|_{L^{2}}^{\frac{2(p-1)}{p}}
\nonumber\\ &\quad+C
\Big(\|\theta\|_{L^{2}}^{1-\frac{1}{s+\frac{\beta}{2}}}
\|\Lambda^{s+\frac{\beta}{2}}\theta\|_{L^{2}}
^{\frac{1}{s+\frac{\beta}{2}}}\Big)^{\frac{2}{p}}\|\nabla \theta\|_{L^{\infty}}^{\frac{p-2}{p}}
\|\omega\|_{L^{\infty}}^{\frac{1}{p}}
\|\Lambda^{s+\sigma}u\|_{L^{2}}^{\frac{p-1}{p}}
\|\nabla\theta\|_{L^{\infty}}^{\frac{1}{p}}
\|\Lambda^{s+\sigma}\theta\|_{L^{2}}^{\frac{p-1}{p}}
\nonumber\\
&\leq C\Big(\|u\|_{L^{2}}^{1-\frac{1}{s+\frac{\alpha}{2}}}
\|\Lambda^{s+\frac{\alpha}{2}}u\|_{L^{2}}
^{\frac{1}{s+\frac{\alpha}{2}}}\Big)^{\frac{2}{p}}
(\|\omega\|_{L^{\infty}}+
\|\nabla\theta\|_{L^{\infty}})
\left(\|\theta\|_{L^{2}}^{1-\frac{s+\sigma}{s+\frac{\beta}{2}}}
\|\Lambda^{s+\frac{\beta}{2}}\theta\|_{L^{2}}
^{\frac{s+\sigma}{s+\frac{\beta}{2}}}
\right)^{\frac{2(p-1)}{p}}
\nonumber\\ &\quad+C
\Big(\|\theta\|_{L^{2}}^{1-\frac{1}{s+\frac{\beta}{2}}}
\|\Lambda^{s+\frac{\beta}{2}}\theta\|_{L^{2}}
^{\frac{1}{s+\frac{\beta}{2}}}\Big)^{\frac{2}{p}}(\|\omega\|_{L^{\infty}}+
\|\nabla\theta\|_{L^{\infty}})
\left(\|u\|_{L^{2}}^{1-\frac{s+\sigma}{s+\frac{\alpha}{2}}}
\|\Lambda^{s+\frac{\alpha}{2}}u\|_{L^{2}}
^{\frac{s+\sigma}{s+\frac{\alpha}{2}}}
\right)^{\frac{p-1}{p}}
\nonumber\\ &\quad \quad\times
\left(\|\theta\|_{L^{2}}^{1-\frac{s+\sigma}{s+\frac{\beta}{2}}}
\|\Lambda^{s+\frac{\beta}{2}}\theta\|_{L^{2}}
^{\frac{s+\sigma}{s+\frac{\beta}{2}}}
\right)^{\frac{p-1}{p}}
\nonumber\\
&\leq C
\|\Lambda^{s+\frac{\alpha}{2}}u\|_{L^{2}}
^{\frac{1}{s+\frac{\alpha}{2}}\frac{2}{p}}
\|\Lambda^{s+\frac{\beta}{2}}\theta\|_{L^{2}}
^{\frac{s+\sigma}{s+\frac{\beta}{2}}\frac{2(p-1)}{p}}(\|\omega\|_{L^{\infty}}+
\|\nabla\theta\|_{L^{\infty}})
\nonumber\\ &\quad+C
\|\Lambda^{s+\frac{\beta}{2}}\theta\|_{L^{2}}
^{\frac{1}{s+\frac{\beta}{2}}\frac{2}{p}
+\frac{s+\sigma}{s+\frac{\beta}{2}}\frac{p-1}{p}}
\|\Lambda^{s+\frac{\alpha}{2}}u\|_{L^{2}}
^{\frac{s+\sigma}{s+\frac{\alpha}{2}}\frac{p-1}{p}}(\|\omega\|_{L^{\infty}}+
\|\nabla\theta\|_{L^{\infty}})
\nonumber\\
&\leq
\frac{1}{8}\|\Lambda^{s+\frac{\alpha}{2}}
u\|_{L^{2}}^{2}+\frac{1}{4}\|\Lambda^{s+\frac{\beta}{2}}\theta\|_{L^{2}}^{2}+
C(\|\omega\|_{L^{\infty}}+
\|\nabla\theta\|_{L^{\infty}})^{\Gamma_{2}(\alpha,\beta)}\nonumber\\
&\quad+
C(\|\omega\|_{L^{\infty}}+
\|\nabla\theta\|_{L^{\infty}})^{\Gamma_{3}(\alpha,\beta)},
\end{align}
where $\Gamma_{2}(\alpha,\beta)$ and $\Gamma_{3}(\alpha,\beta)$ are given by
$$\Gamma_{2}(\alpha,\beta)=\frac{2}{2-\Lambda_{2}},\qquad\Lambda_{2} \triangleq \frac{1}{s+\frac{\alpha}{2}}\frac{2}{p}+\frac{s+\sigma}{s+\frac{\beta}{2}}
\frac{2(p-1)}{p}\in (0,2),\qquad \qquad \ $$
$$\Gamma_{3}(\alpha,\beta)=\frac{2}{2-\Lambda_{3}},\qquad\Lambda_{3} \triangleq \frac{s+\sigma}{s+\frac{\alpha}{2}}\frac{p-
1}{p}+\frac{s+\sigma}{s+\frac{\beta}{2}}
\frac{p-1}{p}+\frac{1}{s+\frac{\beta}{2}}
\frac{2}{p}\in (0,2).$$
Inserting \eqref{BBTE3}, \eqref{BBTE4} and \eqref{BBTE5} into \eqref{BBTE2} and omitting the positive dissipative terms, it yields
\begin{align}\label{BBTE6}
 \frac{d}{dt}(\|\Lambda^{s}u(t)\|_{L^{2}}^{2}+
\|\Lambda^{s}\theta(t)\|_{L^{2}}^{2}) \leq C(1+\|\omega\|_{L^{\infty}}+
\|\nabla\theta\|_{L^{\infty}})^{\Gamma(\alpha,\beta)},
\end{align}
where
$$\Gamma(\alpha,\beta)=\max\{\Gamma_{1}(\alpha,\beta),\ \Gamma_{2}(\alpha,\beta),\ \Gamma_{3}(\alpha,\beta)\}.$$
 Integrating \eqref{BBTE6} in time leads to
\begin{align}\label{BBTE7}
 \|\Lambda^{s}u(t)\|_{L^{2}}^{2}  \leq C_{0}+C\int_{0}^{t}(1+\|\omega(\tau)\|_{L^{\infty}}+
\|\nabla\theta(\tau)\|_{L^{\infty}})^{\Gamma(\alpha,\beta)}\,d\tau.
\end{align}
Now applying \eqref{BBTE7} to \eqref{sdffgbj98}, we thus obtain
\begin{align}
\|\nabla u(t)\|_{L^{\infty}}
&\leq C\Big(1+\|u(t)\|_{L^{2}}+
\|\omega(t)\|_{L^{\infty}}\Big) \ln\big(e+\|
\Lambda^{s}u(t)\|_{L^{2}}\big)\nonumber\\
&\leq C(1+\|\omega(t)\|_{L^{\infty}})\ln
\Big(e+\int_{0}^{t}{\big(1+\|\omega(\tau)\|_{L^{\infty}}+\|\nabla\theta(\tau)
\|_{L^{\infty}}
\big)^{\Gamma}\,d\tau}\Big),\nonumber
\end{align}
which immediately implies the desired estimate \eqref{BBTE1}.
\end{proof}

\vskip .2in
\textbf{Statements and Declarations:} On behalf of all authors, the corresponding author states that there is no conflict of interest.\\

\textbf{Data availability:} Since this work is of abstract theoretical nature, no data sets are generated or analyzed. One
can obtain the relevant materials from the reference list.

\vskip .3in
\section*{Acknowledgments}
Stefanov was partially supported by   NSF-DMS \# 2204788. Wu was partially supported by NSF-DMS \# 2104682 and \# 2309748. Xu was partially supported by the National Key R\&D Programe of China (Grant {No. 2020YFA0712900}) and the National Natural Science Foundation of China (Grant No. 12171040, No. 11771045 and No. 11871087).


\vskip .3in
\begin{thebibliography}{00} \frenchspacing

\bibitem{BCD}
H. Bahouri, J.-Y. Chemin, R. Danchin, \emph{Fourier Analysis and Nonlinear
Partial Differential Equations}, Grundlehren der mathematischen
Wissenschaften, 343, Springer (2011).

\bibitem{Can}
J. Cannon, E. DiBenedetto, \emph{The initial value problem for the Boussinesq equation
with data in $L^{p}$}, Lecture Notes in Mathematics, Vol. 771. Springer, Berlin, (1980), 129--144.

\bibitem{CNpr97}
D. Chae,  H. Nam, \emph{Local existence and blow-up criterion for the
boussinesq equations},  Proc. Roy. Soc. Edinburgh Sect. A, \textbf{127} (1997), 935--946.

\bibitem{C1}
D. Chae, \emph{Global regularity for the 2D Boussinesq equations with
partial viscosity terms}, Adv. Math. \textbf{203} (2006), 497--513.

\bibitem{Chenhou21}
J. Chen, T.Y. Hou, \emph{Finite time blowup of 2D Boussinesq and 3D Euler equations with $C^{1,\alpha}$ velocity and boundary}, Commun. Math. Phys. \textbf{383} (2021) 1559--1667.


\bibitem{Chenhou}
J. Chen, T. Y. Hou, \emph{On stability and instability of $C^{1,\alpha}$ singular solutions to the 3D Euler and 2D Boussinesq equations}, Comm. Math. Phys., \textbf{383} (2024),no. 5, Paper No. 112, 53 pp.

\bibitem{Chenhou25}
J. Chen, T. Hou, \emph{Stable nearly self-similar blowup of the 2D Boussinesq and 3D Euler equations with smooth data II: rigorous numerics}, Multiscale Model. Simul. \textbf{23} (2025), 25--130.

\bibitem{Christw}
M. Christ, M. Weinstein, \emph{Dispersion of small amplitude solutions of the generalized Korteweg-de Vries equation}, J. Funct. Anal., \textbf{100},  (1991), 87--109.

\bibitem{CV}
P. Constantin, V. Vicol, \emph{Nonlinear maximum principles for dissipative linear nonlocal
operators and applications}, Geom. Funct. Anal. \textbf{22} (2012), 1289--1321.

\bibitem{Clmzad}
D. C\'ordoba, A. Lain-Sanclemente, L. Martinez-Zoroa, \emph{Finite-time singularity via multi-layer degenerate pendula for the 2D Boussinesq
equation with uniform $C^{1,\sqrt{\frac{4}{3}}-1-\epsilon}\cap L^2$ force}, Adv. Math. \textbf{480} (2025), part A, Paper No. 110480, 205 pp.

\bibitem{CC}
 A. C\'ordoba, D. C\'ordoba, \emph{A maximum princple applied to quasi-geostroohhic equations}, Comm. Math. Phys. {\bf 249} (2004),  511-528.

\bibitem{Danchinp09}
R. Danchin, M. Paicu, \emph{Global well-posedness issues for the inviscid Boussinesq system with Yudovich's type data}, Commun. Math. Phys. \textbf{290} (2009), 1--14.

\bibitem{Elgindije}
T. Elgindi, I. Jeong, \emph{Finite-time singularity formation for strong solutions to the Boussinesq system}, Ann. PDE \textbf{6} (2020), no. 1, Paper No. 5, 50 pp.

\bibitem{ElgindiPA}
T. Elgindi, F. Pasqualotto, \emph{From instability to singularity formation in incompressible fluids}, arXiv:2310.19780v1 [math.AP].


\bibitem{HK1}
T. Hmidi, S. Keraani, \emph{On the global well-posedness of the Boussinesq system with zero viscosity}, Indiana Univ. Math. J. \textbf{58} (2009), 1591--1618.


\bibitem{HK3}
T. Hmidi, S. Keraani, F. Rousset, \emph{Global well-posedness for a
Boussinesq-Navier-Stokes system with critical dissipation}, J.
Differential Equations \textbf{249} (2010), 2147--2174.

\bibitem{HK4}
T. Hmidi, S. Keraani, F. Rousset, \emph{Global well-posedness for
Euler-Boussinesq system with critical dissipation}, Comm. Partial
Differential Equations \textbf{36} (2011), 420--445.


\bibitem{Hzer10}
T. Hmidi, M. Zerguine, \emph{On the global well-posedness of the Euler-Boussinesq system with fractional dissipation}, Physica D \textbf{239} (2010), 1387--1401.


\bibitem{HL}
T. Y. Hou, C. Li, \emph{Global well-posedness of the viscous Boussinesq
equations}, Discrete Contin. Dyn. Syst. \textbf{12} (2005),
1--12.

\bibitem{JMWZ}
Q. Jiu, C. Miao, J. Wu, Z. Zhang, \emph{The 2D incompressible Boussinesq
equations with general critical dissipation},  SIAM J. Math. Anal. \textbf{46} (2014), 3426--3454.

\bibitem{MB}
 A. Majda, A. Bertozzi, \emph{Vorticity and Incompressible
Flow}, Cambridge University Press, Cambridge, 2001.

\bibitem{MX}
C. Miao, L. Xue, \emph{On the global well-posedness of a class of
Boussinesq-Navier-Stokes systems}, NoDEA Nonlinear Differential
Equations Appl. \textbf{18} (2011), 707--735.

\bibitem{Pe1987}
J. Pedlosky, \emph{Geophysical Fluid Dynamics}, Springer, New York, 1987.

\bibitem{SW}
A. Stefanov, J. Wu, \emph{A global regularity result for the 2D Boussinesq equations with critical dissipation}, J. Anal. Math. \textbf{137} (2019), 269--290.

\bibitem{SWXY}
A. Stefanov, J. Wu, X. Xu, Z. Ye, \emph{Global regularity results of the 2D fractional Boussinesq equations},  Math. Ann. 391 (2025), 5965--6012.

\bibitem{wuxuxueye}
J. Wu, X. Xu, L. Xue, Z. Ye, \emph{Regularity results for the 2D Boussinesq equations with critical or supercritical dissipation}, Commun. Math. Sci. \textbf{14} (2016), 1963--1997.

\bibitem{YJW}
W. Yang, Q. Jiu, J. Wu, \emph{Global well-posedness for a class of 2D Boussinesq systems with fractional dissipation}, J. Differential Equations \textbf{257} (2014), 4188--4213.

\bibitem{Yemmas}
Z. Ye, \emph{Global smooth solution to the 2D Boussinesq equations with fractional dissipation}, Math. Methods Appl. Sci. \textbf{40} (2017), 4595--4612.


\bibitem{YeNA2020}
Z. Ye, \emph{An alternative approach to global regularity for the 2D Euler-Boussinesq equations with critical dissipation}, Nonlinear Anal. \textbf{190} (2020), 111591, 5 pp.

\bibitem{YX201502}
Z. Ye, X. Xu, \emph{Global well-posedness of the 2D Boussinesq equations
with fractional Laplacian dissipation}, J. Differential Equations \textbf{260} (2016), 6716--6744.

\bibitem{zlswyz}
D. Zhou, Z. Li, H. Shang, J. Wu, B. Yuan, J. Zhao, \emph{Global well-posedness for the 2D fractional Boussinesq equations in the subcritical case}, Pacific J. Math. \textbf{298} (2019), 233--255.

\end{thebibliography}
\end{document}